\documentclass[11pt,reqno]{amsart}
\usepackage{cmap}
\usepackage[T1]{fontenc}
\usepackage[utf8]{inputenc}
\usepackage{lmodern}

\usepackage{newunicodechar}
\newunicodechar{，}{,}
\newunicodechar{。}{.}
\usepackage{
  amsmath,
  amssymb,
  amsthm,
  amsfonts,
  mathtools,
  mathrsfs
}

\usepackage[margin=1.05in]{geometry}

\usepackage[
  colorlinks=true,
  linkcolor=blue,
  citecolor=blue,
  urlcolor=blue
]{hyperref}

\usepackage{aliascnt}
\usepackage[nameinlink,capitalize]{cleveref}

\numberwithin{equation}{section}

\newtheorem{theorem}{Theorem}[section]

\newaliascnt{proposition}{theorem}
\newtheorem{proposition}[proposition]{Proposition}
\aliascntresetthe{proposition}
\newaliascnt{lemma}{theorem}
\newtheorem{lemma}[lemma]{Lemma}
\aliascntresetthe{lemma}

\newaliascnt{corollary}{theorem}
\newtheorem{corollary}[corollary]{Corollary}
\aliascntresetthe{corollary}

\theoremstyle{definition}
\newaliascnt{definition}{theorem}
\newtheorem{definition}[definition]{Definition}
\aliascntresetthe{definition}

\theoremstyle{remark}
\newaliascnt{remark}{theorem}
\newtheorem{remark}[remark]{Remark}
\aliascntresetthe{remark}

\usepackage{fancyhdr}
\newcommand{\R}{\mathbb{R}}
\newcommand{\Z}{\mathbb{Z}}
\newcommand{\tdL}{\widetilde{L}}
\newcommand{\Id}{\operatorname{Id}}
\newcommand{\Div}{\operatorname{div}}

\newcommand{\cA}{\mathcal{A}}

\newcommand{\dDelta}{\dot{\Delta}}
\newcommand{\dS}{\dot{S}}
\newcommand{\dB}{\dot{B}}

\newcommand{\norm}[1]{\left\lVert #1\right\rVert}

\newcommand{\dd}{\,\mathrm{d}}

\title{\parbox{0.9\columnwidth}{\centering
 On the stability of the solutions to Compressible Navier-Stokes Equations
}}

\author[Q. Dong]{Qiaojie Dong}

\address{Qiaojie Dong, School of Information Management and Mathematics, Jiangxi University of Finance and
Economics, Nanchang, 330032, China}
\email{2202420242@stu.jxufe.edu.cn}
\author[M. Li]{Min Li}
\address{Min Li, School of Information Management and Mathematics, Jiangxi University of Finance and
Economics, Nanchang, 330032, China}
\email{nculimin@gmail.com}
\author[T. Li]{Yatao Li}
\address{Yatao Li, School of Information Management and Mathematics, Jiangxi University of Finance and
Economics, Nanchang, 330032, China}
\email{liyatao@jxufe.edu.cn}

\author[M. Yang]{Minghua Yang$^*$}\thanks{*Corresponding author}
\address{Minghua Yang, School of Information Management and Mathematics, Jiangxi University of Finance and
Economics, Nanchang, 330032, China}
\email{minghuayang@jxufe.edu.cn}

\date{}

\hypersetup{
  pdftitle={Continuous Dependence for the Compressible Navier--Stokes
  Equations with Large Density Perturbations},
  pdfauthor={XX}
}

\begin{document}
\raggedbottom
\begin{abstract}
This paper investigates the continuous dependence (stability) of solutions to the barotropic compressible Navier--Stokes equations in critical Besov spaces. Using the Lagrangian approach developed in \cite{Danchin2014}, it was shown that, for \(1<p<2d\), the flow map
$(a_0,u_0)\mapsto (\bar a,\bar u)=(a\circ X,u\circ X)$ is Lipschitz continuous from $\dot B_{p,1}^{d/p}\times \dot B_{p,1}^{d/p-1}$ into
$\mathcal{C}([0,T];\dot B_{p,1}^{d/p})\times E_p(T)$. However, this result does not directly imply the corresponding continuous dependence in the original Eulerian coordinates because of the low regularity of the critical initial data. Previously, only for $p<d$, continuous dependence was known only in certain lower-regularity spaces with a loss of one derivative relative to the natural solution spaces arising essentially as a by-product of the uniqueness argument. We close this gap and prove that for $1<p<2d$ the flow map $(a_0,u_0)\mapsto(a,u)$ is continuous (not Lipschitz continuous) from $\dot B_{p,1}^{d/p}\times\dot B_{p,1}^{d/p-1}$ to $\mathcal{C}([0,T];\dot B_{p,1}^{d/p})\times E_p(T)$ in Eulerian coordinates without any loss of regularity, which together with the known existence and uniqueness theory \cite{Danchin2014} completes Hadamard well-posedness in the critical spaces. 
\end{abstract}

\maketitle
\tableofcontents

\section{Introduction}

The barotropic compressible Navier--Stokes system provides one of the
fundamental mathematical models for the motion of viscous compressible
fluids. In the whole space $\mathbb{R}^{d}$, $d\geq 2$, the system reads
\begin{equation}\label{eq:1.1}
\left\{
\begin{aligned}
&\partial_t\rho+\operatorname{div}(\rho u)=0,
\\
&\partial_t(\rho u)+\operatorname{div}(\rho u\otimes u)
-\mathcal{A}u+\nabla P(\rho)=0,
\\
&(\rho,u)|_{t=0}=(\rho_0,u_0),
\end{aligned}
\right.
\end{equation}
where $\rho=\rho(t,x)$ denotes the density, $u=u(t,x)$ is the velocity
field, and $P=P(\rho)$ is the pressure law. The Lam\'e operator
$\mathcal{A}$ is given by
$
\mathcal{A}u
=
\mu\Delta u+(\lambda+\mu)\nabla\operatorname{div}u,
$
with viscosity coefficients satisfying
$
\mu>0,
2\mu+\lambda>0$. We normalize the reference density to one and set
$
a:=\rho-1
$.
Assuming that the density stays away from vacuum, namely
$
1+a(t,x)>0$,
we introduce
\[
b(a):=\frac{1}{1+a},
\qquad
G'(a):=\frac{P'(1+a)}{1+a}.
\]
Then system \eqref{eq:1.1} may be rewritten as
\begin{equation}\label{eq:1.2}
\left\{
\begin{aligned}
&\partial_t a+u\cdot\nabla a
=-(1+a)\operatorname{div}u,
\\
&\partial_tu-b(a)\mathcal{A}u
=-u\cdot\nabla u-\nabla G(a),
\\
&(a,u)|_{t=0}=(a_0,u_0).
\end{aligned}
\right.
\end{equation}

The initial value problem \eqref{eq:1.2} is said to be locally well-posed in a metric space \(X\) if, for any prescribed initial datum \((a_{01},u_{01})\in X\), there exists a neighbourhood \(U\) of this point and a time \(T=T(U)>0\) such that, for every \((a_0,u_0)\in U\), the ensuing conditions are fulfilled:
\begin{enumerate}
\item[(a)] \textbf{Existence} -- the system possesses a distributional solution \((a,u)\) that belongs to some class \(E\subset C([0,T):X)\);
\item[(b)] \textbf{Uniqueness} -- the solution is the sole member of \(E\) satisfying the equation;
\item[(c)] \textbf{Continuous dependence (Stability )} -- the map \((a_0,u_0)\mapsto (a,u)=:S_T(a_0,u_0)\) is continuous as a mapping from \(U\) into \(C([0,T):X)\).
\end{enumerate}
When the same properties hold for every \(T>0\), we call the problem globally well-posed in \(X\).

The compressible Navier--Stokes equations possess the formal scaling
\[
a(t,x)\mapsto a(\ell^2t,\ell x),
\qquad
u(t,x)\mapsto \ell u(\ell^2t,\ell x),
\qquad
\ell>0,
\]
up to the pressure term, which is of lower order with respect to this
parabolic scaling. This naturally leads to the critical initial space
\begin{equation}\label{eq:Xp}
X_p
:=
\dot B^{\frac{d}{p}}_{p,1}
\times
\dot B^{-1+\frac{d}{p}}_{p,1}.
\end{equation}
For $T>0$, we use the solution space
\begin{equation}\label{eq:Ep}
E_p(T)
:=
\mathcal{C}\bigl([0,T];\dot B^{\frac{d}{p}}_{p,1}\bigr)
\times
\left(
\mathcal{C}\bigl([0,T];\dot B^{-1+\frac{d}{p}}_{p,1}\bigr)
\cap
L^1\bigl(0,T;\dot B^{1+\frac{d}{p}}_{p,1}\bigr)
\right).
\end{equation}
We equip \(E_p(T)\) with the norm
\[
\begin{aligned}
\norm{(a,u)}_{E_p(T)}
:=
\norm a_{\tdL^\infty_T(\dB^{\frac dp}_{p,1})}
+
\norm u_{\tdL^\infty_T(\dB^{-1+\frac dp}_{p,1})}+
\norm u_{L^1_T(\dB^{1+\frac dp}_{p,1})}.
\end{aligned}
\]

The well-posedness for compressible Navier--Stokes equations has
been extensively developed over the past several decades. Danchin \cite{D2001,D2005} and Chen-Miao-Zhang \cite{CMZ2010R}established local existence for $1\leq p<2d$ and uniqueness for $1\leq p\leq d$. The uniqueness for $d<p<2d$ was not clear until  Danchin \cite{Danchin2014} by a Lagrangian approach (see  \cref{theo 2.2}). On the other hand, the range $p<2d$ is optimal for the well-posedness in $X_p$. Chen-Miao-Zhang \cite{CMZ2015} proved the ill-posedness in the sense that the continuity of the solution map $S_T$ fails at the origin in $X_p$ when $p>2d$. Iwabuchi-Ogawa \cite{IO2022} proved the same result for $p=2d$.  Only some weak continuous dependence results (in the $X_p^{-1}$ space ) were known as a by-product of the uniqueness argument. For example, in \cite{D2001, D2005}, it was proved that for $1\leq p<d$. For the case $p \geq d$, it is not known whether the above difference estimate holds true since the regularity of $u$ is too negative.  

In light of the preceding results, whether the continuity of the solution map \( S_T \) holds in \( X_p \) for \( 1 \leq p < 2d \); an affirmative answer would complete the Hadamard well-posedness theory in this space. The first resolution was provided in the recent paper \cite{GuoYangZhang2026} by the fourth author together with Guo and Zhang by the frequency-envelope method (see \cite{Tao04}) with the transport-parabolic structure and the Lagrangian approach \cite{Danchin2014}  under a smallness assumption on the density perturbation. In a recent work \cite{GSY}, the fourth author together with Guo and Song proved global well-posedness in the optimal Besov space under some additional assumption on the low frequency of the density and momentum.

 It is natural to ask whether the continuity of the solution map \( S_T \) holds in \( X_p \) for \( p < 2d \) without smallness assumption on the density perturbation. To the best of our knowledge, this issue had not been addressed in earlier works.
\begin{theorem}\label{theo 2.2}
\textup{\cite[Theorem~1.1 and 1.2]{Danchin2014}.}
  Let \( d \geq 2 \), \( 1 < p < 2d \), and let
\[
(a_0, u_0) \in X_p, \quad \inf_x\rho_0(x):=\inf_x\bigl(1+a_0(x)\bigr)>0.
\]
Then there exists \( T > 0 \) such that the compressible Navier--Stokes system admits a unique solution
\[
(a, u) \in E_p(T).
\]
Moreover, if \( (\bar{a}, \bar{u}) \) denotes the corresponding solution in Lagrangian coordinates, then
\[
(\bar{a}, \bar{u}) \in \mathcal{C}([0, T]; \dot{B}_{p,1}^{\frac{d}{p}}) \times \left( \mathcal{C}([0, T]; \dot{B}_{p,1}^{\frac{d}{p}-1}) \cap L^1(0, T; \dot{B}_{p,1}^{\frac{d}{p}+1}) \right),
\]
and the Lagrangian solution map is locally Lipschitz from \( X_p \) into this space.

\noindent Equivalently, in Eulerian coordinates, the solution satisfies
\[
a \in \mathcal{C}([0, T]; \dot{B}_{p,1}^{\frac{d}{p}}), \quad u \in \mathcal{C}([0, T]; \dot{B}_{p,1}^{\frac{d}{p}-1}) \cap L^1(0, T; \dot{B}_{p,1}^{\frac{d}{p}+1}),
\]
with \( 1 + a \) bounded away from zero on \( [0, T] \).  
\end{theorem}

We note that Theorem~\ref{theo 2.2} establishes the existence and uniqueness of solutions, but does not address the continuity of the solution map \(S_T\) in \(X_p\) for \(p<2d\) in Eulerian coordinates. The main difficulty is that, although one has mutual control between \((a,u)\) and \((\bar a,\bar u)\), the required difference estimates in \(X_p\) cannot be recovered directly from those for the transformed variables because of the low regularity of the initial data in the critical spaces. Motivated by the works \cite{GuoYangZhang2026,Danchin2014,CMZ2010R}, the purpose of this paper is to establish the continuous dependence of the solutions constructed in Theorem~\ref{theo 2.2} on the initial data, or equivalently, the continuity of the solution map \(S_T\) in \(X_p\) for \(p<2d\) with respect to the Eulerian topology, without imposing any smallness assumption on the density perturbation. Our main result is stated as follows.

\begin{theorem}\label{thm:main}
Let $d\geq2$ and $1<p<2d$. If
\begin{equation}\label{eq:positive-density}
(a_0,u_0)\in X_p,
\,
\kappa_0:=\inf_x\rho_0(x)=\inf_x\bigl(1+a_0(x)\bigr)>0,
\end{equation}
then there exist a neighborhood $\mathcal{U}$ of $(a_0,u_0)$ in
$X_p$ and a time $T=T(\mathcal{U})>0$ such that, for every
$(\widetilde a_0,\widetilde u_0)\in\mathcal{U}$, system
\eqref{eq:1.2} admits a unique solution
\[
(\widetilde a,\widetilde u)
=
S_T(\widetilde a_0,\widetilde u_0)
\in E_p(T).
\]
Moreover, the solution map
$
S_T:\mathcal{U}\rightarrow E_p(T)
$
is continuous.
\end{theorem}

\begin{remark}
In contrast to the recent work \cite{GuoYangZhang2026}, the main novelty of the present paper lies in establishing the continuity of the Eulerian solution map under the sole non-vacuum condition \eqref{eq:positive-density}, without any smallness assumption on the density perturbation. This demonstrates that large density variations do not preclude Hadamard continuous dependence. Theorem~\ref{theo 2.2} tell us that the Lagrangian approach shows that, for \( 1 < p < 2d \), the flow map $(a_0, u_0) \mapsto (\bar{a}, \bar{u}) = (a \circ X, u \circ X)$ is Lipschitz continuous. However, this result does not directly imply the corresponding continuous dependence in the original Eulerian coordinates. In the present work, we prove that the Eulerian flow map \((a_0, u_0) \mapsto (a, u)\) is continuous (though not Lipschitz) from the same space into $\mathcal{C}([0, T]; \dot{B}_{p,1}^{d/p}) \times E_p(T)$. This provides a link between the Eulerian and Lagrangian approaches.
\end{remark}

In the rest of the introduction, we would like to discuss the main difficult in the proof of \cref{thm:main}. For small density perturbations, relying on  frequency-envelope methods adapting to the transport-parabolic setting \cite{Tao04} and the Lagrangian techniques from \cite{Danchin2014}, the continuity of the Eulerian
solution map was first established in \cite{GuoYangZhang2026}. More precisely, in \cite{GuoYangZhang2026}, 
\begin{align}
b(a)\mathcal{A}u
=
\mathcal{A}u+\bigl(b(a)-1\bigr)\mathcal{A}u,
\label{bu}
\end{align}
and the second term may be absorbed only when $b(a)-1$ is sufficiently
small. But for large density perturbations, 
$
b(a)\geq c_0>0,
$
the quantity
$
\|b(a)-1\|_{\dot B^{\frac{d}{p}}_{p,1}}
$
may be arbitrarily large. To overcome the aforementioned difficulties,  we establish a new \emph{a priori} estimate (see \cref{prop:weighted-lame}) for the variable-coefficient parabolic equation
    $
    \partial_t v - b(a)\mathcal A v = f
  $
    directly in the frequency-envelope weighted Besov spaces. The key idea is to split the coefficient \(b\) into a low-frequency part \(b_m\) (which carries the main information) and a high-frequency part \(b^h\) (a small tail). By exploiting the smallness of \(b^h\) (controlled by the time weight \(\omega_T\)) together with the good properties of \(b_m\) (bounded away from zero), the estimate successfully handles the genuinely variable-coefficient nature of the operator without requiring smallness of \(b-1\). This serves as the cornerstone for treating large density variations.

    We used two distinct weighting mechanisms. This double-layer weighting strategy enables the unified treatment of low regularity and high-frequency oscillations.
    The \emph{time weight} \(\omega_T\) borrowed from \cite{CMZ2010R} is used to gain smallness of high-frequency parts of the coefficients (e.g., \(b^h\)) over a short time interval (see \cref{cor:common-m}), which allows us to temporarily sacrifice the time length to obtain uniform high-frequency control, which is essential for establishing uniform estimates.
  The \emph{frequency envelope} \(\Omega\) is constructed  to handle high-frequency tails by building a weight \(\Omega\) that grows with frequency, one can absorb the energy in the high-frequency region (see \cref{prop:envelope-persistence} and \cref{cor:uniform-tails}), which borrowed ideas from  \cite{GuoYangZhang2026}.

In the low-frequency analysis, the commutator estimates established in
\cref{lem:cutoff-transport-commutator} are \emph{optimal} in the
critical Besov setting (i.e., no derivative loss), and their implicit
constants are uniform with respect to the cutoff indices \(m\) and
\(N\). Here \(m=m(T)\) is the common cutoff in the coefficient
decomposition \(b=b_m+b^h\), chosen so that \(1\le c_*2^{2m}T<4\), whereas \(N\) is the cutoff index used to separate the low- and
high-frequency components of the solutions in the continuity argument. For each fixed \(N\), the low-frequency
part is ultimately controlled by the difference of the initial data,
while the high-frequency remainder is bounded by
\(C\varepsilon_N\), where \(C\) is independent of \(N\) and
\(\varepsilon_N\to0\). Although the low-frequency constant \(C_N\) may depend on \(N\), this dependence causes no difficulty. In the continuity proof, the sequence limit is first taken for fixed \(N\), and \(N\) is then allowed to tend to infinity.

\section{Auxiliary results}\label{sec:prelim}
Throughout this paper, $C$ and $c$ denote generic positive constants, which may vary from line to line. Their dependence on the parameters will be clear from the context. We write $A \lesssim B$ to mean that $A \leq CB$ for some positive constant $C$. When necessary, the dependence of $C$ will be indicated explicitly by subscripts.

\begin{lemma}
\cite[Proposition A.1]{Danchin2014}\label{pro 2.2}
   Let \( X \) be a globally bi-Lipschitz diffeomorphism of \( \mathbb{R}^d \) and \( (s, p, q) \) with \( 1 \leq p < \infty \) and \(-\frac{d}{p'}< s < \frac{d}{p}\) (or just \(-\frac{d}{p'} < s \leq \frac{d}{p}\) if \( q = 1 \) and just \(-\frac{d}{p'} \leq s < \frac{d}{p}\) if \( q = \infty \)). Then \( a \mapsto a \circ X \) is a self-map over \( \dot{B}_{p,q}^s \) in the following cases:

\begin{enumerate}
    \item \( s \in (0, 1) \),
    \item \( s \in (-1, 0] \) and \( J_{X^{-1}} := \det D(X^{-1}) \) is in the multiplier space \( \mathcal{M}(\dot{B}_{p',q'}^{-s}) \) defined in \eqref{mutiple};
    \item \( s \geq 1 \) and \( (DX - \mathrm{Id}) \in \dot{B}_{p,1}^{\frac{d}{p}} \).
    \end{enumerate}
\end{lemma}
\noindent For the inverse-flow bootstrap below, we shall use the following
quantitative refinement of \cref{pro 2.2}.
\begin{lemma}
\label{lem:quantitative-critical-composition}
Let \(d\ge2\), \(1<p<2d\), and  \(L\ge1\). Let
\(\Phi:\R^d\to\R^d\) be a smooth globally bi-Lipschitz
diffeomorphism satisfying
\[
\max\left\{
\|D\Phi\|_{L^\infty},
\|D(\Phi^{-1})\|_{L^\infty}
\right\}
\le
L,
\qquad
D\Phi-\Id\in\dot B^{\frac dp}_{p,1}.
\]
Then there exists an integer \(m_p\ge0\), depending only on
\(d,p\), such that
\begin{equation}\label{eq:quantitative-critical-composition}
\|h\circ\Phi\|_{\dot B^{\frac dp}_{p,1}}
\lesssim
\left(
1+\|D\Phi-\Id\|_{\dot B^{\frac dp}_{p,1}}
\right)^{m_p}
\|h\|_{\dot B^{\frac dp}_{p,1}}
\end{equation}
for every \(h\in\dot B^{\frac dp}_{p,1}\), where the implicit constant depends only on \(d,p,L\).
\end{lemma}
\begin{proof}
  The result follows readily from the proof of \cref{pro 2.2}
by keeping track of the dependence on
\(\|D\Phi-\Id\|_{\dot B^{d/p}_{p,1}}\).  
\end{proof}
\begin{proposition}
\label{lem:eulerian-lagrangian-high-norm}
Assume that \(d\ge2\), \(1<p<2d\), and
\[
\bar v\in C([0,T];\dB^{-1+\frac dp}_{p,1})
\cap L^1(0,T;\dB^{1+\frac dp}_{p,1}).
\]
Let
\[
X(t,y):=y+\int_0^t\bar v(\tau,y)\dd\tau,
\,\,
A(t,y):=(D_yX(t,y))^{-1},
\,\,
Y(t,\cdot):=X(t,\cdot)^{-1}.
\]
There exists a constant \(c>0\), depending only on
\(d,p\), such that, if
\[
\norm{\bar v}_{L^1_T(\dB^{1+\frac dp}_{p,1})}\le c,
\]
then \(X(t,\cdot)\) is a globally bi-Lipschitz diffeomorphism for all
\(0\le t\le T\). Moreover, if \(v:=\bar v\circ Y\), then, for every
\(0\le t\le T\),
\begin{equation}\label{eq:eulerian-lagrangian-comparison1}
\norm v_{L^1_t(\dB^{1+\frac dp}_{p,1})}
\lesssim
\norm{\bar v}_{L^1_t(\dB^{1+\frac dp}_{p,1})}
+
\norm{\bar v}_{L^1_t(\dB^{1+\frac dp}_{p,1})}^2,
\end{equation} where the implicit constant depends only on \(d,p\).
\end{proposition}

\begin{proof}
We shall prove \cref{lem:eulerian-lagrangian-high-norm} using the following four steps.

\textbf{Step 1. Estimates for the deformation and the inverse map.}
Since  \(\dot B^{\frac{d}{p}}_{p,1}\hookrightarrow L^\infty\) and \(
D_yX(t)-\Id=\int_0^tD_y\bar v(\tau)\dd\tau
\), we obtain
\begin{equation}\label{eq:flow-deformation-comparison1}
\sup_{0\le\tau\le t}
\left(
\norm{DX(\tau)-\Id}_{L^\infty}
+
\norm{DX(\tau)-\Id}_{\dB^{\frac dp}_{p,1}}
\right)
\lesssim \norm{\bar v}_{L^1_t(\dB^{1+\frac dp}_{p,1})}.
\end{equation}
Taking \(c\) sufficiently small, we get
\(\|DX-\Id\|_{L^\infty_{t,x}}\le 1/2\). Hence \(X(t,\cdot)\) is a
globally bi-Lipschitz diffeomorphism. Moreover, by the Neumann series and
\eqref{eq:product-critical}, we have
\begin{equation}\label{3}
 \sup_{0\le\tau\le t}
\left(
\|A(\tau)\|_{L^\infty}
+
\|A(\tau)-\Id\|_{\dot B^{\frac{d}{p}}_{p,1}}
\right)
\le C.   
\end{equation}
For \(1\le m\le d\),$
\partial_{y_m}A
=
-A(\partial_{y_m}D_yX)A.
$
Writing \(A=\Id+(A-\Id)\), \eqref{3},
\eqref{eq:product-critical-low}, and
\[
D_y^2X(t)
=
\int_0^tD_y^2\bar v(\tau)\dd\tau
\]
yield
\begin{equation}\label{eq:inverse-flow-derivative1}
\sup_{0\le\tau\le t}
\norm{D_yA(\tau)}_{\dB^{-1+\frac dp}_{p,1}}
\lesssim \norm{\bar v}_{L^1_t(\dB^{1+\frac dp}_{p,1})}.
\end{equation}
We next need to prove
\begin{equation}\label{eq:inverse-map-critical1}
\sup_{0\le\tau\le t}
\norm{DY(\tau)-\Id}_{\dB^{\frac dp}_{p,1}}
\lesssim \norm{\bar v}_{L^1_t(\dB^{1+\frac dp}_{p,1})}.
\end{equation} 
Fix \(\tau\in[0,t]\) and write
\[
F:=X(\tau)-\Id
=
\int_0^\tau\bar v(\rho)\dd\rho.
\]
Then
\[
F\in
\dB^{-1+\frac dp}_{p,1}
\cap
\dB^{1+\frac dp}_{p,1},
\]
and interpolation yields \(F\in\dB^{\frac dp}_{p,1}\).

\noindent For \(N\ge1\), set
\[
J_N:=\dot S_N-\dot S_{-N},
\qquad
F^{(N)}:=J_NF,
\qquad
X_\theta^{(N)}:=\Id+\theta F^{(N)},
\quad 0\le\theta\le1.
\]
Then
\[
F^{(N)}\longrightarrow F
\quad\text{in}\quad
\dot B^{\frac dp}_{p,1}
\cap
\dot B^{1+\frac dp}_{p,1}
\qquad\text{as }N\to\infty,
\]
and
\[
\sup_{N\ge1}
\|DF^{(N)}\|_{\dot B^{\frac dp}_{p,1}}
\lesssim
\|\bar v\|_{L^1_t(\dot B^{1+\frac dp}_{p,1})}.
\]
By the critical embedding   
$
\dot B_{p,1}^{\frac{d}{p}} \hookrightarrow L^\infty,
$
after reducing \( c \) if necessary, we may assume that  
\[
 \sup_{N \geq 1,\theta \in [0,1]} \|\theta D F^{(N)}\|_{L^\infty} \leq \frac{1}{2}.
\]  
Hence
\[
\frac{1}{2} |x - y| \leq |X_\theta^{(N)}(x) - X_\theta^{(N)}(y)| \leq \frac{3}{2} |x - y|,
\]  
and therefore \(X_\theta^{(N)}\) is a smooth globally
bi-Lipschitz diffeomorphism, uniformly in \(N\) and \(\theta\).
We may thus define
\[
Y_\theta^{(N)}:=(X_\theta^{(N)})^{-1}.
\]
For each fixed \(N\), the finite-frequency regularity of
\(F^{(N)}\), together with the uniform bi-Lipschitz bounds,
gives
\[
\theta\mapsto DY_\theta^{(N)}-\Id
\in
C\bigl([0,1];\dot B^{\frac dp}_{p,1}\bigr).
\]
Differentiating
\(X_\theta^{(N)}\circ Y_\theta^{(N)}=\Id\), we obtain
\[
DY_\theta^{(N)}-\Id
=
\left[
(\Id+\theta DF^{(N)})^{-1}-\Id
\right]\circ Y_\theta^{(N)}.
\]
Set
\[
Z_\theta^{(N)}
:=
\|DY_\theta^{(N)}-\Id\|_{\dot B^{\frac dp}_{p,1}}.
\]
The Neumann series and \eqref{eq:product-critical} imply
\[
\|(\Id+\theta DF^{(N)})^{-1}-\Id\|
_{\dot B^{\frac dp}_{p,1}}
\lesssim
\theta
\|DF^{(N)}\|_{\dot B^{\frac dp}_{p,1}}.
\]
Hence, by
\cref{lem:quantitative-critical-composition},
\begin{equation}\label{eq:inverse-bootstrap1}
Z_\theta^{(N)}
\le
C\theta
\|DF^{(N)}\|_{\dot B^{\frac dp}_{p,1}}
(1+Z_\theta^{(N)})^{m_p}.
\end{equation}
Decreasing \(c\) once more, we may assume that
\[
C2^{m_p}
\sup_N
\norm{DF^{(N)}}_{\dB^{\frac dp}_{p,1}}
<1.
\]
Since \(Z_0^{(N)}=0\) and
\(\theta\mapsto Z_\theta^{(N)}\) is continuous on \([0,1]\),
a standard continuity argument applied to \eqref{eq:inverse-bootstrap1} gives
\[
\sup_{0\le\theta\le1}Z_\theta^{(N)}<1,\quad \sup_{0\le\theta\le1}Z_\theta^{(N)}
\lesssim
\norm{DF^{(N)}}_{\dB^{\frac dp}_{p,1}},
\]
uniformly in \(N\).

\noindent Finally,
\[
\|X_1^{(N)}-X(\tau)\|_{L^\infty}
+
\|Y_1^{(N)}-Y(\tau)\|_{L^\infty}
\longrightarrow0
\qquad\text{as }N\to\infty.
\]
Consequently,
\[
DY_1^{(N)}-\Id
\longrightarrow
DY(\tau)-\Id
\quad\text{in }\mathcal S'
\qquad\text{as }N\to\infty.
\]
By the Fatou property of homogeneous Besov spaces,
\[
\begin{aligned}
\|DY(\tau)-\Id\|_{\dot B^{\frac dp}_{p,1}}
&\le
\liminf_{N\to\infty}
\|DY_1^{(N)}-\Id\|_{\dot B^{\frac dp}_{p,1}}
\\
&\lesssim
\|DF\|_{\dot B^{\frac dp}_{p,1}}
=
\|DX(\tau)-\Id\|_{\dot B^{\frac dp}_{p,1}}.
\end{aligned}
\]
Taking the supremum over \(0\le\tau\le t\) and using
\eqref{eq:flow-deformation-comparison1} proves
\eqref{eq:inverse-map-critical1}.

\textbf{Step 2. Uniform composition estimates.}

We next use \cref{pro 2.2} to prove that, for every
\(0\le t\le T\),
\begin{equation}\label{eq:base-composition-comparison1}
\norm{f\circ X(t)}_{\dB^{-1+\frac dp}_{p,1}}
+
\norm{f\circ Y(t)}_{\dB^{-1+\frac dp}_{p,1}}
\lesssim \norm f_{\dB^{-1+\frac dp}_{p,1}}.
\end{equation}

\noindent Since \(d\ge2\), one has $-\frac d{p'}<-1+\frac dp\le\frac dp$. Thus \(-1+\frac dp\) lies in the admissible range of \cref{pro 2.2}.
We distinguish the three cases
\(0<-1+\frac dp<1\), \(-1+\frac dp\ge1\), and \(-1<-1+\frac dp\le0\). 
 
 If \(0<-1+\frac dp<1\), then both \(X(t,\cdot)\)
and \(Y(t,\cdot)\) are globally bi-Lipschitz maps. Therefore case (1) of \cref{pro 2.2},
applied first to \(\Phi=X(t,\cdot)\) and then to \(\Phi=Y(t,\cdot)\)
gives
\[
\|f\circ X(t)\|_{\dot B^{-1+\frac{d}{p}}_{p,1}}
\lesssim \|f\|_{\dot B^{-1+\frac{d}{p}}_{p,1}},
\qquad
\|f\circ Y(t)\|_{\dot B^{-1+\frac{d}{p}}_{p,1}}
\lesssim \|f\|_{\dot B^{-1+\frac{d}{p}}_{p,1}}.
\]
Thus \eqref{eq:base-composition-comparison1} follows in this case.

If \(-1+\frac dp\ge1\), case~\textup{(3)} of
\cref{pro 2.2} additionally requires
\[
D\Phi-\Id\in\dot B^{\frac dp}_{p,1}.
\]
This condition for \(\Phi=X(t,\cdot)\) follows from
\eqref{eq:flow-deformation-comparison1}, while that for
\(\Phi=Y(t,\cdot)\) follows from
\eqref{eq:inverse-map-critical1}. Hence
\eqref{eq:base-composition-comparison1} follows from
case~\textup{(3)} of \cref{pro 2.2}.

It remains to consider \(-1<-1+\frac dp\le0\).  We verify the multiplier condition. Let
\(
q\in\dot B^{\frac{d}{p}}_{p,1}\) and
\(g\in\dot B^{1-\frac{d}{p}}_{p',\infty}.
\) By Bony's decomposition,
\[
qg=\dot T_q g+\dot T_g q+\dot R(q,g).
\]
 Applying  \(\dot B^{\frac{d}{p}}_{p,1}\hookrightarrow L^\infty\) and \eqref{1} gives
\[
\norm{\dot T_q g}_{\dot B^{1-\frac{d}{p}}_{p',\infty}}
\lesssim
\norm q_{\dot B^{\frac{d}{p}}_{p,1}}
\norm g_{\dot B^{1-\frac{d}{p}}_{p',\infty}}.
\]
Writing
$
a_k:=2^{k\frac dp}\norm{\dot\Delta_k q}_{L^p},$ Bernstein's inequality and the condition \(p<2d\) give
\begin{equation}\label{2.391}
2^{j(1-\frac dp)}
\norm{\dot\Delta_j\dot T_g q}_{L^{p'}}
\lesssim
\norm g_{\dot B^{1-\frac{d}{p}}_{p',\infty}}
\sum_{|k-j|\le4}a_k,
\end{equation}
and
\begin{equation}\label{2.401}
2^{j(1-\frac dp)}
\norm{\dot\Delta_j\dot R(q,g)}_{L^{p'}}
\lesssim
\norm g_{\dot B^{1-\frac{d}{p}}_{p',\infty}}
\sum_{k\ge j-3}2^{-(k-j)}a_k.
\end{equation}
Taking the supremum over \(j\) in \eqref{2.391}--\eqref{2.401}, we obtain
\begin{equation}\label{4}
  \norm{qg}_{\dot B^{1-\frac{d}{p}}_{p',\infty}}
\lesssim
\norm q_{\dot B^{\frac{d}{p}}_{p,1}}
\norm g_{\dot B^{1-\frac{d}{p}}_{p',\infty}}. 
\end{equation}
Consequently, for every  \(q\in\dot B^{\frac{d}{p}}_{p,1}\), multiplication by \(1+q\) defines a bounded operator on \(\dot B^{1-\frac{d}{p}}_{p',\infty}\). Hence \(1+q\in \mathcal{M}(\dot B^{1-\frac{d}{p}}_{p',\infty})\).

\noindent By \eqref{eq:flow-deformation-comparison1},
\eqref{eq:inverse-map-critical1}, \cref{p}, and the smoothness of the
determinant,
\[
J_X-1\in\dot B^{\frac{d}{p}}_{p,1},\quad J_Y-1\in\dot B^{\frac{d}{p}}_{p,1}.
\]
Hence, by \eqref{4}, we obtain
\[
J_X\in
\mathcal M(\dot B^{1-\frac{d}{p}}_{p',\infty}),\quad J_Y
\in
\mathcal M(\dot B^{1-\frac{d}{p}}_{p',\infty}).
\]
Case (2) of \cref{pro 2.2} applies to both \(X\) and \(Y\),
and \eqref{eq:base-composition-comparison1} thus follows.

\textbf{Step 3. The high-norm estimate in the smooth case.}
Assume first that \(\bar{v}\) is smooth. Since \(\bar{v} = v \circ X\) and \(A = (D_y X)^{-1}\), the chain rule yields
\[
(\partial_{x_i} h) \circ X = A_{ki} \partial_{y_k} (h \circ X)
\]
for every smooth function \(h\), with summation over repeated indices. Consequently,
\[
(\partial_{x_i} v^\alpha) \circ X = A_{ki} \partial_{y_k} \bar{v}^\alpha,
\]
and
\[
(\partial_{x_l} \partial_{x_i} v^\alpha) \circ X = A_{ml} \partial_{y_m} (A_{ki} \partial_{y_k} \bar{v}^\alpha)
= A_{ml} A_{ki} \partial_{y_m} \partial_{y_k} \bar{v}^\alpha + A_{ml} (\partial_{y_m} A_{ki}) \partial_{y_k} \bar{v}^\alpha.\]
By the derivative characterization of homogeneous Besov spaces,
\eqref{eq:base-composition-comparison1}, \eqref{3},
\eqref{eq:product-critical-low}, and
\eqref{eq:inverse-flow-derivative1}, and writing
\(A=\Id+(A-\Id)\), we obtain
\[
\begin{aligned}
\norm{v(t)}_{\dot B^{1+\frac{d}{p}}_{p,1}}
&\lesssim
\norm{(\nabla_x^2v(t))\circ X(t)}_{\dot B^{-1+\frac{d}{p}}_{p,1}}
\\
&\lesssim
\left(
\norm{\nabla_y^2\bar v(t)}_{\dot B^{-1+\frac{d}{p}}_{p,1}}
+
\norm{D_yA(t)}_{\dot B^{-1+\frac{d}{p}}_{p,1}}
\norm{\nabla_y\bar v(t)}_{\dot B^{\frac{d}{p}}_{p,1}}
\right)
\\
&\lesssim
\norm{\bar v(t)}_{\dot B^{1+\frac{d}{p}}_{p,1}}
\bigl(1+\norm{\bar v}_{L^1_t(\dB^{1+\frac dp}_{p,1})}\bigr).
\end{aligned}
\]
Integrating over \([0,t]\), we get
\begin{align*}
\norm v_{L^1_t(\dot B^{1+\frac{d}{p}}_{p,1})}
&\lesssim \bigl(1+\norm{\bar v}_{L^1_t(\dB^{1+\frac dp}_{p,1})}\bigr)\norm{\bar v}_{L^1_t(\dB^{1+\frac dp}_{p,1})}\\
&\lesssim \norm{\bar v}_{L^1_t(\dB^{1+\frac dp}_{p,1})}+\norm{\bar v}_{L^1_t(\dB^{1+\frac dp}_{p,1})}^2.
\end{align*}
This proves \eqref{eq:eulerian-lagrangian-comparison1} in the smooth case.

\textbf{Step 4. Approximation.}
We now remove the smoothness assumption. Let
\[
J_N:=\dot S_N-\dot S_{-N},
\qquad
\bar v^{(N)}:=J_N\bar v.
\]
Then
\begin{equation}\label{sl1}
\bar v^{(N)}\to\bar v
\quad\text{in}\quad
C_T(\dot B^{-1+\frac{d}{p}}_{p,1})
\cap
L^1_T(\dot B^{1+\frac{d}{p}}_{p,1})\qquad\text{as }N\to\infty.
\end{equation}
By interpolation,
\[
\norm{\bar v^{(N)}-\bar v}_{L^1_t(\dot B^{\frac{d}{p}}_{p,1})}
\lesssim
t^{1/2}
\norm{\bar v^{(N)}-\bar v}_{L^\infty_t(\dot B^{-1+\frac{d}{p}}_{p,1})}^{1/2}
\norm{\bar v^{(N)}-\bar v}_{L^1_t(\dot B^{1+\frac{d}{p}}_{p,1})}^{1/2}.
\]
Therefore
\begin{equation}\label{eq:bar-v-L1Linfty-approx1}
\bar v^{(N)}\to\bar v
\quad\text{in}\quad
L^1_t(\dot B^{\frac{d}{p}}_{p,1})
\hookrightarrow L^1_t(L^\infty) \quad\text{as }N\to\infty.
\end{equation}
By decreasing \(c>0\) if necessary, the uniform boundedness of
\(J_N\) ensures that the smallness assumptions in Steps~1--3 hold
uniformly in \(N\). Hence the smooth estimate applies uniformly in
\(N\).

\noindent Let \(X^{(N)}\) and \(Y^{(N)}\) be the maps associated with
\(\bar v^{(N)}\). From \eqref{eq:bar-v-L1Linfty-approx1},
\[
\norm{X^{(N)}-X}_{L^\infty_{t,x}}
\le
\norm{\bar v^{(N)}-\bar v}_{L^1_t(L^\infty)}
\to0\quad\text{as }N\to\infty.
\]
The uniform bi-Lipschitz bounds imply
\begin{equation}\label{yn1}
\norm{Y^{(N)}-Y}_{L^\infty_{t,x}}\to0\quad\text{as }N\to\infty.
\end{equation}
Moreover, by \eqref{sl1},
\[
\norm{\bar v^{(N)}}_{L^1_t(\dot B^{1+\frac{d}{p}}_{p,1})}
\to
\norm{\bar v}_{L^1_t(\dot B^{1+\frac{d}{p}}_{p,1})}\quad\text{as }N\to\infty.
\]
Setting
\[
v^{(N)}:=\bar v^{(N)}\circ Y^{(N)},
\]
we obtain
\[
\norm{v^{(N)}}_{L^1_t(\dot B^{1+\frac{d}{p}}_{p,1})}
\lesssim 
\norm{\bar v^{(N)}}_{L^1_t(\dot B^{1+\frac{d}{p}}_{p,1})}
+
\norm{\bar v^{(N)}}_{L^1_t(\dot B^{1+\frac{d}{p}}_{p,1})}^2
.
\]
Finally,
\[
v^{(N)}-v
=
(\bar v^{(N)}-\bar v)\circ Y^{(N)}
+
\bar v\circ Y^{(N)}-\bar v\circ Y.
\]
The first term tends to zero in \(L^1_tL^\infty\) as
\(N\to\infty\) by
\eqref{eq:bar-v-L1Linfty-approx1}. For the second term, since
\[
\nabla\bar v\in L^1_t(\dot B^{\frac{d}{p}}_{p,1})
\hookrightarrow L^1_t(L^\infty),
\]
and \eqref{yn1}, we obtain 
\[
\norm{\bar v\circ Y^{(N)}-\bar v\circ Y}_{L^1_tL^\infty}
\le
\norm{\nabla\bar v}_{L^1_tL^\infty}
\norm{Y^{(N)}-Y}_{L^\infty_{t,x}}
\to0\quad\text{as }N\to\infty.
\]
Hence,
\[
v^{(N)}\longrightarrow v
\quad\text{in}\quad
L^1(0,t;L^\infty(\mathbb R^d))\quad\text{as }N\to\infty.
\]
Choose a subsequence \(\{N_k\}_{k\ge1}\) such that
\[
\lim_{k\to\infty}
\|v^{(N_k)}\|_{L^1_t(\dot B^{1+\frac dp}_{p,1})}
=
\liminf_{N\to\infty}
\|v^{(N)}\|_{L^1_t(\dot B^{1+\frac dp}_{p,1})}.
\]
Passing to a further subsequence if necessary, the convergence in
\(L^1_t(L^\infty)\) implies that
\[
v^{(N_k)}(\tau)\longrightarrow v(\tau)
\quad\text{in }L^\infty(\mathbb R^d)\qquad\text{as }k\to\infty
\]
for almost every \(\tau\in(0,t)\), and hence in
\(\mathcal S'(\mathbb R^d)\). Therefore, for every \(j\in\mathbb Z\)
and almost every \(\tau\in(0,t)\),
\[
\|\dot\Delta_jv(\tau)\|_{L^p}
\le
\liminf_{k\to\infty}
\|\dot\Delta_jv^{(N_k)}(\tau)\|_{L^p}.
\]
By Fatou's lemma,
\[
\begin{aligned}
\|v\|_{L^1_t(\dot B^{1+\frac dp}_{p,1})}
&=
\int_0^t\sum_{j\in\mathbb Z}
2^{j(1+\frac dp)}
\|\dot\Delta_jv(\tau)\|_{L^p}\,\dd\tau\\
&\le
\liminf_{k\to\infty}
\int_0^t\sum_{j\in\mathbb Z}
2^{j(1+\frac dp)}
\|\dot\Delta_jv^{(N_k)}(\tau)\|_{L^p}\,\dd\tau=
\liminf_{N\to\infty}
\|v^{(N)}\|_{L^1_t(\dot B^{1+\frac dp}_{p,1})}.
\end{aligned}
\]
Using the uniform estimate for \(v^{(N)}\) and the convergence
\[
\bar v^{(N)}
\longrightarrow \bar v
\quad\text{in}\quad
L^1_t\bigl(\dot B^{1+\frac dp}_{p,1}\bigr)\quad\text{as }N\to\infty,
\]
we obtain
\[
\|v\|_{L^1_t(\dot B^{1+\frac dp}_{p,1})}
\lesssim 
\|\bar v\|_{L^1_t(\dot B^{1+\frac dp}_{p,1})}
+
\|\bar v\|_{L^1_t(\dot B^{1+\frac dp}_{p,1})}^{2}
.
\]
This proves \eqref{eq:eulerian-lagrangian-comparison1} and completes the
proof.
\end{proof}

\begin{proposition}\label{lem:uniform-local}
Let \((a_0,u_0)\) satisfy the assumptions of \cref{thm:main}. There exist
a neighborhood \(\mathcal U_0\subset X_p\) of \((a_0,u_0)\), a time
\(T_0>0\), and constants \(M,\underline b>0\) such that, for every $(\widetilde{a}_0,\widetilde{u}_0) \in \mathcal{U}_0$, the corresponding solution $(\widetilde{a},\widetilde{u})$ exists uniquely on \([0,T_0]\) and satisfies
\begin{equation}\label{eq:uniform-local-bounds}
\begin{aligned}
&\norm {\widetilde{a}}_{\tdL^\infty_{T_0}(\dB^{\frac{d}{p}}_{p,1})}
+\norm {\widetilde{u}}_{\tdL^\infty_{T_0}(\dB^{-1+\frac{d}{p}}_{p,1})}
+\norm {\widetilde u}_{{L^1_{T_0}}(\dB^{{1+\frac{d}{p}}}_{p,1})}
\le M,
\\
&1+\widetilde{a}(t,x)\ge\frac{\kappa_0}{4},
\qquad
\underline b\le \widetilde{b}(t,x)\le\frac4{\kappa_0},
\end{aligned}
\end{equation}
where $\widetilde{b}:=(1+\widetilde{a})^{-1}$.
Moreover, for every \(\eta>0\), there exist a neighborhood
\(\mathcal U_\eta\subset\mathcal U_0\) of \((a_0,u_0)\) and a time
\(0<T_\eta\le T_0\) such that
\begin{equation}\label{eq:small-W}
\sup_{(\widetilde a_0,\widetilde u_0)\in\mathcal U_\eta}
\norm{\widetilde u}_{L^1_{T_\eta}(\dB^{{1+\frac{d}{p}}}_{p,1})}
\le\eta.
\end{equation}
\end{proposition}

\begin{proof}
By \(\cref{theo 2.2}\), there exist \(T>0\), a neighborhood
\(\mathcal U^{(1)}\subset X_p\) of \((a_0,u_0)\), and a constant
\(M>0\) such that every
\((\widetilde a_0,\widetilde u_0)\in\mathcal U^{(1)}\) generates a
unique solution on \([0,T]\) satisfying
\begin{equation}\label{2.47}
\norm{\widetilde a}_{L^\infty_T(\dB^{\frac{d}{p}}_{p,1})}
+\norm{\widetilde u}_{L^\infty_T(\dB^{-1+\frac{d}{p}}_{p,1})}
+\norm{\widetilde u}_{L^1_T(\dB^{{1+\frac{d}{p}}}_{p,1})}
\le M.
\end{equation}
\noindent By \eqref{eq:positive-density} and
\(\dB^{\frac{d}{p}}_{p,1}\hookrightarrow L^\infty\), there exist a
neighborhood \(\mathcal U^{(2)}\) of \((a_0,u_0)\) and a constant $\underline{b}>0$ such that
\begin{equation}\label{2.48}
1+\widetilde a_0\ge\frac{\kappa_0}{2},
\qquad
(1+\widetilde a_0)^{-1}\ge2\underline b
\end{equation}
for all  
$(\tilde{a}_0, \tilde{u}_0) \in \mathcal{U}^{(2)}. $

\noindent Let \(\bar{u}\) and \(\tilde{\bar{u}}\) denote the Lagrangian  
velocities corresponding to \((a_0, u_0)\) and  
\((\tilde{a}_0, \tilde{u}_0)\), respectively. The local Lipschitz  
continuity of the Lagrangian solution map yields a neighborhood  
\(\mathcal U^{(3)}\) of \((a_0, u_0)\)  such that
\begin{equation}\label{2.49}
\norm{\widetilde{\bar u}-\bar u}_{L^1_T(\dB^{1+\frac{d}{p}}_{p,1})}
\lesssim 
\norm{(\widetilde a_0-a_0,\widetilde u_0-u_0)}_{X_p}
\end{equation}
for all  
$(\tilde{a}_0, \tilde{u}_0) \in \mathcal{U}^{(3)}. $ Set
\[
\mathcal U
:=
\mathcal U^{(1)}
\cap\mathcal U^{(2)}
\cap\mathcal U^{(3)}
\cap B_{X_p}\bigl((a_0,u_0),1\bigr),
\]
where
\[
B_{X_p}\bigl((a_0,u_0),1\bigr)
:=
\left\{
(\widetilde a_0,\widetilde u_0)\in X_p:
\|(\widetilde a_0-a_0,\widetilde u_0-u_0)\|_{X_p}<1
\right\}.
\]
Then \(\mathcal U\) is a bounded neighborhood of \((a_0,u_0)\), and
\eqref{2.47}--\eqref{2.49} hold simultaneously for every datum in
\(\mathcal U\).

\noindent We first establish a uniform smallness property for the Eulerian
velocities. Let \(\gamma>0\). Choose \(\varepsilon>0\) so that
\cref{lem:eulerian-lagrangian-high-norm} applies whenever
$
\norm{\widetilde{\bar u}}_{L^1_t
(\dB^{1+\frac dp}_{p,1})}
\le2\varepsilon
$
and yields
$
\norm{\widetilde u}_{L^1_t
(\dB^{1+\frac dp}_{p,1})}
\le\gamma.
$
Since
$
\bar u\in L^1(0,T;\dB^{1+\frac dp}_{p,1}),
$
there exists \(T_\gamma\in(0,T]\) such that
\[
\norm{\bar u}_{L^1_{T_\gamma}
(\dB^{1+\frac dp}_{p,1})}
\le\varepsilon.
\]
By the local Lipschitz continuity of the Lagrangian solution map,
after shrinking \(\mathcal U\) to a neighborhood
\(\mathcal U_\gamma\) of \((a_0,u_0)\), we may assume that
$
\norm{\widetilde{\bar u}-\bar u}_{L^1_{T_\gamma}
(\dB^{1+\frac dp}_{p,1})}
\le\varepsilon.
$
Hence
\[
\norm{\widetilde{\bar u}}_{L^1_{T_\gamma}
(\dB^{1+\frac dp}_{p,1})}
\le\|\bar u\|_{L^1_{T_\gamma}
(\dot B^{1+\frac dp}_{p,1})}
+
\|\widetilde{\bar u}-\bar u\|_{L^1_{T_\gamma}
(\dot B^{1+\frac dp}_{p,1})}\le2\varepsilon,
\]
and therefore
\begin{equation}\label{2.53}
\sup_{(\widetilde a_0,\widetilde u_0)\in\mathcal U_\gamma}
\norm{\widetilde u}_{L^1_{T_\gamma}
(\dB^{1+\frac dp}_{p,1})}
\le\gamma.
\end{equation}
Since \(\gamma>0\) is arbitrary in \eqref{2.53}, we now choose it
sufficiently small so that, by
\(\dot B^{1+\frac dp}_{p,1}\hookrightarrow W^{1,\infty}\),
\[
\sup_{(\widetilde a_0,\widetilde u_0)\in\mathcal U_\gamma}
\int_0^{T_\gamma}
\norm{\Div\widetilde u(t)}_{L^\infty}\,\dd t
\le\log2.
\]
Set
\[
\mathcal U_0:=\mathcal U_\gamma,
\qquad
T_0:=T_\gamma.
\]
Let \(\widetilde X\) be the flow associated with \(\widetilde u\).
The continuity equation yields
\begin{equation}\label{2.57}
1+\widetilde a(t,\widetilde X(t,y))
=
(1+\widetilde a_0(y))
\exp\left(
-\int_0^t
\Div\widetilde u(\tau,\widetilde X(\tau,y))\,\dd\tau
\right).
\end{equation}
Consequently,
\begin{equation}\label{eq:uniform-density-coefficient-bounds}
 1+\widetilde a(t,x)\ge\frac{\kappa_0}{4},
\qquad
\underline b
\le\widetilde b(t,x)
\le\frac4{\kappa_0},
\qquad
0\le t\le T_0.   
\end{equation}
Since \(\mathcal U_0\subset\mathcal U\) and \(T_0\le T\), restricting
\eqref{2.47} to \([0,T_0]\) gives
\begin{equation}\label{eq:uniform-local-bochner-bounds}
\begin{aligned}
&\norm {\widetilde{a}}_{L^\infty_{T_0}(\dB^{\frac{d}{p}}_{p,1})}
+\norm {\widetilde{u}}_{L^\infty_{T_0}(\dB^{-1+\frac{d}{p}}_{p,1})}
+\norm {\widetilde{u}}_{L^1_{T_0}(\dB^{1+\frac{d}{p}}_{p,1})}
\le M.
\end{aligned}
\end{equation}
It remains to upgrade the first two norms in \eqref{eq:uniform-local-bochner-bounds} to
Chemin--Lerner norms. 
By \eqref{eq:uniform-local-bochner-bounds},
\eqref{eq:uniform-density-coefficient-bounds}, and
\eqref{eq:composition}, we have
\begin{equation}\label{eq:uniform-bG-composition}
\norm{\widetilde{b}-1}_{L^\infty_{T_0}(\dB^{\frac{d}{p}}_{p,1})}
+
\norm{G(\widetilde{a})}_{L^\infty_{T_0}(\dB^{\frac{d}{p}}_{p,1})}
\le C.
\end{equation}
Using
\[
\partial_t\widetilde u
=
\widetilde b\cA\widetilde u
-\widetilde u\cdot\nabla\widetilde u
-\nabla G(\widetilde a),
\]
we first estimate the right-hand side. By
\cref{lem:basic-products},
\begin{equation}\label{eq:velocity-product-interface}
\begin{aligned}
\norm{\widetilde b\cA\widetilde u}
_{\dB^{-1+\frac dp}_{p,1}}
&\lesssim
\left(
1+\norm{\widetilde b-1}
_{\dB^{\frac dp}_{p,1}}
\right)
\norm{\widetilde u}
_{\dB^{1+\frac dp}_{p,1}},
\\
\norm{\widetilde u\cdot\nabla\widetilde u}
_{\dB^{-1+\frac dp}_{p,1}}
&\lesssim
\norm{\widetilde u}_{\dB^{-1+\frac dp}_{p,1}}
\norm{\widetilde u}_{\dB^{1+\frac dp}_{p,1}}.
\end{aligned}
\end{equation}
Consequently, by
\eqref{eq:uniform-local-bochner-bounds},
\eqref{eq:uniform-bG-composition}, and
\eqref{eq:velocity-product-interface},
\begin{equation}\label{eq:time-derivative-velocity}
\begin{aligned}
\norm{\partial_t\widetilde u}
_{L^1_{T_0}(\dB^{-1+\frac dp}_{p,1})}
\lesssim{}&
\left(
1+
\norm{\widetilde b-1}
_{L^\infty_{T_0}(\dB^{\frac dp}_{p,1})}
+
\norm{\widetilde u}
_{L^\infty_{T_0}(\dB^{-1+\frac dp}_{p,1})}
\right)
\norm{\widetilde u}
_{L^1_{T_0}(\dB^{1+\frac dp}_{p,1})}
\\
&+
T_0
\norm{G(\widetilde a)}
_{L^\infty_{T_0}(\dB^{\frac dp}_{p,1})}\le{}
C.
\end{aligned}
\end{equation}
It follows that
\begin{equation}\label{eq:velocity-CL-interface}
\begin{aligned}
\norm{\widetilde u}
_{\tdL^\infty_{T_0}
(\dB^{-1+\frac dp}_{p,1})}
&\le
\norm{\widetilde u_0}
_{\dB^{-1+\frac dp}_{p,1}}
+
\norm{\partial_t\widetilde u}
_{L^1_{T_0}
(\dB^{-1+\frac dp}_{p,1})}
\le
C.
\end{aligned}
\end{equation}
For the density equation
\[
\partial_t\widetilde a
+\widetilde u\cdot\nabla\widetilde a
=
-(1+\widetilde a)\Div\widetilde u,
\qquad
\widetilde a|_{t=0}=\widetilde a_0,
\]
set
\[
V(t)
:=
\int_0^t
\norm{\nabla\widetilde u(\tau)}_
{\dB^{\frac dp}_{p,1}}
\dd\tau.
\]
By \eqref{eq:uniform-local-bochner-bounds},
\begin{equation}\label{eq:density-transport-control}
V(T_0)
\lesssim
\norm{\widetilde u}
_{L^1_{T_0}(\dB^{1+\frac dp}_{p,1})}
\le M.
\end{equation}
Moreover, \cref{lem:basic-products} gives
\begin{equation}\label{eq:density-source-interface}
\begin{aligned}
\norm{(1+\widetilde a)\Div\widetilde u}
_{L^1_{T_0}(\dB^{\frac dp}_{p,1})}
&\lesssim
\left(
1+
\norm{\widetilde a}
_{L^\infty_{T_0}(\dB^{\frac dp}_{p,1})}
\right)
\norm{\widetilde u}
_{L^1_{T_0}(\dB^{1+\frac dp}_{p,1})}
\\
&\le
C(1+M)M.
\end{aligned}
\end{equation}
Applying \cref{lem:transport-estimate}, and using
\eqref{eq:density-transport-control} and
\eqref{eq:density-source-interface}, we obtain
\begin{equation}\label{eq:density-CL-interface}
\begin{aligned}
\norm{\widetilde a}
_{\tdL^\infty_{T_0}(\dB^{\frac dp}_{p,1})}
&\le
\exp\bigl(CV(T_0)\bigr)
\left(
\norm{\widetilde a_0}
_{\dB^{\frac dp}_{p,1}}
+
\norm{(1+\widetilde a)\Div\widetilde u}
_{L^1_{T_0}(\dB^{\frac dp}_{p,1})}
\right)
\\
&\le
C.
\end{aligned}
\end{equation}
Combining
\eqref{eq:uniform-density-coefficient-bounds},
\eqref{eq:uniform-local-bochner-bounds},
\eqref{eq:velocity-CL-interface}, and
\eqref{eq:density-CL-interface}, and enlarging \(M\) if necessary, we obtain \eqref{eq:uniform-local-bounds}.

\noindent Finally, let \(\eta>0\). Applying the argument leading to
\eqref{2.53} with \(\gamma=\eta\), there exist a neighborhood
\(\mathcal V_\eta \subset\mathcal U\)  and \(0<T_\eta'\le T\) such that
\[
\sup_{(\widetilde a_0,\widetilde u_0)\in\mathcal V_\eta}
\norm{\widetilde u}_{L^1_{T_\eta'}
(\dB^{1+\frac dp}_{p,1})}
\le \eta.
\]
Set
\[
\mathcal U_\eta:=\mathcal U_0\cap\mathcal V_\eta,
\qquad
T_\eta:=\min\{T_0,T_\eta'\}.
\]
Then \(\mathcal U_\eta\subset\mathcal U_0\) is a neighborhood of
\((a_0,u_0)\), \(0<T_\eta\le T_0\), and
\[
\sup_{(\widetilde a_0,\widetilde u_0)\in\mathcal U_\eta}
\norm{\widetilde u}_{L^1_{T_\eta}
(\dB^{1+\frac dp}_{p,1})}
\le \eta.
\]
This proves \eqref{eq:small-W} and completes the proof.
\end{proof}

\subsection{Time-weighted estimates}

Following \cite{CMZ2010R}, fix \(c_*>0\) and define
\[
e_j(T):=\bigl(1-e^{-c_*2^{2j}T}\bigr)^{1/2},
\qquad
\omega_j(T):=\sum_{\ell\ge j}2^{j-\ell}e_\ell(T).
\]
Then
\begin{equation}\label{eq:time-weight-properties}
 0\le e_j(T)\le\omega_j(T)\le2,
\qquad
\omega_j(T)\rightarrow0\quad 
\text{as }T\downarrow0  
\end{equation}
for each fixed \(j\).
Let \( s \in \mathbb{R} \), \( 1 \leq p, r \leq +\infty \), and \( 0 < T < +\infty \). The weighted Besov space \( \dot{B}_{p,r}^s(\omega_T) \) is defined by
\[
\|f\|_{\dot{B}_{p,r}^s(\omega_T)} \stackrel{\mathrm{def}}{=} \left\| 2^{ks} \omega_k(T) \| \Delta_k f \|_{p} \right\|_{l^r}.
\]

\begin{proposition}
\label{prop:time-density-small}
For every \(\varepsilon>0\), there exist \(r_\varepsilon>0\) and
\(0<T_\varepsilon\le T_0\) such that
\[
\mathcal U_\varepsilon
:=
\left\{
(\widetilde a_0,\widetilde u_0)\in\mathcal U_0:
\norm{
(\widetilde a_0-a_0,\widetilde u_0-u_0)
}_{X_p}
<r_\varepsilon
\right\}
\]
is a neighborhood of \((a_0,u_0)\) in \(X_p\), and
\[
\sup_{
(\widetilde a_0,\widetilde u_0)\in\mathcal U_\varepsilon
}
\norm{\widetilde a}_{\tdL^\infty_{T_\varepsilon}
(\dB^{\frac{d}{p}}_{p,1}(\omega_{T_\varepsilon}))}
\le\varepsilon.
\]
\end{proposition}
\begin{proof}Fix $\varepsilon>0$, and let $0<T\le T
_0$. For \(0\le t\le T\), set
\[
W(t):=
\norm{\widetilde u}_{L^1_t
(\dB^{1+\frac{d}{p}}_{p,1})}.
\]
Applying \cref{lem:time-weighted-transport-CMZ} to
\[
\partial_t\widetilde a
+\widetilde u\cdot\nabla\widetilde a
=-(1+\widetilde a)\Div \widetilde u,
\]
and using $\norm {\nabla \widetilde u}_{L^1_T
(\dB^{\frac{d}{p}}_{p,1})}\lesssim W(T)$, we obtain
\[
\begin{aligned}
\norm{\widetilde a}_{\tdL^\infty_T
(\dB^{\frac{d}{p}}_{p,1}(\omega_T))}
&\lesssim{}
e^{CW(T)}
\Bigl(
\norm{\widetilde a_0}_{\dB^{\frac{d}{p}}_{p,1}(\omega_T)}+
\norm{(1+\widetilde a)\Div\widetilde u}_{L^1_T
(\dB^{\frac{d}{p}}_{p,1}(\omega_T))}
\Bigr).
\end{aligned}
\]
Writing
\[
(1+\widetilde a)\Div\widetilde u=\Div\widetilde u+\widetilde a\Div\widetilde u
\]
and applying \cref{lem:time-weighted-critical-product}, together with 
\(\omega_j(T)\le2\), we obtain
\[
\begin{aligned}
\norm{(1+\widetilde a)\Div\widetilde u}_{L^1_T
(\dB^{\frac{d}{p}}_{p,1}(\omega_T))}
\lesssim{}&
\left(
1+
\norm{\widetilde a}_{\tdL^\infty_T
(\dB^{\frac{d}{p}}_{p,1})}+\norm{\widetilde a}_{\tdL^\infty_T
(\dB^{\frac{d}{p}}_{p,1}(\omega_T))}
\right)
W(T).
\end{aligned}
\]
By \cref{lem:uniform-local}, the unweighted density norm is uniformly
bounded on \(\mathcal U_0\). Hence there exists $\delta_*>0$ such that, whenever $W(T)
\le\delta_*$,
\[
\norm{\widetilde a}_{\tdL^\infty_T
(\dB^{\frac{d}{p}}_{p,1}(\omega_T))}
\le
C_0\left(
\norm{\widetilde a_0}_{\dB^{\frac{d}{p}}_{p,1}(\omega_T)}
+
W(T)
\right),
\]
where \(C_0\) is independent of \(T\), the initial datum, and the corresponding solution.
Choose \(0<\delta\le\delta_*\) such that
\[
C_0\delta\le\frac{\varepsilon}{4}.
\]
By \eqref{eq:small-W}, there exist a neighborhood
\(\mathcal U_\delta\subset\mathcal U_0\) of \((a_0,u_0)\) and a time
\(0<T_\delta\le T_0\) such that
\[
\sup_{
(\widetilde a_0,\widetilde u_0)\in\mathcal U_\delta
}
\norm{\widetilde u}_{L^1_{T_\delta}
(\dB^{1+\frac{d}{p}}_{p,1})}
\le\delta.
\]
Since \(\mathcal U_\delta\) is a neighborhood of \((a_0,u_0)\) in
\(X_p\), we may choose \(r_\varepsilon>0\) such that
\[
\mathcal U_\varepsilon
:=
\left\{
(\widetilde a_0,\widetilde u_0)\in\mathcal U_0:
\norm{
(\widetilde a_0-a_0,\widetilde u_0-u_0)
}_{X_p}
<r_\varepsilon
\right\}
\subset\mathcal U_\delta
\]
and
\[
2C_0r_\varepsilon\le\frac{\varepsilon}{4}.
\]
By \cite[Remark 3.3]{CMZ2010R}, applied to the compact set
\(\{a_0\}\), there exists \(0<T_\varepsilon\le T_\delta\) such that
\[
C_0
\norm{a_0}_{\dB^{\frac{d}{p}}_{p,1}(\omega_{T_\varepsilon})}
\le\frac{\varepsilon}{2}.
\]
For every
\((\widetilde a_0,\widetilde u_0)\in\mathcal U_\varepsilon\), we have
$
\norm{\widetilde u}_{L^1_{T_\varepsilon}
(\dB^{1+\frac{d}{p}}_{p,1})}
\le\delta
$
and, since \(\omega_j(T_\varepsilon)\le2\),
\[
\begin{aligned}
\norm{\widetilde a_0}
_{\dB^{\frac{d}{p}}_{p,1}(\omega_{T_\varepsilon})}
&\le
\norm{a_0}
_{\dB^{\frac{d}{p}}_{p,1}(\omega_{T_\varepsilon})}
+
2\norm{\widetilde a_0-a_0}_{\dB^{\frac{d}{p}}_{p,1}}
\\
&\le
\norm{a_0}
_{\dB^{\frac{d}{p}}_{p,1}(\omega_{T_\varepsilon})}
+
2r_\varepsilon.
\end{aligned}
\]
Consequently,
\[
\begin{aligned}
\norm{\widetilde a}_{\tdL^\infty_{T_\varepsilon}
(\dB^{\frac{d}{p}}_{p,1}(\omega_{T_\varepsilon}))}
&\le
C_0\left(
\norm{a_0}_{\dB^{\frac{d}{p}}_{p,1}
(\omega_{T_\varepsilon})}
+
2r_\varepsilon
+
\delta
\right)\le\varepsilon.
\end{aligned}
\]
Taking the supremum over
\((\widetilde a_0,\widetilde u_0)\in\mathcal U_\varepsilon\)
completes the proof.
\end{proof}
\begin{corollary}\label{cor:common-m}
For every prescribed \(\eta_b>0\), there exist \(r>0\), a time
\(0<T\le T_0\), and an integer \(m=m(T)\) with the following
properties. Set
\[
\mathcal U_r
:=
\left\{
(a_0',u_0')\in\mathcal U_0:
\norm{a_0'-a_0}_{\dB^{\frac dp}_{p,1}}
+
\norm{u_0'-u_0}_{\dB^{-1+\frac dp}_{p,1}}
<r
\right\},
\]
and denote by \(\mathscr S_{r,T}\) the family of solutions \((a,u)\)
on \([0,T]\) issued from initial data in \(\mathcal U_r\). Thus, by \cref{lem:uniform-local},
\[
(a,u)\in
\tdL^\infty_T(\dB^{\frac dp}_{p,1})
\times
\left(
\tdL^\infty_T(\dB^{-1+\frac dp}_{p,1})
\cap
L^1_T(\dB^{1+\frac dp}_{p,1})
\right).
\]
For each \((a,u)\in\mathscr S_{r,T}\), set \(b=(1+a)^{-1}\) and
\begin{equation}\label{eq:coefficient-splitting}
b_m:=1+\dS_m(b-1),
\qquad
b^h:=(\Id-\dS_m)(b-1).
\end{equation}
Then
\[
b_m(t,x)\ge\frac{\underline b}{2}
\qquad
\text{for all }
(a,u)\in\mathscr S_{r,T}
\text{ and }(t,x)\in[0,T]\times\R^d,
\]
and
\begin{equation}\label{eq:coefficient-tail-small}
\sup_{(a,u)\in\mathscr S_{r,T}}
\norm{b^h}_{\tdL^\infty_T
(\dB^{\frac dp}_{p,1})}
\le\eta_b.
\end{equation}
\end{corollary}

\begin{proof}
Let \(N_0\ge1\), depending only on the supports of the
Littlewood--Paley cutoffs, be such that
\[
\dDelta_j(\Id-\dS_m)f=0
\qquad\text{whenever }j\le m-N_0.
\]
Let \(\varepsilon>0\), to be fixed below. By
\cref{prop:time-density-small}, after decreasing \(r\) and \(T\), if
necessary, we may assume that
\begin{equation}\label{eq:uniform-weighted-density-small}
\sup_{(a,u)\in\mathscr S_{r,T}}
\norm a_{\tdL^\infty_T
(\dB^{\frac dp}_{p,1}(\omega_T))}
\le\varepsilon.
\end{equation}
On the other hand, \cref{lem:uniform-local} gives
\begin{equation}\label{eq:uniform-density-range}
\sup_{(a,u)\in\mathscr S_{r,T}}
\norm a_{\tdL^\infty_T(\dB^{\frac dp}_{p,1})}
\le M,
\qquad
1+a(t,x)\ge\frac{\kappa_0}{4}.
\end{equation}
The critical embedding
\(\dB^{\frac dp}_{p,1}\hookrightarrow L^\infty\) and
\eqref{eq:uniform-density-range} imply that the ranges of all densities
in \(\mathscr S_{r,T}\) are contained in a common compact subset of
\((-1,\infty)\).
After extending
\[
F(z):=\frac1{1+z}-1
\]
smoothly outside this compact set, we may apply
\cref{lem:time-weighted-composition-CMZ}. Since
\(b-1=F(a)\), estimates
\eqref{eq:uniform-weighted-density-small} and
\eqref{eq:uniform-density-range} give
\begin{equation}\label{eq:uniform-weighted-b-small}
\begin{aligned}
\sup_{(a,u)\in\mathscr S_{r,T}}
\norm{b-1}_{\tdL^\infty_T
(\dB^{\frac dp}_{p,1}(\omega_T))}
&\le
C
\sup_{(a,u)\in\mathscr S_{r,T}}
\norm a_{\tdL^\infty_T
(\dB^{\frac dp}_{p,1}(\omega_T))}
\le C\varepsilon.
\end{aligned}
\end{equation}
Let \(m=m(T)\) be the smallest integer satisfying
$
c_*2^{2m}T\ge1.
$
Then
\begin{equation}\label{eq:common-m-time-bound}
1\le c_*2^{2m}T<4.
\end{equation}
In particular, \(m\) depends only on \(T\) and is independent of the
particular solution in \(\mathscr S_{r,T}\).

\noindent If \(\dDelta_jb^h\neq0\), then \(j\ge m-N_0\). Hence
\eqref{eq:common-m-time-bound} implies
\[
c_*2^{2j}T\ge2^{-2N_0}.
\]
Consequently,
\[
\omega_j(T)
\ge e_j(T)
\ge
\left(1-e^{-2^{-2N_0}}\right)^{1/2}
=:c_\omega>0.
\]
It follows from \eqref{eq:uniform-weighted-b-small} that
\begin{equation}\label{eq:bh-small-from-weight}
\begin{aligned}
\sup_{(a,u)\in\mathscr S_{r,T}}
\norm{b^h}_{\tdL^\infty_T
(\dB^{\frac dp}_{p,1})}
&\lesssim
c_\omega^{-1}
\sup_{(a,u)\in\mathscr S_{r,T}}
\norm{b-1}_{\tdL^\infty_T
(\dB^{\frac dp}_{p,1}(\omega_T))}\le
C\varepsilon.
\end{aligned}
\end{equation}
Choose \(\varepsilon>0\) sufficiently small so that
\[
C\varepsilon
\le
\min\left\{
\eta_b,\frac{\underline b}{2C_{\rm emb}}
\right\}.
\]
where \(C_{\mathrm{emb}}\) is the constant in the embedding
\(\dB^{\frac dp}_{p,1}\hookrightarrow L^\infty\).
Then \eqref{eq:bh-small-from-weight} proves
\eqref{eq:coefficient-tail-small} and gives
\[
\sup_{(a,u)\in\mathscr S_{r,T}}
\norm{b^h}_{L^\infty_{T,x}}
\le\frac{\underline b}{2}.
\]
Since \(b\ge\underline b\) uniformly by
\cref{lem:uniform-local}, we conclude that
\[
b_m=b-b^h\ge\frac{\underline b}{2}
\]
on \([0,T]\times\R^d\), uniformly for all
\((a,u)\in\mathscr S_{r,T}\).
\end{proof}

\begin{remark}\label{rem:common-m}
Let $\alpha:=\underline b\min\{\mu,\lambda+2\mu\}.$
Choose \(\eta_b>0\) sufficiently small so that $C\eta_b\le\eta_0\alpha,$
where \(\eta_0\) is the absorption constant in
\cref{prop:weighted-lame}. By \cref{cor:common-m}, after decreasing
\(r\) and \(T\) if necessary, there exists a common cutoff \(m=m(T)\)
such that
\[
b_m\ge\frac{\underline b}{2},
\qquad
\sup_{(a,u)\in\mathscr S_{r,T}}
\norm{\nabla b^h}_{L^\infty_T
(\dB^{-1+\frac dp}_{p,1})}
\le\eta_0\alpha,
\]
and
\[
1\le c_*2^{2m}T<4.
\]
These provide the coefficient smallness and time--frequency conditions required in the subsequent variable-coefficient estimate.
In particular, no smallness assumption is imposed on
$
\sup_{(a,u)\in\mathscr S_{r,T}}
\norm a_{\tdL^\infty_T
(\dB^{\frac dp}_{p,1})}.
$
\end{remark}

\subsection{Frequency-envelope-weighted estimates}

Throughout this subsection, fix $0 < \delta_0 < \min \left\{ 1, \frac{d}{p} - 1 + \min \left( \frac{d}{p}, \frac{d}{p'}\right) \right\}.$ We use the following frequency-envelop weight function.
\begin{definition}\cite[Definition~2.1]{GuoYangZhang2026}\label{def:AF}
A sequence
\(\Omega=(\Omega_j)_{j\in\Z}\) belongs to \(AF(\delta_0)\) if
\begin{equation}\label{eq:AF}
\Omega_j=1\ (j\le0),
\qquad
1\le\Omega_j\le\Omega_{j+1}\le2^{\delta_0}\Omega_j\ (j>0).
\end{equation}
\noindent For \(\sigma\in\R\) and \(1\le p,r,q\le\infty\), set
\[
\norm f_{\dB^\sigma_{p,r}(\Omega)}
:=\|(\Omega_j2^{j\sigma}\norm{\dDelta_jf}_{L^p})_{j\in\mathbb{Z}}\|_{l^r}.
\]
The corresponding Chemin--Lerner norm is
\[
\norm f_{\widetilde{L}^q_T(\dB^\sigma_{p,r}(\Omega))}
:=\|(\Omega_j2^{j\sigma}\norm{\dDelta_jf}_{L^q_T(L^p)})_{j\in\mathbb{Z}}\|_{l^r}.
\]
\end{definition}
\noindent In what follows, we use admissible weights satisfying
\(\Omega_1\le 2^{\delta_0}\).
\begin{remark}\label{rem:envelope-discrete-convolution}
Let \(\Omega\in AF(\delta_0)\). By \eqref{eq:AF} and the above normalization,, for all
\(j,k\in\mathbb Z\),
\begin{equation}\label{eq:weight-ratio-abstract}
\frac{\Omega_j}{\Omega_k}
\le 2^{\delta_0|j-k|}.
\end{equation}
Hence, for every \(\gamma>\delta_0\) and every nonnegative sequence
\((x_k)_{k\in\mathbb Z}\),
\begin{equation}\label{eq:envelope-linear-convolution}
\sum_{j,k\in\mathbb Z}
2^{-\gamma|j-k|}
\frac{\Omega_j}{\Omega_k}x_k
\le
C\sum_{k\in\mathbb Z}x_k.
\end{equation}
Moreover, for all nonnegative \(\ell^1\)-sequences \(x\) and \(y\),
\begin{equation}\label{eq:envelope-bilinear-convolution}
\sum_{j,k\in\mathbb Z}
2^{-\gamma|j-k|}
\frac{\Omega_j}{\Omega_k}x_jy_k
\le
C\norm{x}_{\ell^1}\norm{y}_{\ell^1},
\end{equation}
where \(C\) depends only on \(\gamma\) and \(\delta_0\).
\end{remark}

\begin{proposition}\label{lem:common-envelope}
Assume \((a_0^n,u_0^n)\to(a_0,u_0)\) in \(X_p\) as $n\to\infty$. There exists
\(\Omega\in AF(\delta_0)\) such that
\begin{equation}\label{eq:Omega-diverges}
\Omega_j\rightarrow\infty\quad(j\to\infty)
\end{equation}
Moreover,
\begin{equation}\label{eq:common-envelope-bound}
\begin{aligned}
\sup_{n\ge1}
\left(
\norm{a_0^n}_{\dB^{\frac dp}_{p,1}(\Omega)}
+
\norm{u_0^n}_{\dB^{-1+\frac dp}_{p,1}(\Omega)}
\right)
+
\norm{a_0}_{\dB^{\frac dp}_{p,1}(\Omega)}
+
\norm{u_0}_{\dB^{-1+\frac dp}_{p,1}(\Omega)}
<\infty.
\end{aligned}
\end{equation}
\end{proposition}

\begin{proof}
For \(n\ge1\), set
\[
c_j^n
:=
2^{j\frac dp}\norm{\dDelta_j a_0^n}_{L^p}
+
2^{j(-1+\frac dp)}\norm{\dDelta_j u_0^n}_{L^p}
\]
and 
\[
c_j
:=
2^{j\frac dp}\norm{\dDelta_j a_0}_{L^p}
+
2^{j(-1+\frac dp)}\norm{\dDelta_j u_0}_{L^p}.
\]
Then  \(c^n \to c \) in $l^1(\mathbb{Z})$.  Consequently, $
\{c^n : n \geq 1\} \cup \{c\}$ is compact in \(\ell^1(\mathbb{Z})\), and thus 
\[
    r_N := \sup_{n \geq 1} \sum_{j \geq N} c_j^n \to 0 \quad (N \to \infty).
\]
Choose $1 \leq N_1 < N_2 < \cdots$ such that
\[
r_{N_k} \leq 2^{-3k}.
\]
For \( j \in \mathbb{Z} \), define
\[
K(j) := \max\{k : N_k \leq j\},
\]
with \( K(j) = 0 \) if \( j < N_1 \). Let $\theta_j = 0 \quad (j \leq N_1),$
and for \( j \geq N_1 \)
\[
\theta_{j+1} := \min\{\theta_j + \delta_0, K(j+1)\}.
\]
Then
\[
0 \leq \theta_{j+1} - \theta_j \leq \delta_0,\quad\theta_j \leq K(j), \quad \theta_j \to \infty.
\]
Hence
\(\Omega_j:=2^{\theta_j}\) belongs to \(AF(\delta_0)\), satisfies
\(\Omega_1=1\), and \(\Omega_j\to\infty\) as \(j\to\infty\).

\noindent If \(N_k \leq j < N_{k+1}\), then $
\Omega_j \leq 2^k$. Therefore,
\begin{align*}
    \sup_{n \geq 1} \sum_{j \geq N_1} \Omega_j c_j^n &\leq \sum_{k \geq 1} 2^k \sup_{n \geq 1} \sum_{N_k \leq j < N_{k+1}}c_j^n\\
    &\leq \sum_{k \geq 1} 2^k r_{N_k} \leq \sum_{k \geq 1} 2^{-2k} < \infty.
\end{align*}
Since \(\Omega_j = 1\) for \(j < N_1\),
\[
\sup_{n \geq 1} \sum_{j < N_1} \Omega_j c_j^n \leq \sup_{n \geq 1} \| (a_0^n, u_0^n) \|_{X_p} < \infty.
\]
Thus,
\begin{equation}\label{weight c}
\sup_{n\ge1}
\sum_{j\in\Z}\Omega_jc_j^n
<\infty.
\end{equation}
Moreover, \(c_j^n\to c_j\) for every \(j\in\Z\). Fatou's lemma gives
\begin{equation}\label{weight c1}
\sum_{j\in\Z}\Omega_jc_j
\le
\liminf_{n\to\infty}
\sum_{j\in\Z}\Omega_jc_j^n
<\infty.
\end{equation}
Therefore, summing up \eqref{weight c} and  \eqref{weight c1} gives \eqref{eq:common-envelope-bound}.
\end{proof}

\subsubsection{Frequency-envelope weighted estimates for the variable-coefficient Lam\'e equation}\label{sec:lame}
\begin{proposition}
\label{prop:weighted-lame}
Assume that $d\ge 2, 1<p<2d$ and $\Omega=(\Omega_j)_{j\in\Z} \in AF(\delta_0)$. Let $\underline b$  and $\overline b$ be two positive constants, 
\(b(t,x)\) be real-valued function satisfying
\begin{equation}\label{eq:beta-assumptions}
0<\underline b\le b(t,x)\le\overline b,
\qquad
b-1\in \tdL^\infty_T(\dB^{\frac{d}{p}}_{p,1})
\cap \tdL^\infty_T(\dB^{\frac{d}{p}}_{p,1}(\Omega)).
\end{equation}
For \(m\in \mathbb{Z}\), define
\begin{align*}
&b_m:=1+\dS_m(b-1),
\qquad
b^h:=(\Id-\dS_m)(b-1),\quad\alpha=\underline b\min\{\mu,\lambda+2\mu\},\\ &H_m:=\norm{\nabla b^h}_{L^\infty_T(\dB^{-1+\frac{d}{p}}_{p,1})},
K_m(t):=\norm{\dS_m\nabla b(t)}_{\dB^{\frac{d}{p}}_{p,1}},\quad H_m^\Omega(t):=\norm{\nabla b^h(t)}_{\dB^{-1+\frac{d}{p}}_{p,1}(\Omega)}.  
\end{align*}
Moreover, for $(t,x)\in[0,T]\times\R^d$ assume that
\begin{equation}\label{eq:bm-uniform-ellipticity}
b_m(t,x)\ge\frac{\underline b}{2}
\end{equation}
and  $v\in\tdL^\infty_T(\dB^{-1+\frac{d}{p}}_{p,1})\cap  L^1_T(\dB^{1+\frac{d}{p}}_{p,1})$
is a  solution of
\begin{equation}\label{eq:linear-lame}
\partial_tv-b\cA v=f,
\qquad
v|_{t=0}=v_0
\end{equation}
with $v_0\in\dB^{-1+\frac{d}{p}}_{p,1}(\Omega)$ and $ f\in L^1_T(\dB^{-1+\frac{d}{p}}_{p,1}(\Omega))$. There exists $\eta_0>0$ depending  on $d,p,\delta_0,\mu$ and $\lambda$ such that if $H_m\le \eta_0\alpha$, then
\[
v\in
\tdL^\infty_T(\dB^{-1+\frac{d}{p}}_{p,1}(\Omega))
\cap
L^1_T(\dB^{1+\frac{d}{p}}_{p,1}(\Omega)),
\]
and
\begin{equation}\label{eq:weighted-lame-estimate}
\begin{aligned}
&\norm v_{\tdL^\infty_T(\dB^{-1+\frac{d}{p}}_{p,1}(\Omega))}
+\alpha
\norm v_{L^1_T(\dB^{1+\frac{d}{p}}_{p,1}(\Omega))}
\\
&\quad\le C_0 \exp\left(
\frac {C_1}{\alpha}\int_0^T K_m(t)^2\dd t
\right)
\Biggl(
\norm{v_0}_{\dB^{-1+\frac{d}{p}}_{p,1}(\Omega)}
+\norm f_{L^1_T(\dB^{-1+\frac{d}{p}}_{p,1}(\Omega))}
\\
&\hspace{43mm}
+\int_0^T
H_m^\Omega(t)
\norm{v(t)}_{\dB^{1+\frac{d}{p}}_{p,1}}
\dd t
\Biggr),
\end{aligned}
\end{equation}
where $C_0, C_1>0$ depends only on \(d,p,\delta_0,\mu,\lambda\) and the fixed Littlewood–Paley decomposition.
\end{proposition}
\begin{proof}
For \(L\ge1\), define the bounded truncated weight
\begin{equation}\label{eq:truncated-envelope-lame}
\Theta_j^{(L)}
:=
\min\{\Omega_j,L\},\,\, \qquad\Theta^{(L)}:=
(\Theta_j^{(L)})_{j\in\Z}.
\end{equation}
For simplicity, we write $\Theta=\Theta^{(L)}$ and $\Theta_j=\Theta^{(L)}_j$. By the definition, since $\Omega\in AF(\delta_0)$, the sequence $\Theta \in AF(\delta_0)$, with admissibility constants independent of $L$ and $1\le\Theta_j\le L$. For every \(\sigma\in\R\) and \(1\le q\le\infty\), one easily infer 
\(
\norm{z}_{\tdL^q_T
(\dB^\sigma_{p,1}(\Theta))}
\le
L\norm{z}_{\tdL^q_T
(\dB^\sigma_{p,1})}
\), which shall be used in the following.

\noindent Let \(d_v:=R(D)v\) and \(w_v:=C(D)v\), where
\[
R(D)v := \Lambda^{-1}\Div v,
\qquad
\bigl(C(D)v\bigr)_{ij}
:=
\Lambda^{-1}(\partial_i v_j-\partial_jv_i),
\quad 1\le i,j\le d,
\]
which follows that
\[
d_v:[0,T]\times\R^d\to\R,
\qquad
(w_v)_{ii}=0,
\qquad
(w_v)_{ji}=-(w_v)_{ij},
\quad 1\le i,j\le d.
\]
Applying \(R(D)\) and \(C(D)\) to
\eqref{eq:linear-lame} gives
\begin{equation}\label{eq:divergence-form}
\partial_td_v-(\lambda+2\mu)\Div(b_m\nabla d_v)
=
R(D)f+F_d,
\end{equation}
and
\begin{equation}\label{eq:curl-form}
\partial_tw_v-\mu\Div(b_m\nabla w_v)
=
C(D)f+F_w,
\end{equation}
where
\begin{equation}\label{eq:divergence-remainder}
\begin{aligned}
F_d={}&
(\lambda+2\mu)b^h\Delta d_v
-(\lambda+2\mu)\nabla b_m\cdot\nabla d_v
+\mu[R(D),b_m]\Delta v
+\mu[R(D),b^h]\Delta v
\\
&+(\lambda+\mu)[R(D),b_m]\nabla\Div v
+(\lambda+\mu)[R(D),b^h]\nabla\Div v
\end{aligned}
\end{equation}
and
\begin{equation}\label{eq:curl-remainder}
\begin{aligned}
F_w={}&
\mu b^h\Delta w_v
-\mu \nabla b_m\cdot\nabla w_v
+\mu[C(D),b_m]\Delta v
+\mu[C(D),b^h]\Delta v
\\
&-(\lambda+\mu)C(D)
\bigl((\nabla b_m)\Div v\bigr)
-(\lambda+\mu)C(D)
\bigl((\nabla b^h)\Div v\bigr).
\end{aligned}
\end{equation}
Define
\[
h_m(t):=
\norm{\nabla b^h(t)}_{\dB^{-1+\frac{d}{p}}_{p,1}},
\qquad
h_m^\Theta(t):=
\norm{\nabla b^h(t)}_{\dB^{-1+\frac{d}{p}}_{p,1}(\Theta)}.
\]
By \eqref{eq:high-principal},
\begin{equation}\label{eq:high-principal-bound}
\begin{aligned}
&\norm{b^h\Delta d_v}_{\dB^{-1+\frac{d}{p}}_{p,1}(\Theta)}
+
\norm{b^h\Delta w_v}_{\dB^{-1+\frac{d}{p}}_{p,1}(\Theta)}\lesssim h_m(t)\norm v_{\dB^{1+\frac{d}{p}}_{p,1}(\Theta)}.
\end{aligned}
\end{equation}
Since
$
K_m(t)
:=\norm{\dS_m\nabla b(t)}_{\dB^{\frac{d}{p}}_{p,1}}=\norm{\nabla b_m(t)}_{\dB^{\frac{d}{p}}_{p,1}}
$,
hence \eqref{eq:weighted-one-sided} gives
\begin{equation}\label{eq:low-gradient-bound}
\begin{aligned}
&\norm{\nabla b_m\cdot\nabla d_v}_{\dB^{-1+\frac{d}{p}}_{p,1}(\Theta)}
+
\norm{\nabla b_m\cdot\nabla w_v}_{\dB^{-1+\frac{d}{p}}_{p,1}(\Theta)}
\lesssim K_m(t)\norm v_{\dB^{\frac{d}{p}}_{p,1}(\Theta)}.
\end{aligned}
\end{equation}
By \eqref{eq:mult-low},
\begin{equation}\label{eq:low-commutator-bound}
\begin{aligned}
&\norm{[R(D),b_m]\Delta v}_{\dB^{-1+\frac{d}{p}}_{p,1}(\Theta)}
+
\norm{[R(D),b_m]\nabla\Div v}_{\dB^{-1+\frac{d}{p}}_{p,1}(\Theta)}
\\
&\qquad+
\norm{[C(D),b_m]\Delta v}_{\dB^{-1+\frac{d}{p}}_{p,1}(\Theta)}
\lesssim K_m(t)\norm v_{\dB^{\frac{d}{p}}_{p,1}(\Theta)}.
\end{aligned}
\end{equation}
By \eqref{eq:mult-high},
\begin{equation}\label{eq:high-commutator-bound}
\begin{aligned}
&\norm{[R(D),b^h]\Delta v}_{\dB^{-1+\frac{d}{p}}_{p,1}(\Theta)}
+
\norm{[R(D),b^h]\nabla\Div v}_{\dB^{-1+\frac{d}{p}}_{p,1}(\Theta)}
\\
&\qquad+
\norm{[C(D),b^h]\Delta v}_{\dB^{-1+\frac{d}{p}}_{p,1}(\Theta)}
\lesssim h_m(t)\norm v_{\dB^{1+\frac{d}{p}}_{p,1}(\Theta)}
+
h_m^\Theta(t)\norm v_{\dB^{1+\frac{d}{p}}_{p,1}}.
\end{aligned}
\end{equation}
Since \(C(D)\) is a homogeneous Fourier multiplier of
order zero, the low-frequency curl term satisfies
\begin{equation}\label{eq:low-curl-bound}
\begin{aligned}
&\norm{
C(D)
\bigl((\nabla b_m)\Div v\bigr)}
_{\dB^{-1+\frac{d}{p}}_{p,1}(\Theta)}
\lesssim K_m(t)\norm v_{\dB^{\frac{d}{p}}_{p,1}(\Theta)}.
\end{aligned}
\end{equation}
For the high-frequency curl term, \eqref{eq:high-gradient} gives
\begin{equation}\label{eq:high-curl-bound}
\begin{aligned}
&\norm{
C(D)
\bigl((\nabla b^h)\Div v\bigr)}
_{\dB^{-1+\frac{d}{p}}_{p,1}(\Theta)}
\lesssim
h_m(t)\norm v_{\dB^{1+\frac{d}{p}}_{p,1}(\Theta)}
+
h_m^\Theta(t)\norm v_{\dB^{1+\frac{d}{p}}_{p,1}}.
\end{aligned}
\end{equation}
Combining
\eqref{eq:high-principal-bound}--\eqref{eq:high-curl-bound}, we obtain
\begin{equation}\label{eq:combined-remainder-bound}
\begin{aligned}
&\norm{F_d(t)}_{\dB^{-1+\frac{d}{p}}_{p,1}(\Theta)}
+
\norm{F_w(t)}_{\dB^{-1+\frac{d}{p}}_{p,1}(\Theta)}
\\
&\lesssim
K_m(t)\norm{v(t)}_{\dB^{\frac{d}{p}}_{p,1}(\Theta)}
+
h_m(t)\norm{v(t)}_{\dB^{1+\frac{d}{p}}_{p,1}(\Theta)}
+
h_m^\Theta(t)\norm{v(t)}_{\dB^{1+\frac{d}{p}}_{p,1}}.
\end{aligned}
\end{equation}
Let \(z\) be any scalar component of \(d_v\) or \(w_v\). Then \(z\)
satisfies
\begin{equation}\label{eq:scalar-model}
\partial_tz-\nu\Div(b_m\nabla z)=g+F,
\qquad
\nu\in\{\lambda+2\mu ,\mu \},
\end{equation}
applying \(\dDelta_j\) to which yields
\begin{equation}\label{eq:localized-scalar-model}
\partial_tz_j-\nu\Div(b_m\nabla z_j)
=
g_j+F_j
+\nu\Div\bigl([\dDelta_j,b_m]\nabla z\bigr),
\end{equation}
where \(z_j:=\dDelta_jz\).
For \(z_j=\dDelta_jz\), one has
\[
\operatorname{supp}\widehat z_j
\subset
\left\{
\xi\in\R^d:
r_*2^j\le|\xi|\le R_*2^j
\right\},
\qquad
0<r_*<R_*,
\]
where \(r_*\) and \(R_*\) depend only on the fixed
Littlewood--Paley decomposition. Moreover,
\[
b_m(t,\cdot)\in C_b(\R^d),
\qquad
\inf_{x\in\R^d}b_m(t,x)
\ge\frac{\underline b}{2}.
\]
Hence \cite[Lemma~2.4]{{CMZ2010R}} gives
\begin{equation}\label{eq:localized-coercivity}
\begin{aligned}
-\nu\int_{\R^d}
\Div(b_m\nabla z_j)
|z_j|^{p-2}z_j\,\dd x
\ge
c_p\nu\underline b\,2^{2j}
\|z_j\|_{L^p}^p
\ge
c_p\alpha\,2^{2j}\|z_j\|_{L^p}^p,
\end{aligned}
\end{equation}
where \(c_p>0\) depends only on $d$, \, $p$ and the fixed
Littlewood--Paley decomposition.
\noindent After applying the standard regularization to \(\norm{z_j}_{L^p}\), we obtain
\begin{equation}\label{eq:dyadic-parabolic}
\begin{aligned}
\frac{d}{dt}\norm{z_j}_{L^p}
+c_p\alpha2^{2j}\norm{z_j}_{L^p}
\lesssim &
\norm{g_j}_{L^p}
+\norm{F_j}_{L^p}
+
\norm{
\Div\bigl([\dDelta_j,b_m]\nabla z\bigr)}_{L^p}.
\end{aligned}
\end{equation}
Multiplying \eqref{eq:dyadic-parabolic} by \(\Theta_j2^{j{(-1+\frac{d}{p})}}\), integrating over \([0,t]\) and
summing over \(j\in\Z\) give
\begin{equation}\label{eq:integrated-parabolic}
\begin{aligned}
&\norm z_{\tdL^\infty_t(\dB^{-1+\frac{d}{p}}_{p,1}(\Theta))}
+c_p\alpha\norm z_{L^1_t(\dB^{1+\frac{d}{p}}_{p,1}(\Theta))}
\\
&\lesssim
\norm{z_0}_{\dB^{-1+\frac{d}{p}}_{p,1}(\Theta)}
+\norm g_{L^1_t(\dB^{-1+\frac{d}{p}}_{p,1}(\Theta))}
+\norm F_{L^1_t(\dB^{-1+\frac{d}{p}}_{p,1}(\Theta))}
\\
&\quad+
\sum_j\Theta_j2^{j(-1+\frac{d}{p})}
\int_0^t
\norm{
\Div\bigl([\dDelta_j,b_m]\nabla z\bigr)}_{L^p}
\dd\tau.
\end{aligned}
\end{equation}
By \eqref{eq:comm-low},
\begin{equation}\label{eq:localized-commutator-bound}
\sum_j\Theta_j2^{j(-1+\frac{d}{p})}
\norm{
\Div\bigl([\dDelta_j,b_m]\nabla z\bigr)}_{L^p}
\lesssim K_m(t)\norm z_{\dB^{\frac{d}{p}}_{p,1}(\Theta)}.
\end{equation}
Combining \eqref{eq:integrated-parabolic},
\eqref{eq:localized-commutator-bound}, and
\eqref{eq:combined-remainder-bound} for \(d_v\) and \(w_v\), there exist constants
$c_0$ and $C_0>0$ depend only on
\(d,p,\delta_0,\mu,\lambda\), but independent of \(L\) and \(m\) such that
\begin{equation}\label{eq:lame-before-absorption}
\begin{aligned}
&\norm v_{\tdL^\infty_t
(\dB^{-1+\frac dp}_{p,1}(\Theta))}
+c_0\alpha
\norm v_{L^1_t
(\dB^{1+\frac dp}_{p,1}(\Theta))}\\
&\le
C_0\Biggl(
\norm{v_0}_{\dB^{-1+\frac dp}_{p,1}(\Theta)}
+
\norm f_{L^1_t
(\dB^{-1+\frac dp}_{p,1}(\Theta))}
+
\int_0^t
K_m(\tau)
\norm{v(\tau)}_{\dB^{\frac dp}_{p,1}(\Theta)}
\dd\tau
\\
&\qquad+
\int_0^t
h_m(\tau)
\norm{v(\tau)}_{\dB^{1+\frac dp}_{p,1}(\Theta)}
\dd\tau+
\int_0^t
h_m^\Theta(\tau)
\norm{v(\tau)}_{\dB^{1+\frac dp}_{p,1}}
\dd\tau
\Biggr).
\end{aligned}
\end{equation}
Choose $\eta_0:=\frac{c_0}{4C_0}.$ Since $h_m(t)\le H_m\le\eta_0\alpha,$
we have
\begin{equation}\label{eq:high-tail-absorption}
C_0\int_0^t
h_m(\tau)
\norm{v(\tau)}_{\dB^{1+\frac dp}_{p,1}(\Theta)}
\dd\tau
\le
\frac{c_0\alpha}{4}
\norm v_{L^1_t
(\dB^{1+\frac dp}_{p,1}(\Theta))}.
\end{equation}
With aid of $\norm v_{\dB^{\frac dp}_{p,1}(\Theta)}\lesssim\norm v_{\dB^{-1+\frac dp}_{p,1}(\Theta)}^{1/2}\norm v_{\dB^{1+\frac dp}_{p,1}(\Theta)}^{1/2}
$ and Young's inequality, there exists a constant \(C_1>0\), independent
of \(L\) and \(m\), such that
\begin{equation}\label{eq:weighted-young}
\begin{aligned}
C_0K_m(t)
\norm{v(t)}_{\dB^{\frac dp}_{p,1}(\Theta)}
\le
\frac{c_0\alpha}{4}
\norm{v(t)}_{\dB^{1+\frac dp}_{p,1}(\Theta)}
+
\frac{C_1}{\alpha}K_m(t)^2
\norm{v(t)}_{\dB^{-1+\frac dp}_{p,1}(\Theta)}.
\end{aligned}
\end{equation}
Substituting \eqref{eq:high-tail-absorption} and
\eqref{eq:weighted-young} into
\eqref{eq:lame-before-absorption}, we obtain
\begin{equation}\label{eq:lame-gronwall-form}
\begin{aligned}
&\norm v_{\tdL^\infty_t
(\dB^{-1+\frac dp}_{p,1}(\Theta))}
+
\frac{c_0\alpha}{2}
\norm v_{L^1_t
(\dB^{1+\frac dp}_{p,1}(\Theta))}\\
&\le
C_0\Biggl(
\norm{v_0}_{\dB^{-1+\frac dp}_{p,1}(\Theta)}
+
\norm f_{L^1_t
(\dB^{-1+\frac dp}_{p,1}(\Theta))}
+
\int_0^t
h_m^\Theta(\tau)
\norm{v(\tau)}_{\dB^{1+\frac dp}_{p,1}}
\dd\tau
\Biggr)
\\
&\quad+
\frac{C_1}{\alpha}
\int_0^t
K_m(\tau)^2
\norm{v(\tau)}_{\dB^{-1+\frac dp}_{p,1}(\Theta)}
\dd\tau.
\end{aligned}
\end{equation}
Therefore,  after enlarging \(C_0\) if necessary, Gronwall's inequality yields
\begin{equation}\label{eq:lame-final-bound}
\begin{aligned}
&\norm v_{\tdL^\infty_t
(\dB^{-1+\frac dp}_{p,1}(\Theta))}
+
\alpha
\norm v_{L^1_t
(\dB^{1+\frac dp}_{p,1}(\Theta))}
\\
&\le C_0
\exp\left(
\frac{C_1}{\alpha}
\int_0^tK_m(\tau)^2\dd\tau
\right)
\Biggl(
\norm{v_0}_{\dB^{-1+\frac dp}_{p,1}(\Theta)}
+
\norm f_{L^1_t
(\dB^{-1+\frac dp}_{p,1}(\Theta))}
\\
&\quad+
\int_0^t
h_m^\Theta(\tau)
\norm{v(\tau)}_{\dB^{1+\frac dp}_{p,1}}
\dd\tau
\Biggr).
\end{aligned}
\end{equation}
We now remove the smoothness assumption by a Friedrichs
approximation. For \(L\ge1\), set
\[
J_n:=\dS_n-\dS_{-n},
\qquad
v_n:=J_nv.
\]
 Applying \(J_n\) to
\eqref{eq:linear-lame}, we obtain 
\begin{equation}\label{eq:friedrichs-projected-lame}
\partial_tv_n-b\cA v_n
=
J_nf+[J_n,b]\cA v,
\qquad
v_n|_{t=0}=J_nv_0.
\end{equation}
For \(M\ge1\), set
\[
b^{(M)}
:=
1+\sum_{|k|\le M}\dDelta_k(b-1),
\qquad
r^{(M)}:=b-b^{(M)}.
\]
Since \(\Theta_j\le\Omega_j\) and $b-1\in\tdL^\infty_T(\dB^{\frac dp}_{p,1})\cap\tdL^\infty_T(\dB^{\frac dp}_{p,1}(\Omega))$, we thus have
\begin{equation}\label{eq:coefficient-double-approximation}
\begin{aligned}
\norm{r^{(M)}}_{\tdL^\infty_T
(\dB^{\frac dp}_{p,1})}
+
\norm{r^{(M)}}_{\tdL^\infty_T
(\dB^{\frac dp}_{p,1}(\Theta))}
\le
\norm{r^{(M)}}_{\tdL^\infty_T
(\dB^{\frac dp}_{p,1})}
+
\norm{r^{(M)}}_{\tdL^\infty_T
(\dB^{\frac dp}_{p,1}(\Omega))}
\rightarrow0
\end{aligned}
\end{equation}
as \(M\to\infty\), uniformly in \(L\).

\noindent For fixed \(M\), \eqref{eq:weighted-one-sided} gives
\[
\begin{aligned}
\norm{b^{(M)}\cA v}_{L^1_T
(\dB^{-1+\frac dp}_{p,1}(\Theta))}
\lesssim
\left(
1+
\norm{b^{(M)}-1}_{\tdL^\infty_T
(\dB^{\frac dp}_{p,1})}
\right)
\norm{ \cA v}_{L^1_T
(\dB^{-1+\frac dp}_{p,1}(\Theta))}.
\end{aligned}
\]
In view of 
\begin{equation}\label{eq:fixed-M-commutator-identity}
[J_n,b^{(M)}]\cA v
=
(J_n-\Id)(b^{(M)}\cA v)
-
b^{(M)}(J_n-\Id)\cA v,
\end{equation}
it follows that
\[
\begin{aligned}
&\norm{[J_n,b^{(M)}]\cA v}_{L^1_T
(\dB^{-1+\frac dp}_{p,1}(\Theta))}\\
&\lesssim
\norm{(J_n-\Id)(b^{(M)}\cA v)}_{L^1_T
(\dB^{-1+\frac dp}_{p,1}(\Theta))}
+
\norm{b^{(M)}(J_n-\Id)\cA v}_{L^1_T
(\dB^{-1+\frac dp}_{p,1}(\Theta))}
\\
&\lesssim
\norm{(J_n-\Id)(b^{(M)}\cA v)}_{L^1_T
(\dB^{-1+\frac dp}_{p,1}(\Theta))}+
\left(
1+
\norm{b^{(M)}-1}_{\tdL^\infty_T
(\dB^{\frac dp}_{p,1})}
\right)\\
&\quad\cdot
\norm{(J_n-\Id)\cA v}_{L^1_T
(\dB^{-1+\frac dp}_{p,1}(\Theta))},
\end{aligned}
\]
from which, we obtain
\begin{equation}\label{eq:fixed-M-commutator}
\lim\limits_{n\to\infty}\norm{[J_n,b^{(M)}]\cA v}_{L^1_T
(\dB^{-1+\frac dp}_{p,1}(\Theta))}=0.
\end{equation}
On the other hand, the uniform boundedness of \(J_n\) in Besov spaces and \eqref{eq:weighted-mixed} give
\begin{equation}\label{eq:remainder-M-commutator}
\begin{aligned}
&\norm{[J_n,r^{(M)}]\cA v}_{L^1_T
(\dB^{-1+\frac dp}_{p,1}(\Theta))}\\
&\lesssim
\norm{r^{(M)}}_{\tdL^\infty_T
(\dB^{\frac dp}_{p,1})}
\norm v_{L^1_T
(\dB^{1+\frac dp}_{p,1}(\Theta))}
+
\norm{r^{(M)}}_{\tdL^\infty_T
(\dB^{\frac dp}_{p,1}(\Theta))}
\norm v_{L^1_T
(\dB^{1+\frac dp}_{p,1})}.
\end{aligned}
\end{equation}

\noindent For the fixed \(L\), let \(\varepsilon>0\). By
\eqref{eq:coefficient-double-approximation} and
\eqref{eq:remainder-M-commutator}, we may choose \(M\) sufficiently
large so that
\begin{equation}\label{small one}
\begin{aligned}
\norm{[J_n,r^{(M)}]\cA v}_{L^1_T
(\dB^{-1+\frac dp}_{p,1}(\Theta))}
\le\varepsilon    \qquad
\text{uniformly in  } n.
\end{aligned}
\end{equation}
For this fixed \(M\),
\eqref{eq:fixed-M-commutator} gives an integer \(n_0\) such that
\begin{equation}\label{small two}
\begin{aligned}
\norm{[J_n,b^{(M)}]\cA v}_{L^1_T
(\dB^{-1+\frac dp}_{p,1}(\Theta))}
\le\varepsilon
\qquad
\text{for all }n\ge n_0.
\end{aligned}
\end{equation}
Since $[J_n,b]\cA v=[J_n,b^{(M)}]\cA v+[J_n,r^{(M)}]\cA v$, from \eqref{small one} and
\eqref{small two}, we obtain
\[
\norm{[J_n,b]\cA v}_{L^1_T
(\dB^{-1+\frac dp}_{p,1}(\Theta))}
\le2\varepsilon
\qquad
\text{for all }n\ge n_0,
\]
from which and \(\varepsilon>0\) is arbitrary, we obtain
\[
\lim\limits_{n\to\infty}
\norm{[J_n,b]\cA v}_{L^1_T
(\dB^{-1+\frac dp}_{p,1}(\Theta))}
=0.
\]
Applying \eqref{eq:lame-final-bound} to \(v_n\) and \eqref{eq:friedrichs-projected-lame}, we obtain
\begin{equation}\label{eq:friedrichs-weighted-estimate}
\begin{aligned}
&\norm {v_n}_{\tdL^\infty_T
(\dB^{-1+\frac dp}_{p,1}(\Theta))}
+\alpha
\norm {v_n}_{L^1_T
(\dB^{1+\frac dp}_{p,1}(\Theta))}\\
&\le
C_0\exp\left(
\frac{C_1}{\alpha}
\int_0^T K_m(t)^2\dd t
\right)
\Biggl(
\norm{J_nv_0}_{\dB^{-1+\frac dp}_{p,1}(\Theta)}\\
& \ \  \ \ +\norm{J_nf+[J_n,b]\cA v}_{L^1_T
(\dB^{-1+\frac dp}_{p,1}(\Theta))}
+\int_0^T
h_m^\Theta(t)
\norm{v_n(t)}_{\dB^{1+\frac dp}_{p,1}}
\dd t
\Biggr).
\end{aligned}
\end{equation}
For fixed \(L\), since \(\Theta_j\le L\),
\[
\begin{aligned}
&\norm{J_nv-v}_{\tdL^\infty_T
(\dB^{-1+\frac dp}_{p,1}(\Theta))}
+
\norm{J_nv-v}_{L^1_T
(\dB^{1+\frac dp}_{p,1}(\Theta))}
\\
&\le
L\left(
\norm{J_nv-v}_{\tdL^\infty_T
(\dB^{-1+\frac dp}_{p,1})}
+
\norm{J_nv-v}_{L^1_T
(\dB^{1+\frac dp}_{p,1})}
\right)
\longrightarrow0
\qquad(n\to\infty),
\end{aligned}
\]
where the convergence follows from  \(J_n\to\Id\) as $n\to\infty$. Similarly, we have
\(
\lim\limits_{n\to\infty}\norm{J_nv_0-v_0}_{\dB^{-1+\frac dp}_{p,1}(\Theta)}
=0
\)
and
\(
\lim\limits_{n\to\infty}\norm{J_nf-f}_{L^1_T
(\dB^{-1+\frac dp}_{p,1}(\Theta))}
=0
\). Moreover, since
\(
\norm{h_m^\Theta}_{L^\infty_T}
\le
\norm{\nabla b^h}_{\tdL^\infty_T
(\dB^{-1+\frac dp}_{p,1}(\Theta))}
\), consequently, we have
\[
\begin{aligned}
\int_0^T
h_m^\Theta(t)
\norm{J_nv(t)-v(t)}_{\dB^{1+\frac dp}_{p,1}}
\dd t
\le
\norm{h_m^\Theta}_{L^\infty_T}
\norm{J_nv-v}_{L^1_T
(\dB^{1+\frac dp}_{p,1})}
\rightarrow0 \qquad(n\to\infty).
\end{aligned}
\]
Passing to the limit \(n\to\infty\) in
\eqref{eq:friedrichs-weighted-estimate}, using the convergences
established above for \(v_n\), \(J_nv_0\), \(J_nf\), and
\([J_n,b]\cA v\), we obtain
\begin{equation}\label{eq:lame-truncated-envelope-bound}
\begin{aligned}
&\norm v_{\tdL^\infty_T
(\dB^{-1+\frac dp}_{p,1}(\Theta^{(L)}))}
+\alpha
\norm v_{L^1_T
(\dB^{1+\frac dp}_{p,1}(\Theta^{(L)}))}\\
&\le
C_0
\exp\left(
\frac{C_1}{\alpha}
\int_0^T K_m(t)^2\dd t
\right)
\Biggl(
\norm{v_0}_{\dB^{-1+\frac dp}_{p,1}(\Theta^{(L)})}\\
&\quad+\norm f_{L^1_T
(\dB^{-1+\frac dp}_{p,1}(\Theta^{(L)}))}
+\int_0^T
h_m^{\Theta^{(L)}}(t)
\norm{v(t)}_{\dB^{1+\frac dp}_{p,1}}
\dd t
\Biggr),
\end{aligned}
\end{equation}
where the constants $C_0$ and $C_1$ are independent of \(L\).

\noindent By the definition of \(\Theta^{(L)}\), for every \(j\in\Z\),
\[
\Theta_j^{(L)}
=
\min\{\Omega_j,L\}
\rightarrow
\Omega_j
\qquad(L\to\infty),
\]
and the convergence is monotone increasing.
\noindent Hence the monotone convergence theorem applied to the nonnegative
dyadic sums gives
\[
\begin{aligned}
\lim\limits_{L\to\infty}\norm v_{\tdL^\infty_T
(\dB^{-1+\frac dp}_{p,1}(\Theta^{(L)}))}
=
\norm v_{\tdL^\infty_T
(\dB^{-1+\frac dp}_{p,1}(\Omega))},
\\
\lim\limits_{L\to\infty}\norm v_{L^1_T
(\dB^{1+\frac dp}_{p,1}(\Theta^{(L)}))}
=
\norm v_{L^1_T
(\dB^{1+\frac dp}_{p,1}(\Omega))},
\end{aligned}
\]
\[
\begin{aligned}
\lim\limits_{L\to\infty}\norm{v_0}_{\dB^{-1+\frac dp}_{p,1}(\Theta^{(L)})}
=
\norm{v_0}_{\dB^{-1+\frac dp}_{p,1}(\Omega)},
\\
\lim\limits_{L\to\infty}\norm f_{L^1_T
(\dB^{-1+\frac dp}_{p,1}(\Theta^{(L)}))}
=
\norm f_{L^1_T
(\dB^{-1+\frac dp}_{p,1}(\Omega))}.
\end{aligned}
\]
 Moreover, for almost every \(t\in[0,T]\),
\[
\lim\limits_{L\to\infty}h_m^{\Theta^{(L)}}(t)
=
H_m^\Omega(t).
\]
The monotone convergence theorem gives
\[
\begin{aligned}
\lim\limits_{L\to\infty}\int_0^T
h_m^{\Theta^{(L)}}(t)
\norm{v(t)}_{\dB^{1+\frac dp}_{p,1}}
\dd t
=
\int_0^T
H_m^\Omega(t)
\norm{v(t)}_{\dB^{1+\frac dp}_{p,1}}
\dd t.
\end{aligned}
\]
From above estimates and letting \(L\to\infty\) in
\eqref{eq:lame-truncated-envelope-bound}, we obtain \eqref{eq:weighted-lame-estimate}. This completes
the proof.
\end{proof}

\subsubsection{Frequency-envelope weighted estimates for uniform solutions}
\label{sec:envelope-persistence}
\begin{proposition}[Uniform boundedness]
\label{prop:envelope-persistence}
Let \(d\ge2\), \(1<p<2d\), and
\(\Omega\in AF(\delta_0)\). Let \((a,u)\) be a solution of
\eqref{eq:1.2} on \([0,T]\), where \(0<T\le T_0\le1\), satisfying
\eqref{eq:uniform-local-bounds}. Assume that
\eqref{eq:coefficient-splitting}--\eqref{eq:coefficient-tail-small}
hold with the common cutoff \(m=m(T)\) furnished by
\cref{cor:common-m}. If
\[
(a_0,u_0)
\in
\dot B^{\frac dp}_{p,1}(\Omega)
\times
\dot B^{-1+\frac dp}_{p,1}(\Omega),
\]
then
\begin{equation}\label{eq:envelope-persistence}
\begin{aligned}
\norm{a}_{\tdL^\infty_T
(\dB^{\frac dp}_{p,1}(\Omega))}
+
\norm{u}_{\tdL^\infty_T
(\dB^{-1+\frac dp}_{p,1}(\Omega))}
+
\norm{u}_{L^1_T
(\dB^{1+\frac dp}_{p,1}(\Omega))}
\lesssim
\norm{a_0}_{\dB^{\frac dp}_{p,1}(\Omega)}
+
\norm{u_0}_{\dB^{-1+\frac dp}_{p,1}(\Omega)}
,
\end{aligned}
\end{equation}
 where the implicit constant is independent of  \(T\), \(m\).
\end{proposition}

\begin{proof}
Let \(C_{\mathrm{der}}>0\) be a constant such that \(
\|\nabla f\|_{\dot B^{-1+\frac dp}_{p,1}}
\le
C_{\mathrm{der}}
\|f\|_{\dot B^{\frac dp}_{p,1}}
\). We choose the parameter \(\eta_b\) in
\cref{cor:common-m} such that
\begin{equation}\label{eq:envelope-eta-b-choice}
C_{\mathrm{der}}\eta_b
\le
\eta_0\alpha,
\end{equation}
where $\alpha:=\underline b\min\{\mu,\lambda+2\mu\}$ and  \(\eta_0\) is the constant in
\cref{prop:weighted-lame}. 

 Fix \(L\ge1\) and write
\(\Theta=\Theta^{(L)}\) (see \eqref{eq:truncated-envelope-lame}). Since \(1\le\Theta_j\le L\), all the weighted norms below are finite
for fixed \(L\) by \eqref{eq:uniform-local-bounds}. By \(\dB^{\frac dp}_{p,1}\hookrightarrow L^\infty\), 
\cref{lem:CL-envelope-composition} and \eqref{eq:uniform-local-bounds}, we have
\begin{equation}\label{eq:composition-envelope-persistence}
\begin{aligned}
\norm{b-1}_{\tdL^\infty_t
(\dB^{\frac dp}_{p,1}(\Theta))}
+
\norm{G(a)-G(0)}_{\tdL^\infty_t
(\dB^{\frac dp}_{p,1}(\Theta))}
\lesssim
\norm{a}_{\tdL^\infty_t
(\dB^{\frac dp}_{p,1}(\Theta))}
\end{aligned}
\end{equation}
and
\begin{equation}\label{eq:Hm-envelope-persistence}
H_m^\Theta(t)
=
\norm{\nabla b^h(t)}
_{\dB^{-1+\frac dp}_{p,1}(\Theta)}
\lesssim
\norm{a}_{\tdL^\infty_t
(\dB^{\frac dp}_{p,1}(\Theta))}.
\end{equation}
By \eqref{eq:coefficient-tail-small} and
\eqref{eq:envelope-eta-b-choice},
\[
\begin{aligned}
H_m
=
\|\nabla b^h\|_{L^\infty_T
(\dot B^{-1+\frac dp}_{p,1})}
\le
C_{\mathrm{der}}
\|b^h\|_{\widetilde L^\infty_T
(\dot B^{\frac dp}_{p,1})}
\le
C_{\mathrm{der}}\eta_b
\le
\eta_0\alpha.
\end{aligned}
\]
Together with $b_m\ge\frac{\underline b}{2},$
which follows from \cref{cor:common-m}, this verifies the coefficient
assumptions of \cref{prop:weighted-lame}.

\noindent Bernstein's inequality and
\eqref{eq:uniform-local-bounds} imply
\[
K_m(t)
=
\norm{\dS_m\nabla b(t)}_{\dB^{\frac dp}_{p,1}}\lesssim 2^m\norm{b(t)-1}_{
\dB^{\frac dp}_{p,1}}
\lesssim2^m,
\]
from which, by \eqref{eq:common-m-time-bound}, we have
\begin{equation}\label{eq:lame-exponential-envelope-persistence}
\frac{C_1}{\alpha}
\int_0^T K_m(t)^2\,\dd t
\le
\frac{CT2^{2m}}{\alpha}
\le C,
\end{equation}
where \(C\) is independent of \(T\), \(m\).

\noindent We apply \cref{prop:weighted-lame} to the velocity equation
\[
\partial_tu-b\cA u
=
-u\cdot\nabla u-\nabla G(a),
\qquad
u|_{t=0}=u_0.
\]
By \eqref{eq:weighted-one-sided} and
\eqref{eq:composition-envelope-persistence},
\[
\begin{aligned}
\norm{u\cdot\nabla u}
_{\dB^{-1+\frac dp}_{p,1}(\Theta)}
&\lesssim
\norm{u}_{\dB^{1+\frac dp}_{p,1}}
\norm{u}_{\dB^{-1+\frac dp}_{p,1}(\Theta)},
\\
\norm{\nabla G(a)}
_{\dB^{-1+\frac dp}_{p,1}(\Theta)}
&\lesssim
\norm{a}_{\tdL^\infty_t
(\dB^{\frac dp}_{p,1}(\Theta))}.
\end{aligned}
\]
Inserting these estimates, together with
\eqref{eq:Hm-envelope-persistence} and
\eqref{eq:lame-exponential-envelope-persistence}, into
\eqref{eq:weighted-lame-estimate}, we obtain
\begin{equation}\label{eq:velocity-envelope-persistence}
\begin{aligned}
\norm{u}_{\tdL^\infty_t
(\dB^{-1+\frac dp}_{p,1}(\Theta))}
&+
\norm{u}_{L^1_t
(\dB^{1+\frac dp}_{p,1}(\Theta))}\lesssim
\norm{u_0}_{\dB^{-1+\frac dp}_{p,1}(\Theta)}
\\
&+
\int_0^t
\left(
1+\norm{u(\tau)}_{\dB^{1+\frac dp}_{p,1}}
\right)
\Bigg(
\norm{a}_{\tdL^\infty_\tau
(\dB^{\frac dp}_{p,1}(\Theta))}
\\
&
+
\norm{u}_{\tdL^\infty_\tau
(\dB^{-1+\frac dp}_{p,1}(\Theta))}
+
\norm{u}_{L^1_\tau
(\dB^{1+\frac dp}_{p,1}(\Theta))}
\Bigg)\,\dd\tau .
\end{aligned}
\end{equation}
We next apply \cref{lem:static-weighted-transport} to the density equation
\[
\partial_ta+u\cdot\nabla a
=
-(1+a)\operatorname{div}u,
\qquad
a|_{t=0}=a_0.
\]
The index condition
\eqref{eq:static-transport-index-range} follows from
\(0<\delta_0<1\). By \eqref{eq:weighted-endpoint},
\[
\begin{aligned}
\norm{(1+a)\operatorname{div}u}
_{\dB^{\frac dp}_{p,1}(\Theta)}\lesssim
\left(
1+\norm{a}_{\dB^{\frac dp}_{p,1}}
\right)
\norm{u}_{\dB^{1+\frac dp}_{p,1}(\Theta)}
+
\norm{a}_{\dB^{\frac dp}_{p,1}(\Theta)}
\norm{u}_{\dB^{1+\frac dp}_{p,1}}.
\end{aligned}
\]
Moreover, \eqref{eq:uniform-local-bounds} yields
\[
\int_0^T
\norm{\nabla u(t)}
_{\dB^{\frac dp}_{p,\infty}\cap L^\infty}
\,\dd t
\lesssim
\norm{u}_{L^1_T
(\dB^{1+\frac dp}_{p,1})}
\le M.
\]
It follows from \eqref{eq:static-weighted-transport} that
\begin{equation}\label{eq:density-envelope-persistence}
\begin{aligned}
\norm{a}_{\tdL^\infty_t
(\dB^{\frac dp}_{p,1}(\Theta))}
\lesssim
\norm{a_0}_{\dB^{\frac dp}_{p,1}(\Theta)}
+
\norm{u}_{L^1_t
(\dB^{1+\frac dp}_{p,1}(\Theta))}
+
\int_0^t
\norm{u(\tau)}_{\dB^{1+\frac dp}_{p,1}}
\norm{a}_{\tdL^\infty_\tau
(\dB^{\frac dp}_{p,1}(\Theta))}
\,\dd\tau .
\end{aligned}
\end{equation}
Combining
\eqref{eq:velocity-envelope-persistence} and
\eqref{eq:density-envelope-persistence}, we obtain
\begin{equation}\label{eq:combined-envelope-persistence}
\begin{aligned}
&\norm{a}_{\tdL^\infty_t
(\dB^{\frac dp}_{p,1}(\Theta))}
+
\norm{u}_{\tdL^\infty_t
(\dB^{-1+\frac dp}_{p,1}(\Theta))}
+
\norm{u}_{L^1_t
(\dB^{1+\frac dp}_{p,1}(\Theta))}\lesssim
\norm{a_0}_{\dB^{\frac dp}_{p,1}(\Theta)}
+
\norm{u_0}_{\dB^{-1+\frac dp}_{p,1}(\Theta)}
\\
&+
\int_0^t
\left(
1+
\norm{u(\tau)}_{\dB^{1+\frac dp}_{p,1}}
\right)
\Bigg(
\norm{a}_{\tdL^\infty_\tau
(\dB^{\frac dp}_{p,1}(\Theta))}
+
\norm{u}_{\tdL^\infty_\tau
(\dB^{-1+\frac dp}_{p,1}(\Theta))}
+
\norm{u}_{L^1_\tau
(\dB^{1+\frac dp}_{p,1}(\Theta))}
\Bigg)\,\dd\tau,
\end{aligned}
\end{equation}
from which, by Gronwall's inequality and \eqref{eq:uniform-local-bounds}, we have
\begin{equation}\label{eq:truncated-envelope-persistence}
\begin{aligned}
\norm{a}_{\tdL^\infty_T
(\dB^{\frac dp}_{p,1}(\Theta))}
+
\norm{u}_{\tdL^\infty_T
(\dB^{-1+\frac dp}_{p,1}(\Theta))}
+
\norm{u}_{L^1_T
(\dB^{1+\frac dp}_{p,1}(\Theta))}
\lesssim 
\norm{a_0}_{\dB^{\frac dp}_{p,1}(\Theta)}
+
\norm{u_0}_{\dB^{-1+\frac dp}_{p,1}(\Theta)},
\end{aligned}
\end{equation}
  where the implicit constant is independent of \(L\), \(T\), \(m\). Since
\[
\Theta_j^{(L)}
=
\min\{\Omega_j,L\}
\rightarrow
\Omega_j
\qquad(L\to\infty)
\]
monotonically, letting \(L\to\infty\) in
\eqref{eq:truncated-envelope-persistence} and using the monotone
convergence theorem yields \eqref{eq:envelope-persistence}.
\end{proof}

\section{Continuous dependence}\label{sec:continuity}

Let \((a_0^n, u_0^n) \to (a_0, u_0)\) in \(X_p\), and denote the corresponding solutions by \((a^n, u^n)\) and \((a, u)\). Upon possibly shrinking the common time \(T\), these solutions satisfy the conclusions of  \cref{lem:uniform-local}, \cref{cor:common-m} and \cref{lem:common-envelope}.
\subsection{Small-tail estimates at high frequencies}
\begin{proposition}\label{cor:uniform-tails}
Let \(r>0\), \(0<T\le T_0\), and
\(\mathscr S_{r,T}\) be as in \cref{cor:common-m} for this choice of
\(\eta_b\). Let \(\Omega\in AF(\delta_0)\) satisfy $
\Omega_j\longrightarrow\infty
\quad(\text{as }j\to\infty),
$ and let \(M_\Omega>0\). Set
\begin{equation}\label{eq:uniform-initial-envelope-bound}
\mathcal S
:=
\left\{
(a,u)\in\mathscr S_{r,T}:
\norm{a_0}_{\dB^{\frac dp}_{p,1}(\Omega)}
+
\norm{u_0}_{\dB^{-1+\frac dp}_{p,1}(\Omega)}
\le M_\Omega
\right\}.
\end{equation}
Then there exists a nonincreasing sequence
\((\varepsilon_N)_{N\in\Z}\), independent of the particular solution,
such that
\[
\varepsilon_N\rightarrow0
\qquad\text{as }N\to\infty,
\]
and
\begin{equation}\label{eq:uniform-tails}
\begin{aligned}
\sup_{(a,u)\in\mathcal S}\Bigl(
&
\norm{(\Id-\dS_N)a}
_{\tdL^\infty_T(\dB^{\frac dp}_{p,1})}
+
\norm{(\Id-\dS_N)u}
_{\tdL^\infty_T(\dB^{-1+\frac dp}_{p,1})}
\\
&+
\norm{(\Id-\dS_N)u}
_{L^1_T(\dB^{1+\frac dp}_{p,1})}
+
\norm{(\Id-\dS_N)(b(a)-1)}
_{\tdL^\infty_T(\dB^{\frac dp}_{p,1})}
\\
&+
\norm{(\Id-\dS_N)(G(a)-G(0))}
_{\tdL^\infty_T(\dB^{\frac dp}_{p,1})}
\Bigr)
\le\varepsilon_N .
\end{aligned}
\end{equation}
\end{proposition}

\begin{proof}
Since \(\mathcal S\subset\mathscr S_{r,T}\), every
\((a,u)\in\mathcal S\) satisfies
\eqref{eq:uniform-local-bounds} and the common coefficient splitting
\eqref{eq:coefficient-splitting}--%
\eqref{eq:coefficient-tail-small}, with the same cutoff \(m=m(T)\).
Hence all the assumptions of
\cref{prop:envelope-persistence} are satisfied uniformly over
\(\mathcal S\).

\noindent It follows from
\cref{prop:envelope-persistence} and
\eqref{eq:uniform-initial-envelope-bound} that
\begin{equation}\label{eq:uniform-weighted-solution-bound}
\begin{aligned}
\sup_{(a,u)\in\mathcal S}\Bigl(
\norm{a}_{\tdL^\infty_T
(\dB^{\frac dp}_{p,1}(\Omega))}
+
\norm{u}_{\tdL^\infty_T
(\dB^{-1+\frac dp}_{p,1}(\Omega))}+
\norm{u}_{L^1_T
(\dB^{1+\frac dp}_{p,1}(\Omega))}
\Bigr)
\lesssim M_{\Omega}.
\end{aligned}
\end{equation}

\noindent By \cref{lem:CL-envelope-composition},
\eqref{eq:uniform-local-bounds}, and
\eqref{eq:uniform-weighted-solution-bound}, we have
\begin{equation}\label{eq:uniform-weighted-composition-bound}
\begin{aligned}
\sup_{(a,u)\in\mathcal S}\Bigl(
\norm{b(a)-1}_{\tdL^\infty_T
(\dB^{\frac dp}_{p,1}(\Omega))}
+
\norm{G(a)-G(0)}_{\tdL^\infty_T
(\dB^{\frac dp}_{p,1}(\Omega))}
\Bigr)
\lesssim M_{\Omega}.
\end{aligned}
\end{equation}
Define \(Q_N:=\Id-\dS_N\). Then there is an integer \(N_0\ge1\) with the property that for every \(j,k,N\in\mathbb Z\),
\[
\dDelta_jQ_N\dDelta_k=0
\]
unless both \(|j-k|\le N_0\) and \(k\ge N-N_0\). As a consequence, for each \(\sigma\in\mathbb R\) and \(1\le q\le\infty\), we have
\begin{equation}\label{tail-from-envelope-Linfty}
\begin{aligned}
\norm{Q_Nz}_{\tdL^q_T(\dB^\sigma_{p,1})}
&\lesssim
\sum_{j\in\Z}
\sum_{\substack{|k-j|\le N_0\\ k\ge N-N_0}}
2^{j\sigma}
\norm{\dDelta_kz}_{L^q_T(L^p)}
\\
&\lesssim
\frac{1}{\Omega_{N-N_0}}
\sum_{k\ge N-N_0}
\Omega_k2^{k\sigma}
\norm{\dDelta_kz}_{L^q_T(L^p)}\lesssim
\frac{1}{\Omega_{N-N_0}}
\norm{z}_{\tdL^q_T
(\dB^\sigma_{p,1}(\Omega))}.
\end{aligned}
\end{equation}
Taking \(q=\infty\) and \(q=1\), respectively, in
\eqref{tail-from-envelope-Linfty}, we obtain
\begin{equation}\label{eq:tail-from-envelope-Linfty}
\norm{Q_Nz}_{\tdL^\infty_T(\dB^\sigma_{p,1})}
\lesssim
\frac{1}{\Omega_{N-N_0}}
\norm{z}_{\tdL^\infty_T
(\dB^\sigma_{p,1}(\Omega))}
\end{equation}
and
\begin{equation}\label{eq:tail-from-envelope-Lone}
\norm{Q_Nz}_{L^1_T(\dB^\sigma_{p,1})}
\lesssim
\frac{1}{\Omega_{N-N_0}}
\norm{z}_{L^1_T
(\dB^\sigma_{p,1}(\Omega))}.
\end{equation}
Combining
\eqref{eq:uniform-weighted-solution-bound},
\eqref{eq:uniform-weighted-composition-bound},
\eqref{eq:tail-from-envelope-Linfty}, and
\eqref{eq:tail-from-envelope-Lone}, we obtain a constant \(C>0\), independent of \(N\) and of the
particular solution, such that
\[
\begin{aligned}
\sup_{(a,u)\in\mathcal S}\Bigl(
&
\norm{Q_Na}_{\tdL^\infty_T(\dB^{\frac dp}_{p,1})}
+
\norm{Q_Nu}_{\tdL^\infty_T
(\dB^{-1+\frac dp}_{p,1})}
+
\norm{Q_Nu}_{L^1_T
(\dB^{1+\frac dp}_{p,1})}
\\
&+
\norm{Q_N(b(a)-1)}
_{\tdL^\infty_T(\dB^{\frac dp}_{p,1})}
+
\norm{Q_N(G(a)-G(0))}
_{\tdL^\infty_T(\dB^{\frac dp}_{p,1})}
\Bigr)
\le
\frac{C}{\Omega_{N-N_0}}.
\end{aligned}
\]
Set \(\varepsilon_N := C/ \Omega_{N-N_0}\). As \(\Omega\) is nondecreasing, \((\varepsilon_N)_{N \in \mathbb{Z}}\) is nonincreasing. Combined with the fact that \(\lim_{N \to \infty} \Omega_{N-N_0} = \infty\), this yields \(\varepsilon_N \to 0\) as \(N \to \infty\), thereby establishing \eqref{eq:uniform-tails}.
\end{proof}

\subsection{The velocity difference estimate in \(L^1_T(L^\infty)\)}

Let \((a_i,u_i)\), \(i=1,2\), be two solutions issued from the common
initial-data neighborhood. Let \(X_i\) be the flow associated with
\(u_i\), and set
\[
\bar a_i:=a_i\circ X_i,
\qquad
\bar u_i:=u_i\circ X_i.
\]
By \cref{theo 2.2}, we have
\begin{equation}\label{eq:lagrangian-Lipschitz}
\begin{aligned}
\norm{\bar a_2-\bar a_1}
_{L^\infty_T(\dB^{\frac dp}_{p,1})}
+
\norm{\bar u_2-\bar u_1}
_{L^\infty_T(\dB^{-1+\frac dp}_{p,1})}+
\norm{\bar u_2-\bar u_1}
_{L^1_T(\dB^{1+\frac dp}_{p,1})}\le
C_T
\norm{
(a_{2,0}-a_{1,0},u_{2,0}-u_{1,0})
}_{X_p}.
\end{aligned}
\end{equation}

\begin{proposition}\label{lem:L1Linfty}
Under the above assumptions, there exists a constant \(C_T>0\),
independent of the two solutions, such that
\begin{equation}\label{eq:L1Linfty}
\norm{u_2-u_1}_{L^1_T(L^\infty)}
\le
C_T
\norm{(a_{2,0}-a_{1,0},u_{2,0}-u_{1,0})}_{X_p}.
\end{equation}
\end{proposition}

\begin{proof}
Following the ideas of \cite[Page 15]{GuoYangZhang2026}, the velocity difference estimate \eqref{eq:L1Linfty} can be obtained, we present its proof here for completeness.
The standard interpolation inequality yields
\[
\begin{aligned}
\norm{\bar u_2-\bar u_1}
_{L^2_T(\dB^{\frac dp}_{p,1})}\lesssim
\norm{\bar u_2-\bar u_1}
_{L^\infty_T(\dB^{-1+\frac dp}_{p,1})}^{\frac12}
\norm{\bar u_2-\bar u_1}
_{L^1_T(\dB^{1+\frac dp}_{p,1})}^{\frac12}.
\end{aligned}
\]
Consequently, \eqref{eq:lagrangian-Lipschitz} and 
\(\dB^{\frac dp}_{p,1}\hookrightarrow L^\infty\) yield
\begin{equation}\label{eq:bar-L1Linfty}
\begin{aligned}
\norm{\bar u_2-\bar u_1}_{L^1_T(L^\infty)}
&\lesssim
T^{\frac12}
\norm{\bar u_2-\bar u_1}
_{L^2_T(\dB^{\frac dp}_{p,1})}
\\
&\le
C_T
\norm{(a_{2,0}-a_{1,0},u_{2,0}-u_{1,0})}_{X_p}.
\end{aligned}
\end{equation}
Because
\[
X_i(t,y)
=
y+\int_0^t\bar u_i(\tau,y)\dd\tau,
\qquad i=1,2,
\]
from which, we  have
\begin{equation}\label{eq:exact-lagrangian-flow-difference}
X_2(t,y)-X_1(t,y)
=
\int_0^t
(\bar u_2-\bar u_1)(\tau,y)\dd\tau
\end{equation}
and 
\begin{equation}\label{eq:flow-difference-from-lagrangian-difference}
\norm{X_2-X_1}_{L^\infty_T(L^\infty)}
\le
\norm{\bar u_2-\bar u_1}_{L^1_T(L^\infty)}.
\end{equation}
For every \(t\in[0,T]\), the map \(X_2(t,\cdot)\) is a
bi-Lipschitz diffeomorphism of \(\R^d\). Therefore,
\[
\begin{aligned}
\norm{u_2(t)-u_1(t)}_{L^\infty}
&=
\norm{
u_2(t)\circ X_2(t)
-
u_1(t)\circ X_2(t)
}_{L^\infty}
\\
&\le
\norm{\bar u_2(t)-\bar u_1(t)}_{L^\infty}
\\
&\quad+
\norm{
u_1(t)\circ X_1(t)
-
u_1(t)\circ X_2(t)
}_{L^\infty}
\\
&\le
\norm{\bar u_2(t)-\bar u_1(t)}_{L^\infty}
+
\norm{\nabla u_1(t)}_{L^\infty}
\norm{X_2(t)-X_1(t)}_{L^\infty},
\end{aligned}
\]
from which and using
\eqref{eq:flow-difference-from-lagrangian-difference}, we obtain
\[
\begin{aligned}
\norm{u_2-u_1}_{L^1_T(L^\infty)}
&\le
\left(
1+
\norm{\nabla u_1}_{L^1_T(L^\infty)}
\right)
\norm{\bar u_2-\bar u_1}_{L^1_T(L^\infty)}
\\
&\lesssim
\left(
1+
\norm{u_1}_{L^1_T(\dB^{1+\frac dp}_{p,1})}
\right)
\norm{\bar u_2-\bar u_1}_{L^1_T(L^\infty)}.
\end{aligned}
\]
Consequently, \eqref{eq:uniform-local-bounds} and
\eqref{eq:bar-L1Linfty} give
\[
\begin{aligned}
\norm{u_2-u_1}_{L^1_T(L^\infty)}
&\lesssim
(1+
\norm{u_1}_{L^1_T(\dB^{1+\frac dp}_{p,1})})
\norm{\bar u_2-\bar u_1}_{L^1_T(L^\infty)}
\\
&\le
C_T
\norm{
(a_{2,0}-a_{1,0},u_{2,0}-u_{1,0})
}_{X_p}.
\end{aligned}
\]
This proves \eqref{eq:L1Linfty}.
\end{proof}
\subsection{Low-frequency stability}
\begin{proposition}\label{prop:low-frequency-stability}
Let \(\eta_b>0\) be chosen sufficiently small in
\cref{cor:common-m}, and let
\(\mathcal S\) and \((\varepsilon_N)_{N\in\Z}\) be as in
\cref{cor:uniform-tails} for this choice of \(\eta_b\).
For \((a_i,u_i)\in\mathcal S\), \(i=1,2\), set
\[
\delta a:=a_2-a_1,
\qquad
\delta u:=u_2-u_1,
\]
and
\[
\delta a_0:=a_{2,0}-a_{1,0},
\qquad
\delta u_0:=u_{2,0}-u_{1,0}.
\]
Then, for every \(N\in\Z\),
\begin{equation}\label{eq:low-frequency-stability}
\begin{aligned}
\norm{\dS_N\delta a}
_{\tdL^\infty_T(\dB^{\frac dp}_{p,1})}
&+
\norm{\dS_N\delta u}
_{\tdL^\infty_T(\dB^{-1+\frac dp}_{p,1})}
+
\norm{\dS_N\delta u}
_{L^1_T(\dB^{1+\frac dp}_{p,1})}
\\
&\le
C_N\left(
\norm{\delta a_0}_{\dB^{\frac dp}_{p,1}}
+
\norm{\delta u_0}_{\dB^{-1+\frac dp}_{p,1}}
+
\norm{\delta u}_{L^1_T(L^\infty)}
\right)
+
C\varepsilon_N.
\end{aligned}
\end{equation}
Here \(C\) is independent of \(N\) and of the particular pair of
solutions, whereas \(C_N\) may additionally depend on \(N\).
\end{proposition}
\begin{proof}
Set
\[
b_i:=(1+a_i)^{-1},
\qquad
b_{i,m}:=1+\dS_m(b_i-1),
\qquad
b_i^h:=(\Id-\dS_m)(b_i-1).
\]
From \eqref{eq:1.2} the difference satisfies
\begin{equation}\label{eq:difference-system}
\left\{
\begin{aligned}
\partial_t\delta a+u_2\cdot\nabla\delta a
={}&
-\delta u\cdot\nabla a_1
-(1+a_2)\Div\delta u
-\delta a\Div u_1,
\\
\partial_t\delta u-b_2\cA\delta u
={}&
-\delta u\cdot\nabla u_2
-u_1\cdot\nabla\delta u
+(b_2-b_1)\cA u_1
\\
&-
\nabla\bigl(G(a_2)-G(a_1)\bigr).
\end{aligned}
\right.
\end{equation}
For \(0\le t\le T\), denote
\[
A_N(t)
:=
\norm{\dS_N\delta a}
_{\tdL^\infty_t(\dB^{\frac dp}_{p,1})}
\]
and
\[
U_N(t)
:=
\norm{\dS_N\delta u}
_{\tdL^\infty_t(\dB^{-1+\frac dp}_{p,1})}
+
\norm{\dS_N\delta u}
_{L^1_t(\dB^{1+\frac dp}_{p,1})}. 
\]
By \eqref{eq:uniform-tails},
\begin{equation}\label{eq:difference-tail}
\begin{aligned}
\norm{(\Id-\dS_N)\delta a}
_{\tdL^\infty_T(\dB^{\frac dp}_{p,1})}
+
\norm{(\Id-\dS_N)\delta u}
_{\tdL^\infty_T(\dB^{-1+\frac dp}_{p,1})}
+
\norm{(\Id-\dS_N)\delta u}
_{L^1_T(\dB^{1+\frac dp}_{p,1})}
\le C\varepsilon_N.
\end{aligned}
\end{equation}
Applying \(\dS_N\) to the density equation in
\eqref{eq:difference-system}, we obtain
\[
\begin{aligned}
\partial_t\dS_N\delta a
+u_2\cdot\nabla\dS_N\delta a
={}&
-\dS_N(\delta u\cdot\nabla a_1)
-\dS_N((1+a_2)\Div\delta u)
\\
&-
\dS_N(\delta a\Div u_1)
+
[u_2\cdot\nabla,\dS_N]\delta a.
\end{aligned}
\]
Moreover, \eqref{eq:uniform-local-bounds} gives
\[
\int_0^t
\norm{\nabla u_2(\tau)}_{\dB^{\frac dp}_{p,1}}
\dd\tau
\lesssim
\norm{u_2}_{L^1_t(\dB^{1+\frac dp}_{p,1})}
\le M.
\]
Hence, by \cref{lem:transport-estimate}, we obtain
\begin{equation}\label{eq:low-density-pre}
\begin{aligned}
A_N(t)
\lesssim{}&
\norm{a_{2,0}-a_{1,0}}_{\dB^{\frac dp}_{p,1}}
+
\norm{\dS_N(\delta u\cdot\nabla a_1)}
_{L^1_t(\dB^{\frac dp}_{p,1})}
\\
&+
\norm{\dS_N((1+a_2)\Div\delta u)}
_{L^1_t(\dB^{\frac dp}_{p,1})}
+
\norm{\dS_N(\delta a\Div u_1)}
_{L^1_t(\dB^{\frac dp}_{p,1})}
\\
&+
\norm{[u_2\cdot\nabla,\dS_N]\delta a}
_{L^1_t(\dB^{\frac dp}_{p,1})}.
\end{aligned}
\end{equation}
Bony's decomposition gives
\[
\delta u\cdot\nabla a_1
=
\dot T_{\delta u}\nabla a_1
+
\dot T_{\nabla a_1}\delta u
+
\dot R(\delta u,\nabla a_1).
\]
By Bernstein's inequality and \cref{u}，
\begin{equation}\label{bony1}
\begin{aligned}
\norm{\dS_N\dot T_{\delta u}\nabla a_1}
_{L^1_t(\dB^{\frac dp}_{p,1})}
&\lesssim
2^N
\norm{\dot T_{\delta u}\nabla a_1}
_{L^1_t(\dB^{-1+\frac dp}_{p,1})}
\lesssim
C_N
\norm{\delta u}_{L^1_t(L^\infty)}
\norm{a_1}_{\tdL^\infty_t(\dB^{\frac dp}_{p,1})}.
\end{aligned}
\end{equation}
By \cref{u} and
\(\dB^{-1+\frac dp}_{p,1}\hookrightarrow
\dB^{-1}_{\infty,1}\),
\[
\norm{\dot T_{\nabla a_1}\delta u}
_{\dB^{\frac dp}_{p,1}}
\lesssim
\norm{a_1}_{\dB^{\frac dp}_{p,1}}
\norm{\delta u}_{\dB^{1+\frac dp}_{p,1}}.
\]
If \(2\le p<2d\), then \cref{v} gives
\[
\begin{aligned}
\norm{\dot R(\delta u,\nabla a_1)}
_{\dB^{\frac dp}_{p,1}}
\lesssim
\norm{\dot R(\delta u,\nabla a_1)}
_{\dB^{\frac{2d}{p}}_{\frac p2,1}}
\lesssim
\norm{\delta u}_{\dB^{1+\frac dp}_{p,2}}
\norm{\nabla a_1}_{\dB^{-1+\frac dp}_{p,2}},
\end{aligned}
\]
and if \(1<p<2\), then
\[
\begin{aligned}
\norm{\dot R(\delta u,\nabla a_1)}
_{\dB^{\frac dp}_{p,1}}
&\lesssim
\norm{\dot R(\delta u,\nabla a_1)}
_{\dB^d_{1,1}}
\lesssim
\norm{\delta u}_{\dB^{1+\frac d{p'}}_{p',2}}
\norm{\nabla a_1}_{\dB^{-1+\frac dp}_{p,2}}
\end{aligned}
\]
from which, 
\begin{equation}\label{bony2}
\begin{aligned}
\norm{
\dot T_{\nabla a_1}\delta u
+
\dot R(\delta u,\nabla a_1)}
_{L^1_t(\dB^{\frac dp}_{p,1})}
\lesssim
\norm{a_1}_{\tdL^\infty_t(\dB^{\frac dp}_{p,1})}
\norm{\delta u}_{L^1_t(\dB^{1+\frac dp}_{p,1})}.
\end{aligned}
\end{equation}
Hence, since 
\((a_i,u_i)\in\mathcal S\), \(i=1,2\), with aid of \eqref{bony1} and \eqref{bony2}, we have
\begin{equation}\label{eq:density-convection-bound}
\begin{aligned}
\norm{\dS_N(\delta u\cdot\nabla a_1)}
_{L^1_t(\dB^{\frac dp}_{p,1})}
\le
C_N\norm{\delta u}_{L^1_t(L^\infty)}
+
CU_N(t)
+
C\varepsilon_N.
\end{aligned}
\end{equation}
By \cref{lem:basic-products},
\begin{equation}\label{eq:density-divergence-bound}
\begin{aligned}
\norm{\dS_N((1+a_2)\Div\delta u)}
_{L^1_t(\dB^{\frac dp}_{p,1})}
&\lesssim
\left(
1+
\norm{a_2}_{\tdL^\infty_t(\dB^{\frac dp}_{p,1})}
\right)
\norm{\delta u}_{L^1_t(\dB^{1+\frac dp}_{p,1})}
\\
&\le
CU_N(t)+C\varepsilon_N,
\end{aligned}
\end{equation}
and
\begin{equation}\label{eq:density-product-bound}
\begin{aligned}
\norm{\dS_N(\delta a\Div u_1)}
_{L^1_t(\dB^{\frac dp}_{p,1})}
&\lesssim
\int_0^t
\norm{\delta a(\tau)}_{\dB^{\frac dp}_{p,1}}
\norm{u_1(\tau)}_{\dB^{1+\frac dp}_{p,1}}
\dd\tau
\\
&\le
C\int_0^t
\norm{u_1(\tau)}_{\dB^{1+\frac dp}_{p,1}}
A_N(\tau)\dd\tau
+
C\varepsilon_N.
\end{aligned}
\end{equation}
Moreover, \cref{lem:cutoff-transport-commutator} yields
\begin{equation}\label{eq:density-commutator-bound}
\begin{aligned}
\norm{[u_2\cdot\nabla,\dS_N]\delta a}
_{L^1_t(\dB^{\frac dp}_{p,1})}
&\lesssim
\int_0^t
\norm{u_2(\tau)}_{\dB^{1+\frac dp}_{p,1}}
\norm{\delta a(\tau)}_{\dB^{\frac dp}_{p,1}}
\dd\tau
\\
&\le
C\int_0^t
\norm{u_2(\tau)}_{\dB^{1+\frac dp}_{p,1}}
A_N(\tau)\dd\tau
+
C\varepsilon_N.
\end{aligned}
\end{equation}
Combining
\eqref{eq:low-density-pre}--%
\eqref{eq:density-commutator-bound}, we obtain
\begin{equation}\label{eq:low-density-estimate}
\begin{aligned}
A_N(t)\le{}&
C\norm{a_{2,0}-a_{1,0}}_{\dB^{\frac dp}_{p,1}}
+
C_N\norm{\delta u}_{L^1_t(L^\infty)}
+
CU_N(t)
+
C\varepsilon_N
\\
&+
C\int_0^t
\left(
\norm{u_1(\tau)}_{\dB^{1+\frac dp}_{p,1}}
+
\norm{u_2(\tau)}_{\dB^{1+\frac dp}_{p,1}}
\right)
A_N(\tau)\dd\tau.
\end{aligned}
\end{equation}
Applying \(\dS_N\) to the velocity equation in
\eqref{eq:difference-system}, we have
\begin{equation}\label{eq:truncated-velocity}
\left\{
\begin{aligned}
\partial_t\dS_N\delta u-b_2\cA \dS_N\delta u
={}&
-\dS_N(\delta u\cdot\nabla u_2)
-\dS_N(u_1\cdot\nabla\delta u)
\\
&+
\dS_N((b_2-b_1)\cA u_1)
-
\dS_N\nabla\bigl(G(a_2)-G(a_1)\bigr)
\\
&+
[\dS_N,b_2]\cA\delta u,
\\
\dS_N\delta u|_{t=0}
={}&
\dS_N(u_{2,0}-u_{1,0}).
\end{aligned}
\right.
\end{equation}
For the common cutoff \(m=m(T)\), set
\[
\alpha_*
:=
\frac{\underline b}{2}\min\{\mu,\lambda+2\mu\},
\qquad
K_m(t)
:=
\norm{\dS_m\nabla b_2(t)}_{\dB^{\frac dp}_{p,1}}.
\]
Fix \(\eta>0\) sufficiently small. By taking \(\eta_b>0\)
sufficiently small and using \cref{cor:common-m}, we have
\begin{equation}\label{eq:low-coefficient-smallness}
b_{2,m}\ge\frac{\underline b}{2},
\qquad
\norm{\nabla b_2^h}
_{L^\infty_T(\dB^{-1+\frac dp}_{p,1})}
\le\eta_0\alpha_*,
\qquad
\norm{b_2^h}
_{\tdL^\infty_T(\dB^{\frac dp}_{p,1})}
\le\eta\alpha_*.
\end{equation}
The composition estimate and
\eqref{eq:uniform-local-bounds} give
$
\norm{b_2-1}_{L^\infty_T(\dB^{\frac dp}_{p,1})}
\le C，
$
and \eqref{eq:common-m-time-bound} implies
$
T2^{2m}<\frac{4}{c_*}$. Consequently,
\begin{equation}\label{eq:Km-uniform-bound}
\begin{aligned}
\int_0^T K_m(t)^2\dd t\lesssim
T2^{2m}
\norm{b_2-1}_{L^\infty_T(\dB^{\frac dp}_{p,1})}^2
\le
C.
\end{aligned}
\end{equation}
Applying the unweighted case of
\cref{prop:weighted-lame} to
\eqref{eq:truncated-velocity}, using \eqref{eq:low-coefficient-smallness} and \eqref{eq:Km-uniform-bound}, we obtain
\begin{equation}\label{eq:velocity-pre}
\begin{aligned}
&\norm{\dS_N\delta u}
_{\tdL^\infty_t(\dB^{-1+\frac dp}_{p,1})}
+
\alpha_*
\norm{\dS_N\delta u}
_{L^1_t(\dB^{1+\frac dp}_{p,1})}
\\
&\lesssim 
\norm{\delta u_0}_{\dB^{-1+\frac dp}_{p,1}}
+
\norm{\dS_N(\delta u\cdot\nabla u_2)}
_{L^1_t(\dB^{-1+\frac dp}_{p,1})}
\\
&\quad+
\norm{\dS_N(u_1\cdot\nabla\delta u)}
_{L^1_t(\dB^{-1+\frac dp}_{p,1})}
+
\norm{\dS_N((b_2-b_1)\cA u_1)}
_{L^1_t(\dB^{-1+\frac dp}_{p,1})}
\\
&\quad+
\norm{\dS_N\nabla(G(a_2)-G(a_1))}
_{L^1_t(\dB^{-1+\frac dp}_{p,1})}
+
\norm{[\dS_N,b_2]\cA\delta u}
_{L^1_t(\dB^{-1+\frac dp}_{p,1})}.
\end{aligned}
\end{equation}
By \cref{lem:cutoff-commutator} and $\delta u=\dS_N\delta u+(\Id-\dS_N)\delta u$, we have
\begin{equation}\label{sb1}
\begin{aligned}
\norm{[\dS_N,b_2]\cA\delta u}
_{L^1_t(\dB^{-1+\frac dp}_{p,1})}
\lesssim{}&
\int_0^t
K_m(\tau)
\norm{\dS_N\delta u(\tau)}_{\dB^{\frac dp}_{p,1}}
\dd\tau
\\
&+
\int_0^t
K_m(\tau)
\norm{(\Id-\dS_N)\delta u(\tau)}
_{\dB^{\frac dp}_{p,1}}
\dd\tau
\\
&+
\norm{b_2^h}_{\tdL^\infty_t
(\dB^{\frac dp}_{p,1})}
\norm{\dS_N\delta u}_{L^1_t
(\dB^{1+\frac dp}_{p,1})}
\\
&+
\norm{b_2^h}_{\tdL^\infty_t
(\dB^{\frac dp}_{p,1})}
\norm{(\Id-\dS_N)\delta u}_{L^1_t
(\dB^{1+\frac dp}_{p,1})}.
\end{aligned}
\end{equation}
The interpolation inequality
\begin{equation}\label{eq:unweighted-critical-interpolation}
\norm{z}_{\dB^{\frac dp}_{p,1}}
\lesssim
\norm{z}_{\dB^{-1+\frac dp}_{p,1}}^{\frac12}
\norm{z}_{\dB^{1+\frac dp}_{p,1}}^{\frac12}.
\end{equation}
By
\eqref{eq:unweighted-critical-interpolation} and Young's inequality,
\begin{equation}\label{eq:low-commutator-young}
\begin{aligned}
K_m(\tau)
\norm{\dS_N\delta u(\tau)}_{\dB^{\frac dp}_{p,1}}
\le
\eta\alpha_*
\norm{\dS_N\delta u(\tau)}
_{\dB^{1+\frac dp}_{p,1}}
+
\frac{C}{\alpha_*}
K_m(\tau)^2
\norm{\dS_N\delta u(\tau)}
_{\dB^{-1+\frac dp}_{p,1}}.
\end{aligned}
\end{equation}
Moreover, \eqref{eq:unweighted-critical-interpolation} and
\eqref{eq:difference-tail} yield the following
\begin{equation}
\begin{aligned}\label{eq:high-tail-interpolation}
\norm{(\Id-\dS_N)\delta u}
_{L^2_t(\dB^{\frac dp}_{p,1})}^2
\lesssim
\norm{(\Id-\dS_N)\delta u}
_{\tdL^\infty_t(\dB^{-1+\frac dp}_{p,1})}
\norm{(\Id-\dS_N)\delta u}
_{L^1_t(\dB^{1+\frac dp}_{p,1})}
\lesssim
\varepsilon_N^2.
\end{aligned}
\end{equation}
Hence, by \eqref{eq:Km-uniform-bound}, we have
\begin{equation}\label{eqq2}
\begin{aligned}
\int_0^t
K_m(\tau)
\norm{(\Id-\dS_N)\delta u(\tau)}
_{\dB^{\frac dp}_{p,1}}
\dd\tau
\le
\norm{K_m}_{L^2_t}
\norm{(\Id-\dS_N)\delta u}
_{L^2_t(\dB^{\frac dp}_{p,1})}
\le
C\varepsilon_N.
\end{aligned}
\end{equation}
Hence, by \eqref{eq:low-coefficient-smallness} and
\eqref{eq:difference-tail},
\begin{equation}\label{eqq3}
\begin{aligned}
\norm{b_2^h}_{\tdL^\infty_t
(\dB^{\frac dp}_{p,1})}
\Bigl(
\norm{\dS_N\delta u}_{L^1_t
(\dB^{1+\frac dp}_{p,1})}
+
\norm{(\Id-\dS_N)\delta u}_{L^1_t
(\dB^{1+\frac dp}_{p,1})}
\Bigr)
\le
\eta\alpha_*
\norm{\dS_N\delta u}_{L^1_t
(\dB^{1+\frac dp}_{p,1})}
+
C\varepsilon_N.
\end{aligned}
\end{equation}
Integrating \eqref{eq:low-commutator-young} over \((0,t)\) and
using \eqref{eqq2} and \eqref{eqq3}, we deduce from
\eqref{sb1} that
\begin{equation}\label{eq:commutator-low-result}
\begin{aligned}
&\norm{[\dS_N,b_2]\cA\delta u}
_{L^1_t(\dB^{-1+\frac dp}_{p,1})}
\\
&\le
2\eta\alpha_*
\norm{\dS_N\delta u}_{L^1_t
(\dB^{1+\frac dp}_{p,1})}
+
\frac{C}{\alpha_*}
\int_0^t
K_m(\tau)^2
\norm{\dS_N\delta u(\tau)}
_{\dB^{-1+\frac dp}_{p,1}}
\dd\tau
+
C\varepsilon_N.
\end{aligned}
\end{equation}
By \cref{lem:basic-products} and \eqref{eq:difference-tail}, we have
\begin{equation}\label{eq:velocity-first-convection}
\begin{aligned}
\norm{\dS_N(\delta u\cdot\nabla u_2)}
_{L^1_t(\dB^{-1+\frac dp}_{p,1})}
&\lesssim
\int_0^t
\norm{u_2(\tau)}_{\dB^{1+\frac dp}_{p,1}}
\norm{\delta u(\tau)}_{\dB^{-1+\frac dp}_{p,1}}
\dd\tau
\\
&\le
C\int_0^t
\norm{u_2(\tau)}_{\dB^{1+\frac dp}_{p,1}}
\norm{\dS_N\delta u(\tau)}_{\dB^{-1+\frac dp}_{p,1}}
\dd\tau
+
C\varepsilon_N.
\end{aligned}
\end{equation}
By \eqref{eq:unweighted-critical-interpolation},  \eqref{eq:uniform-local-bounds} and Young's
inequality, we obtain
\begin{equation}\label{eqq4}
\begin{aligned}
&\norm{u_1\cdot\nabla \dS_N\delta u}_{\dB^{-1+\frac dp}_{p,1}}
\lesssim
\norm{u_1}_{\dB^{\frac dp}_{p,1}}
\norm{\dS_N\delta u}_{\dB^{\frac dp}_{p,1}}\\
&\lesssim
\norm{u_1}_{\dB^{-1+\frac dp}_{p,1}}^{\frac12}
\norm{u_1}_{\dB^{1+\frac dp}_{p,1}}^{\frac12}
\norm{\dS_N\delta u}_{\dB^{-1+\frac dp}_{p,1}}^{\frac12}
\norm{\dS_N\delta u}_{\dB^{1+\frac dp}_{p,1}}^{\frac12}
\\
&\le
\eta\alpha_*
\norm{\dS_N\delta u}_{\dB^{1+\frac dp}_{p,1}}
+
C
\norm{u_1}_{\dB^{1+\frac dp}_{p,1}}
\norm{\dS_N\delta u}_{\dB^{-1+\frac dp}_{p,1}}.
\end{aligned}
\end{equation}
With aid of  
$
\norm{u_1}_{L^2_t(\dB^{\frac dp}_{p,1})}
\lesssim1$ and $\norm{(\Id-\dS_N)\delta u}_{L^2_t(\dB^{\frac dp}_{p,1})}
\lesssim\varepsilon_N$
and we thus have
\begin{equation}\label{eqq5}
\begin{aligned}
\norm{u_1\cdot\nabla (\Id-\dS_N)\delta u}
_{L^1_t(\dB^{-1+\frac dp}_{p,1})}
\le C\varepsilon_N.
\end{aligned}
\end{equation}
By \eqref{eqq4} and \eqref{eqq5}, we have
\begin{equation}\label{eq:velocity-second-convection}
\begin{aligned}
\norm{\dS_N(u_1\cdot\nabla\delta u)}
_{L^1_t(\dB^{-1+\frac dp}_{p,1})}
&\quad\le
\eta\alpha_*
\norm{\dS_N\delta u}_{L^1_t(\dB^{1+\frac dp}_{p,1})}
\\
&\qquad+
C
\int_0^t
\norm{u_1(\tau)}_{\dB^{1+\frac dp}_{p,1}}
\norm{\dS_N\delta u(\tau)}_{\dB^{-1+\frac dp}_{p,1}}
\dd\tau
+
C\varepsilon_N.
\end{aligned}
\end{equation}
By \eqref{eq:uniform-bG-composition} and
\eqref{eq:uniform-local-bounds},
\[
\norm{b_i}_{L^\infty_T(L^\infty)}
+
\norm{b_i-1}_{L^\infty_T
(\dB^{\frac dp}_{p,1})}
\le C,
\qquad i=1,2.
\]
Since
\[
b_2-b_1=-b_1b_2\delta a,
\qquad
b_1b_2-1
=
(b_1-1)+(b_2-1)+(b_1-1)(b_2-1),
\]
from which and \cref{p} give
\[
\norm{b_1b_2}_{L^\infty_T(L^\infty)}
+
\norm{b_1b_2-1}_{L^\infty_T
(\dB^{\frac dp}_{p,1})}
\le C.
\]
Hence, by \eqref{eq:product-critical-low},
\eqref{eq:difference-tail}, and
\eqref{eq:uniform-local-bounds},
\begin{equation}\label{eq:velocity-coefficient-difference}
\begin{aligned}\norm{\dS_N((b_2-b_1)\cA u_1)}
_{L^1_t(\dB^{-1+\frac dp}_{p,1})}
&\lesssim
\int_0^t
\norm{b_1(\tau)b_2(\tau)\delta a(\tau)}
_{\dB^{\frac dp}_{p,1}}
\norm{\cA u_1(\tau)}
_{\dB^{-1+\frac dp}_{p,1}}
\,\dd\tau
\\
&\quad\lesssim
\int_0^t
\norm{\delta a(\tau)}_{\dB^{\frac dp}_{p,1}}
\norm{u_1(\tau)}_{\dB^{1+\frac dp}_{p,1}}
\,\dd\tau
\\
&\quad\le
C\int_0^t
\norm{u_1(\tau)}_{\dB^{1+\frac dp}_{p,1}}
A_N(\tau)\,\dd\tau
+
C\varepsilon_N.
\end{aligned}
\end{equation}
The identity
\[
G(a_2)-G(a_1)
=
\left(
\int_0^1
G'\bigl((1-\theta)a_1+\theta a_2\bigr)
\dd\theta
\right)\delta a
\]
and the composition estimate imply
\[
\begin{aligned}
\norm{
\int_0^1
G'\bigl((1-\theta)a_1+\theta a_2\bigr)
\dd\theta
}_{L^\infty}
+
\norm{
\int_0^1
\left[
G'\bigl((1-\theta)a_1+\theta a_2\bigr)-G'(0)
\right]\dd\theta
}_{\dB^{\frac dp}_{p,1}}
\le C.
\end{aligned}
\]
Consequently, for \(0\le\tau\le t\),
\begin{equation}\label{eq:velocity-pressure-difference}
\begin{aligned}
\norm{\dS_N\nabla(G(a_2)-G(a_1))}
_{L^1_t(\dB^{-1+\frac dp}_{p,1})}
&\lesssim
\int_0^t
\norm{\delta a(\tau)}_{\dB^{\frac dp}_{p,1}}
\dd\tau
\\
&\le
C\int_0^t(A_N(\tau)+
\varepsilon_N)\dd\tau.
\end{aligned}
\end{equation}
Substituting
\eqref{eq:commutator-low-result} and
\eqref{eq:velocity-first-convection}--%
\eqref{eq:velocity-pressure-difference} into
\eqref{eq:velocity-pre}, we obtain
\begin{equation}\label{eq:low-velocity-estimate-precise}
\begin{aligned}
U_N(t)\le{}&
C\norm{u_{2,0}-u_{1,0}}_{\dB^{-1+\frac dp}_{p,1}}
+
C\varepsilon_N
\\
&+
C\int_0^t
\left(
1+
\frac{K_m(\tau)^2}{\alpha_*}
+
\norm{u_1(\tau)}_{\dB^{1+\frac dp}_{p,1}}
+
\norm{u_2(\tau)}_{\dB^{1+\frac dp}_{p,1}}
\right)
\\
&\times
\bigl(A_N(\tau)+U_N(\tau)+\varepsilon_N\bigr)\dd\tau.
\end{aligned}
\end{equation}
Substituting \eqref{eq:low-velocity-estimate-precise} into the
term \(CU_N(t)\) on the right-hand side of
\eqref{eq:low-density-estimate}, and then adding
\eqref{eq:low-velocity-estimate-precise}, we obtain
\begin{equation}\label{eq:combined-low-frequency-estimate}
\begin{aligned}
A_N(t)+U_N(t)
\le{}&
C_N\left(
\norm{\delta a_0}_{\dB^{\frac dp}_{p,1}}
+
\norm{\delta u_0}_{\dB^{-1+\frac dp}_{p,1}}
+
\norm{\delta u}_{L^1_t(L^\infty)}
\right)
+
C\varepsilon_N
\\
&+
C\int_0^t
\left(
1+
\frac{K_m(\tau)^2}{\alpha_*}
+
\norm{u_1(\tau)}_{\dB^{1+\frac dp}_{p,1}}
+
\norm{u_2(\tau)}_{\dB^{1+\frac dp}_{p,1}}
\right)
\\
&\times
\bigl(A_N(\tau)+U_N(\tau)+\varepsilon_N\bigr)
\dd\tau.
\end{aligned}
\end{equation}
Set
\[
Y_N(t)
:=
A_N(t)+U_N(t)+\varepsilon_N, \,\quad\lambda(t)=1+
\frac{K_m(t)^2}{\alpha_*}
+
\norm{u_1(t)}_{\dB^{1+\frac dp}_{p,1}}
+
\norm{u_2(t)}_{\dB^{1+\frac dp}_{p,1}}.
\]
Then \eqref{eq:combined-low-frequency-estimate} implies
\begin{equation}\label{eq:low-frequency-gronwall-form}
\begin{aligned}
Y_N(t)
\le{}&
C_N
\left(
\norm{\delta a_0}_{\dB^{\frac dp}_{p,1}}
+
\norm{\delta u_0}_{\dB^{-1+\frac dp}_{p,1}}
+
\norm{\delta u}_{L^1_t(L^\infty)}
\right)
\\
&+
C\varepsilon_N
+
C\int_0^t 
\lambda(\tau)
Y_N(\tau)\,\dd\tau.
\end{aligned}
\end{equation}
Hence Gronwall's inequality for \eqref{eq:low-frequency-gronwall-form} gives
\[
Y_N(T)
\le
\exp \left( C \int_{0}^{T} \lambda(\tau) \, d\tau \right)\left(C_N
\left(
\norm{\delta a_0}_{\dB^{\frac dp}_{p,1}}
+
\norm{\delta u_0}_{\dB^{-1+\frac dp}_{p,1}}
+
\norm{\delta u}_{L^1_T(L^\infty)}
\right)
+
C\varepsilon_N\right)
\]
By \eqref{eq:uniform-local-bounds} and
\eqref{eq:Km-uniform-bound},
\[
\int_0^T
\lambda(t)\dd t
\le C,
\]
where \(C\) is independent of \(N\). Hence, we have
\[
\begin{aligned}
A_N(T)+U_N(T)
\le{}&
C_N\left(
\norm{\delta a_0}_{\dB^{\frac dp}_{p,1}}
+
\norm{\delta u_0}_{\dB^{-1+\frac dp}_{p,1}}
+
\norm{\delta u}_{L^1_T(L^\infty)}
\right)
+
C\varepsilon_N,
\end{aligned}
\]
which proves \eqref{eq:low-frequency-stability}.
\end{proof}
\begin{proof}[Proof of \cref{thm:main}]
Fix \(\eta_b>0\) as in
\cref{prop:low-frequency-stability}. For this choice of
\(\eta_b\), let \(r>0\), \(0<T\le T_0\), and \(m=m(T)\)
be furnished by \cref{cor:common-m}, and set
\(\mathcal U:=\mathcal U_r\).
Then every solution issued from \(\mathcal U\) belongs to
\(\mathscr S_{r,T}\) and satisfies
\eqref{eq:uniform-local-bounds} and
\eqref{eq:coefficient-splitting}--%
\eqref{eq:coefficient-tail-small}.

Since \(X_p\) and \(E_p(T)\) are metric spaces, it suffices to prove
sequential continuity. Fix $(a_0^\ast,u_0^\ast)\in\mathcal U$
and let 
$
(a_0^n,u_0^n)\in\mathcal U$ for $n\ge1$ satisfy
$
(a_0^n,u_0^n)
\rightarrow
(a_0^\ast,u_0^\ast)$ in $X_p$ as $ n\to\infty$. Denote the corresponding solutions by
\[
(a^n,u^n):=S_T(a_0^n,u_0^n),
\qquad
(a^\ast,u^\ast):=S_T(a_0^\ast,u_0^\ast),
\]
and set
\[
\delta a^n:=a^n-a^\ast,
\qquad
\delta u^n:=u^n-u^\ast.
\]
Applying \cref{lem:common-envelope}, there exists
\(\Omega\in AF(\delta_0)\) with
\(\Omega_j\to\infty\) as \(j\to\infty\) such that
\begin{equation}\label{eq:final-common-envelope}
\begin{aligned}
\sup_{n\ge1}
\left(
\norm{a_0^n}_{\dB^{\frac dp}_{p,1}(\Omega)}
+
\norm{u_0^n}_{\dB^{-1+\frac dp}_{p,1}(\Omega)}
\right)
+
\norm{a_0^\ast}_{\dB^{\frac dp}_{p,1}(\Omega)}
+
\norm{u_0^\ast}_{\dB^{-1+\frac dp}_{p,1}(\Omega)}
<\infty.
\end{aligned}
\end{equation}
Set
\[
\begin{aligned}
M_\Omega
:=
1+
\sup_{n\ge1}
\left(
\norm{a_0^n}_{\dB^{\frac dp}_{p,1}(\Omega)}
+
\norm{u_0^n}_{\dB^{-1+\frac dp}_{p,1}(\Omega)}
\right)+
\norm{a_0^\ast}_{\dB^{\frac dp}_{p,1}(\Omega)}
+
\norm{u_0^\ast}_{\dB^{-1+\frac dp}_{p,1}(\Omega)}
\end{aligned}
\]
and
\[
\mathcal S_\Omega
:=
\left\{
(a,u)\in\mathscr S_{r,T}:
\norm{a_0}_{\dB^{\frac dp}_{p,1}(\Omega)}
+
\norm{u_0}_{\dB^{-1+\frac dp}_{p,1}(\Omega)}
\le M_\Omega
\right\}.
\]
By \cref{cor:uniform-tails}, there exists a nonincreasing sequence
\((\varepsilon_N)_{N\in\Z}\), independent of \(n\), such that
\[
\varepsilon_N\rightarrow0
\qquad\text{as }N\to\infty.
\]
By \eqref{eq:uniform-tails}, we have
\begin{equation}\label{eq:final-uniform-difference-tail}
\begin{aligned}
\sup_{n\ge1}\Bigl(
&
\norm{(\Id-\dS_N)\delta a^n}
_{\tdL^\infty_T(\dB^{\frac dp}_{p,1})}
+
\norm{(\Id-\dS_N)\delta u^n}
_{\tdL^\infty_T(\dB^{-1+\frac dp}_{p,1})}
\\
&+
\norm{(\Id-\dS_N)\delta u^n}
_{L^1_T(\dB^{1+\frac dp}_{p,1})}
\Bigr)
\le C\varepsilon_N.
\end{aligned}
\end{equation}
By \cref{lem:L1Linfty},
\begin{equation}\label{eq:gamma-to-zero}
\begin{aligned}
\norm{\delta u^n}_{L^1_T(L^\infty)}\le
C_T
\norm{
(a_0^n-a_0^\ast,u_0^n-u_0^\ast)
}_{X_p}\rightarrow0 \qquad\text{as }n\to\infty.
\end{aligned}
\end{equation}
For every fixed \(N\), applying
\cref{prop:low-frequency-stability} to
\((a^n,u^n),(a^\ast,u^\ast)\in\mathcal S_\Omega\) gives
\[
\begin{aligned}
\norm{\dS_N\delta a^n}
_{\tdL^\infty_T(\dB^{\frac dp}_{p,1})}
&+
\norm{\dS_N\delta u^n}
_{\tdL^\infty_T(\dB^{-1+\frac dp}_{p,1})}
+
\norm{\dS_N\delta u^n}
_{L^1_T(\dB^{1+\frac dp}_{p,1})}
\\
&\le
C_N\left(
\norm{a_0^n-a_0^\ast}_{\dB^{\frac dp}_{p,1}}
+
\norm{u_0^n-u_0^\ast}_{\dB^{-1+\frac dp}_{p,1}}
+
\norm{\delta u^n}_{L^1_T(L^\infty)}
\right)
+
C\varepsilon_N,
\end{aligned}
\]
from which it follows that there exists a constant \(C\), independent of \(N\) such that 
\begin{equation}\label{eq:final-low-frequency-limsup}
\begin{aligned}
\limsup_{n\to\infty}\Bigl(
\norm{\dS_N\delta a^n}
_{\tdL^\infty_T(\dB^{\frac dp}_{p,1})}
+
\norm{\dS_N\delta u^n}
_{\tdL^\infty_T(\dB^{-1+\frac dp}_{p,1})}
+
\norm{\dS_N\delta u^n}
_{L^1_T(\dB^{1+\frac dp}_{p,1})}
\Bigr)
\le C\varepsilon_N.
\end{aligned}
\end{equation}

\noindent By the triangle inequality,
\eqref{eq:final-uniform-difference-tail} and
\eqref{eq:final-low-frequency-limsup}, we obtain
\[
\begin{aligned}
\limsup_{n\to\infty}\Bigl(
\norm{\delta a^n}
_{\tdL^\infty_T(\dB^{\frac dp}_{p,1})}
+
\norm{\delta u^n}
_{\tdL^\infty_T(\dB^{-1+\frac dp}_{p,1})}+
\norm{\delta u^n}
_{L^1_T(\dB^{1+\frac dp}_{p,1})}
\Bigr)
\le C\varepsilon_N
\end{aligned}
\]
for every \(N\in\Z\). Letting \(N\to\infty\), we conclude that
\[
\begin{aligned}
\norm{\delta a^n}
_{\tdL^\infty_T(\dB^{\frac dp}_{p,1})}
+  
\norm{\delta u^n}
_{\tdL^\infty_T(\dB^{-1+\frac dp}_{p,1})}
+
\norm{\delta u^n}
_{L^1_T(\dB^{1+\frac dp}_{p,1})}
\rightarrow0
\quad\text{as }n\to\infty.
\end{aligned}
\]
Therefore,
\[
\norm{
(a^n,u^n)-(a^\ast,u^\ast)
}_{E_p(T)}
\rightarrow0
\qquad\text{as }n\to\infty.
\]
Thus \(S_T\) is continuous at
\((a_0^\ast,u_0^\ast)\). Since this point is arbitrary,
\(S_T\) is continuous on \(\mathcal U\).
\end{proof}
\appendix
\section{Some Harmonic analysis tools}
We use $S$ and $S'$ to denote the spaces of Schwartz functions and tempered distributions on $\mathbb{R}^d$, respectively. The Fourier transform of a function $f$ is defined by $\mathcal{F} f(\xi) = \hat{f}(\xi) = (2\pi)^{-d/2} \int_{\mathbb{R}^d} e^{-ix \cdot \xi} f(x) \, dx$,
and its inverse Fourier transform is denoted by $\mathcal{F} ^{-1}f = \check{f}$. All function spaces considered below are defined over $\mathbb{R}^d$. For simplicity, the underlying domain will often be omitted; for instance, we write $L^p$ instead of $L^p(\mathbb{R}^d)$, unless otherwise specified.

Let $\varphi \in C_c^\infty(\mathbb{R}^d)$ satisfy $0 \leq \varphi \leq 1$, with $\varphi = 1$ on $B(0, 1/2)$ and $\varphi = 0$ outside $B(0, 1)$. Set $\psi(\xi) := \varphi(\xi/2) - \varphi(\xi),$
and define $\psi_j(\xi):= \psi(2^{-j}\xi), \,\varphi_j(\xi) := \varphi(2^{-j}\xi).$
The homogeneous Littlewood--Paley operators are defined by
\[
\dot\Delta_j := (\mathcal{F}^{-1} \psi_j) * , \quad 
\dot{S}_j =P_{\le j} := (\mathcal{F}^{-1} \varphi_j) * , \quad 
P_{>j} := \Id - \dot{S}_j, 
\]
where $*$ is the convolution operator in $\mathbb{R}^d$.
It is straightforward to verify that $\dot{S}_j f = \sum_{k \leq j-1} \dot\Delta_k f$ for $j \in \mathbb{Z}$.
Formally, the product of two distributions $f$ and $g$ may be decomposed according to Bony's decomposition as
\[
fg = \dot{T}_f g + \dot{T}_g f + \dot{R}(f, g),  
\]
where
\[
\dot{T}_f g := \sum_j \dot{S}_{j-1} f \cdot \dot\Delta_j g, 
  \quad\dot{R}(f, g) := \sum_{|j-k| \leq 1} \dot\Delta_j f \cdot \dot\Delta_k g = \sum_j \dot\Delta_j f \cdot \widetilde{\Delta}_j g. 
\]
Here,
\[
\widetilde{\dot \Delta}_j := \dot \Delta_{j-1} + \dot\Delta_j + \dot \Delta_{j+1}.
\]
With the above choice of $\varphi$, we have
\[
\dot\Delta_j \dot\Delta_k f = 0, \quad |j - k| \geq 2,  
\]
and
\[
\dot\Delta_j (\dot{S}_{k-1} f \cdot \dot\Delta_k g) = 0, \quad |j - k| \geq 5. 
\]

The following Bernstein's inequalities describe how derivatives act on spectrally localized functions. Let $ \mathcal{B}$ be a Ball and $\mathcal{C}$ be an annulus. There exists a constant $C>0$ such that for all $k\in \mathbb{N}\cup \{0\}$, any positive real number $\lambda$ and any function $f\in L^p$ with $1\leq p \leq q \leq \infty$, we have
\begin{align*}
&{\rm{supp}}\hat{f}\subset \lambda \mathcal{B}\;\Rightarrow\; \|D^{k}f\|_{L^q}\leq C^{k+1}\lambda^{k+d(\frac{1}{p}-\frac{1}{q})}\|f\|_{L^p},  \\
&{\rm{supp}}\hat{f}\subset \lambda \mathcal{C}\;\Rightarrow\; C^{-k-1}\lambda^{k}\|f\|_{L^p} \leq \|D^{k}f\|_{L^p} \leq C^{k+1}\lambda^{k}\|f\|_{L^p}.
\end{align*}

Let $S' / \mathcal{P}\cong
S_h' := \left\{ f \in S' : \lim_{j \to -\infty} \dot{S}_j f = 0 \text{ in } S' \right\}
$ denote the space of tempered distributions modulo polynomials. For $1 \leq p, q \leq \infty$ and $s \in \mathbb{R}$, define
\[
\dot{B}_{p,q}^s := \{ f \in S' / \mathcal{P} : \| f \|_{\dot{B}_{p,q}^s} < \infty \},
\]
where
\[
\| f \|_{\dot{B}_{p,q}^s} := \left( \sum_{j \in \mathbb{Z}} 2^{jsq} \| \dot\Delta_j f \|_{L^p}^{q} \right)^{1/q},
\]
with the usual modification when $q = \infty$. Let
\[
\|f(t, \cdot)\|_{L^q_T(\dot{B}_{p,r}^s)} := \left\| \|f(t, \cdot)\|_{\dot{B}_{p,r}^s} \right\|_{L^q_T}
\]
and
\[
\|f(t, \cdot)\|_{\widetilde{L}^q_T(\dot{B}_{p,r}^s)} := \left\| 2^{js} \| \dot{\Delta}_j f(t, \cdot) \|_{L^q_T(L^p)} \right\|_{\ell^r(\mathbb{Z})}.
\]
It is worth emphasizing that, due to the Minkowski inequality, if \( r \leq q \), then
\[
 \|f(t, \cdot)\|_{L^q_T(\dot{B}_{p,r}^s)} \leq \|f(t, \cdot)\|_{\widetilde{L}^q_T(\dot{B}_{p,r}^s)}.
\]
Conversely, the reverse inequality holds if \( r \geq q \).
For \( s \in \mathbb{R} \) and \( 1 \leq p, q \leq \infty \), the multiplier space associated with \( \dot{B}_{p,q}^s \) is defined by
\[
\mathcal{M}(\dot{B}_{p,q}^s) := \left\{ f : \psi f \in \dot{B}_{p,q}^s \text{ for every } \psi \in \dot{B}_{p,q}^s \right\},
\]
endowed with the norm
\begin{equation}
 \|f\|_{\mathcal{M}(\dot{B}_{p,q}^s)} := \sup_{\|\psi\|_{\dot{B}_{p,q}^s} \leq 1} \|\psi f\|_{\dot{B}_{p,q}^s}. \label{mutiple}   
\end{equation}
\begin{lemma}\label{lem:convolution-commutator}
\textup{\cite[Lemma~2.97]{BaChDa11}.}
Let \(\theta\) be a \(C^1\) function on \(\R^d\) such that $(1+|\cdot|)\widehat{\theta}\in L^1.$ There exists a constant \(C\) such that for any Lipschitz function
\(a\) with gradient in \(L^p\) and any function \(b\) in \(L^q\), we
have, for any positive \(\lambda\),
\[
\norm{[\theta(\lambda^{-1}D),a]b}_{L^r}
\le
C\lambda^{-1}
\norm{\nabla a}_{L^p}
\norm b_{L^q},
\qquad
\text{with}
\qquad
\frac1p+\frac1q=\frac1r.
\]
\end{lemma}
\begin{lemma}\cite[Theorem~2.47]{BaChDa11}
\label{u}
There exists a constant \(C\) such that for any real number \(s\)
and any \((p,r)\) in \([1,\infty]^2\), we have, for any
\((u,v)\) in \(L^\infty\times\dot B^s_{p,r}\),
\begin{equation}\label{1}
\|\dot T_u v\|_{\dot B^s_{p,r}}
\le
C^{1+|s|}
\|u\|_{L^\infty}
\|v\|_{\dot B^s_{p,r}}.
\end{equation}
Moreover, for any \((s,t)\) in
\(\mathbb R\times(-\infty,0)\) and any
\((p,r_1,r_2)\) in \([1,\infty]^3\), we have, for any $(u,v)\in
\dot B^t_{\infty,r_1}
\times
\dot B^s_{p,r_2},$
\begin{equation}\label{2}
\|\dot T_u v\|_{\dot B^{s+t}_{p,r}}
\le
\frac{C^{1+|s+t|}}{-t}
\|u\|_{\dot B^t_{\infty,r_1}}
\|v\|_{\dot B^s_{p,r_2}},
\qquad
\text{with}
\qquad
\frac1r
\overset{\mathrm{def}}{=}
\min\left\{
1,\frac1{r_1}+\frac1{r_2}
\right\}.
\end{equation}
\end{lemma}

\begin{lemma}\cite[Theorem~2.52]{BaChDa11}
\label{v}
A constant \(C\) exists which satisfies the following inequalities.
Let \((s_1,s_2)\) be in \(\mathbb R^2\) and
\((p_1,p_2,r_1,r_2)\) be in \([1,\infty]^4\). Assume that
\[
\frac1p
\overset{\mathrm{def}}{=}
\frac1{p_1}+\frac1{p_2}
\le1
\qquad\text{and}\qquad
\frac1r
\overset{\mathrm{def}}{=}
\frac1{r_1}+\frac1{r_2}
\le1.
\]
If \(s_1+s_2\) is positive, then we have, for any
$(u,v)\in
\dot B^{s_1}_{p_1,r_1}
\times
\dot B^{s_2}_{p_2,r_2},$
\begin{equation}\label{r}
\|\dot R(u,v)\|_{\dot B^{s_1+s_2}_{p,r}}
\le
\frac{C^{|s_1+s_2|+1}}{s_1+s_2}
\|u\|_{\dot B^{s_1}_{p_1,r_1}}
\|v\|_{\dot B^{s_2}_{p_2,r_2}}.
\end{equation}
\end{lemma}

\begin{lemma}\cite[Corollary~2.54]{BaChDa11}
\label{p}
If $(s,p,r)\in ]0,\infty[\times[1,\infty]^2$
satisfies $s<\frac dp,$
or $
s=\frac dp$
and $r=1,$
then \(L^\infty\cap\dot B^s_{p,r}\) is an algebra. Moreover, there
exists a constant \(C\), depending only on the dimension \(d\), such
that
\begin{equation}\label{eq:BCD-algebra}
\|uv\|_{\dot B^s_{p,r}}
\le
\frac{C^{s+1}}{s}
\left(
\|u\|_{L^\infty}\|v\|_{\dot B^s_{p,r}}
+
\|u\|_{\dot B^s_{p,r}}\|v\|_{L^\infty}
\right).
\end{equation}
\end{lemma}
\begin{lemma}\label{lem:basic-products}\cite[Proposition 1.6]{D2001}
Let \(d\ge2\). For  \( 1 \leq p < \infty \) and \( f, g \in \dot{B}_{p,1}^{\frac{d}{p}}  \).  
Then the following estimate holds  
\begin{align}
\norm{fg}_{\dB^{\frac{d}{p}}_{p,1}}
&\lesssim \norm f_{\dB^{\frac{d}{p}}_{p,1}}\norm g_{\dB^{\frac{d}{p}}_{p,1}};
\label{eq:product-critical}
\end{align}
Moreover, if \(1\leq p<2d\). and $f\in\dot B^{\frac{d}{p}}_{p,1},\,
g\in\dot B^{\frac{d}{p}-1}_{p,1}$, then
\begin{equation}\label{eq:product-critical-low}
\|fg\|_{\dot B^{\frac{d}{p}-1}_{p,1}}
\lesssim
\|f\|_{\dot B^{\frac{d}{p}}_{p,1}}
\|g\|_{\dot B^{\frac{d}{p}-1}_{p,1}}.
\end{equation}
\end{lemma}
\begin{lemma}\cite[Lemma 2.9]{CMZ2010R}
Let \( s > 0 \) and \( 1 \leq p, q \leq \infty \). Assume that \( F \in W_{\text{loc}}^{[s]+3,\infty}(\mathbb{R}) \) with \( F(0) = 0 \). Then for any \( f \in L_T^\infty(L^\infty) \cap \widetilde{L}_T^q(\dot{B}_{p,1}^s) \), we have
\begin{equation}
 \|F(f)\|_{\widetilde{L}_T^q(\dot{B}_{p,1}^s)} \leq C \left(1 + \|f\|_{L_T^\infty(L^\infty)}\right)^{[s]+2} \|f\|_{\widetilde{L}_T^q(\dot{B}_{p,1}^s)}. \label{eq:composition}  
\end{equation}                                      
\end{lemma}
\begin{lemma}
\label{lem:transport-estimate}
\textup{\cite[Theorem~3.14 and Remarks~3.15--3.16]{BaChDa11}.}
Let \(1\le p<\infty\) and
\[
-\min\left\{\frac{d}{p},\frac{d}{p'}\right\}
<s\le\frac{d}{p}.
\]
Let \(f\) solve
\[
\partial_t f+v\cdot\nabla f=g,
\qquad
f|_{t=0}=f_0.
\]
Assume that
\[
f_0\in\dB^s_{p,1},
\qquad
g\in L^1_T(\dB^s_{p,1}),
\qquad
\nabla v\in L^1_T
\bigl(\dB^{\frac{d}{p}}_{p,1}\bigr).
\]
Set
\[
V(t):=
\int_0^t
\norm{\nabla v(\tau)}
_{\dB^{\frac{d}{p}}_{p,1}}
\dd\tau.
\]
Then there exists a constant \(C>0\), depending only on
\(d,p\), and \(s\), such that, for every \(t\in[0,T]\),
\begin{equation}\label{eq:transport-estimate}
\norm f_{\tdL^\infty_t(\dB^s_{p,1})}
\le
e^{CV(t)}
\left(
\norm{f_0}_{\dB^s_{p,1}}
+
\int_0^t
e^{-CV(\tau)}
\norm{g(\tau)}_{\dB^s_{p,1}}
\dd\tau
\right).
\end{equation}
\end{lemma}

\begin{lemma}\label{lem:cutoff-commutator}
Let \(d\ge2\), \(1<p<2d\), and $b=b_m+b^h$, where
\[b_m:=1+\dS_m(b-1),\qquad b^h:=(\Id-\dS_m)(b-1).\]
Assume that
\[
\nabla b_m\in\dB^{\frac dp}_{p,1},
\qquad
b^h\in\dB^{\frac dp}_{p,1},
\qquad
v\in
\dB^{\frac dp}_{p,1}
\cap
\dB^{1+\frac dp}_{p,1}.
\]
Then, for every \(N\in\Z\) and  \(L(D):=\mathcal{A}=\mu\Delta +(\lambda+\mu)\nabla\operatorname{div},\)
\begin{equation}\label{eq:cutoff-commutator-lemma}
\begin{aligned}
\norm{[\dS_N,b]L(D)v}_{\dB^{-1+\frac dp}_{p,1}}
\le C
\norm{\nabla b_m}_{\dB^{\frac dp}_{p,1}}
\norm v_{\dB^{\frac dp}_{p,1}}+
C\norm{b^h}_{\dB^{\frac dp}_{p,1}}
\norm v_{\dB^{1+\frac dp}_{p,1}}.
\end{aligned}
\end{equation}
The implicit constant $C$ is independent of \(m\) and \(N\), and may
depend on the coefficients of \(L(D)\).

\label{lem:cutoff-transport-commutator}
\noindent Moreover, if
\[
u\in\dB^{1+\frac dp}_{p,1},
\qquad
f\in\dB^{\frac dp}_{p,1},
\]
then, for every \(N\in\Z\),
\begin{equation}\label{eq:cutoff-transport-commutator}
\norm{[u\cdot\nabla,\dS_N]f}_{\dB^{\frac dp}_{p,1}}
\le C
\norm u_{\dB^{1+\frac dp}_{p,1}}
\norm f_{\dB^{\frac dp}_{p,1}},
\end{equation}
where the implicit constant $C$ is independent of \(N\).
\end{lemma}

\begin{proof}It suffices to prove the following auxiliary estimate. Let
\(\rho\in\{1,2\}\) and \(P_\rho(D)\) be a constant-coefficient
homogeneous differential operator of order \(\rho\). Assume that
\[
\nabla e\in\dB^{\frac dp}_{p,1},
\qquad
w\in\dB^{\frac dp}_{p,1}.
\]
We claim that
\begin{equation}\label{eq:general-lowpass-commutator}
\norm{[\dS_N,e]P_\rho(D)w}
_{\dB^{1-\rho+\frac dp}_{p,1}}
\lesssim
\norm{\nabla e}_{\dB^{\frac dp}_{p,1}}
\norm w_{\dB^{\frac dp}_{p,1}}.
\end{equation}
Indeed, Bony's decomposition gives
\begin{equation}\label{eq:general-lowpass-bony-decomposition}
\begin{aligned}
[\dS_N,e]P_\rho(D)w
={}&
[\dS_N,\dot T_e]P_\rho(D)w+
\dS_N\dot T_{P_\rho(D)w}e
-
\dot T_{\dS_NP_\rho(D)w}e
\\
&+
\dS_N\dot R\bigl(P_\rho(D)w, e\bigr)
-
\dot R\bigl(\dS_NP_\rho(D)w, e\bigr).
\end{aligned}
\end{equation}
For the first term, there exists an integer \(N_0\ge1\), such that
\[
[\dS_N,\dot T_e]P_\rho(D)w
=
\sum_{|k-N|\le N_0}
[\dS_N,\dS_{k-1}e]
P_\rho(D)\dDelta_kw,
\]
from which, by \cref{lem:convolution-commutator} and Bernstein's inequality, we have
\begin{equation}\label{eq4}
\begin{aligned}
\norm{
[\dS_N,\dot T_e]P_\rho(D)w
}_{\dB^{1-\rho+\frac dp}_{p,1}}
&\lesssim
\sum_{|k-N|\le N_0}
2^{k(1-\rho+\frac dp)}
2^{-N}
\norm{\nabla\dS_{k-1}e}_{L^\infty}
\norm{P_\rho(D)\dDelta_kw}_{L^p}
\\
&\lesssim
\sum_{|k-N|\le N_0}
2^{k\frac dp}
\norm{\dDelta_kw}_{L^p}
\sum_{\ell\le k-2}
2^{\ell\frac dp}
\norm{\dDelta_\ell\nabla e}_{L^p}
\\
&\lesssim
\norm{\nabla e}_{\dB^{\frac dp}_{p,1}}
\norm w_{\dB^{\frac dp}_{p,1}}.
\end{aligned}
\end{equation}
For the second term in
\eqref{eq:general-lowpass-bony-decomposition},
\[
\begin{aligned}
\dS_N\dot T_{P_\rho(D)w}e
-
\dot T_{\dS_NP_\rho(D)w}e
={}&
\sum_{k\in\Z}
\sum_{\ell\le k-2}
\Bigl[
\dS_N\bigl(
\dDelta_\ell P_\rho(D)w\,
\dDelta_ke
\bigr)\\
&-
\bigl(
\dS_N\dDelta_\ell P_\rho(D)w
\bigr)
\dDelta_ke
\Bigr],
\end{aligned}
\]
from which, we thus have
\begin{equation}\label{eq5}
\begin{aligned}
\norm{
\dS_N\dot T_{P_\rho(D)w}e
-
\dot T_{\dS_NP_\rho(D)w}e
}_{\dB^{1-\rho+\frac dp}_{p,1}}
&\lesssim
\sum_{k\in\Z}
2^{k(1-\rho+\frac dp)}
\norm{\dDelta_ke}_{L^p}
\sum_{\ell\le k-2}
2^{\ell(\rho+\frac dp)}
\norm{\dDelta_\ell w}_{L^p}
\\
&\lesssim
\sum_{k\in\Z}
2^{k\frac dp}
\norm{\dDelta_k\nabla e}_{L^p}
\sum_{\ell\le k-2}
2^{-\rho(k-\ell)}
2^{\ell\frac dp}
\norm{\dDelta_\ell w}_{L^p}
\\
&\lesssim
\norm{\nabla e}_{\dB^{\frac dp}_{p,1}}
\norm w_{\dB^{\frac dp}_{p,1}}.
\end{aligned}
\end{equation}
For the remainder terms, we have
\[
\begin{aligned}
\dDelta_j
\left(
\dS_N\dot R\bigl(P_\rho(D)w,e\bigr)
-
\dot R\bigl(\dS_NP_\rho(D)w,e\bigr)
\right)
={}&
\sum_{k\ge j-3}
\dDelta_j
\Bigl[
\dS_N\bigl(
\dDelta_ke\,
\widetilde{\dDelta}_kP_\rho(D)w
\bigr)\\
&-
\dDelta_ke\,
\widetilde{\dDelta}_k
\dS_NP_\rho(D)w
\Bigr].
\end{aligned}
\]
Suppose first that \(2\le p<2d\). H\"older's inequality in
\(L^{\frac p2}\) and Bernstein's inequality from
\(L^{\frac p2}\) to \(L^p\) yield
\begin{equation}\label{eq1}
\begin{aligned}
&
2^{j(1-\rho+\frac dp)}
\norm{
\dDelta_j
\left(
\dS_N\dot R\bigl(P_\rho(D)w,e\bigr)
-
\dot R\bigl(\dS_NP_\rho(D)w,e\bigr)
\right)
}_{L^p}
\\
&\quad\lesssim
\sum_{k\ge j-3}
2^{j(1-\rho+\frac{2d}{p})}
2^{k(\rho-1)}
\norm{\dDelta_k\nabla e}_{L^p}
\norm{\widetilde{\dDelta}_kw}_{L^p}
\\
&\quad=
\sum_{k\ge j-3}
2^{-(k-j)(1-\rho+\frac{2d}{p})}
\left(
2^{k\frac dp}
\norm{\dDelta_k\nabla e}_{L^p}
\right)
\left(
2^{k\frac dp}
\norm{\widetilde{\dDelta}_kw}_{L^p}
\right).
\end{aligned}
\end{equation}
Suppose now that \(1<p<2\). By Bernstein's inequality and $\frac{1}{p'}+\frac{1}{p}=1$, we obtain
\begin{equation}\label{eq2}
\begin{aligned}
&
2^{j(1-\rho+\frac dp)}
\norm{
\dDelta_j
\left(
\dS_N\dot R\bigl(P_\rho(D)w,e\bigr)
-
\dot R\bigl(\dS_NP_\rho(D)w,e\bigr)
\right)
}_{L^p}
\\
&\quad\lesssim
\sum_{k\ge j-3}
2^{j(1-\rho+d)}
2^{k(\rho-1+d(\frac2p-1))}
\norm{\dDelta_k\nabla e}_{L^p}
\norm{\widetilde{\dDelta}_kw}_{L^p}
\\
&\quad=
\sum_{k\ge j-3}
2^{-(k-j)(1-\rho+d)}
\left(
2^{k\frac dp}
\norm{\dDelta_k\nabla e}_{L^p}
\right)
\left(
2^{k\frac dp}
\norm{\widetilde{\dDelta}_kw}_{L^p}
\right).
\end{aligned}
\end{equation}
Thus, from \eqref{eq1} and \eqref{eq2}, we have
\begin{equation}\label{eq3}
\begin{aligned}
&
2^{j(1-\rho+\frac dp)}
\norm{
\dDelta_j
\left(
\dS_N\dot R\bigl(P_\rho(D)w,e\bigr)
-
\dot R\bigl(\dS_NP_\rho(D)w,e\bigr)
\right)
}_{L^p}
\\
&\quad\lesssim
\sum_{k\ge j-3}
2^{-(k-j)\gamma_\rho}
\left(
2^{k\frac dp}
\norm{\dDelta_k\nabla e}_{L^p}
\right)
\left(
2^{k\frac dp}
\norm{\widetilde{\dDelta}_kw}_{L^p}
\right),
\end{aligned}
\end{equation}
where
\[
\gamma_\rho
:=
1-\rho+\frac dp+
\min\left\{\frac dp,\frac d{p'}\right\}
=
\begin{cases}
\displaystyle
1-\rho+\frac{2d}{p},
&2\le p<2d,
\\[2mm]
1-\rho+d,
&1<p<2.
\end{cases}
\]
For \(\rho\in\{1,2\}\) and \(1<p<2d\), one has
\(\gamma_\rho>0\). Consequently, from \eqref{eq3}, we have
\begin{equation}\label{eq6}
\begin{aligned}
&
\norm{
\dS_N\dot R\bigl(P_\rho(D)w,e\bigr)
-
\dot R\bigl(\dS_NP_\rho(D)w,e\bigr)
}_{\dB^{1-\rho+\frac dp}_{p,1}}
\\
&\quad\lesssim
\sum_{j\in\Z}
\sum_{k\ge j-3}
2^{-(k-j)\gamma_\rho}
\left(
2^{k\frac dp}
\norm{\dDelta_k\nabla e}_{L^p}
\right)
\left(
2^{k\frac dp}
\norm{\widetilde{\dDelta}_kw}_{L^p}
\right)
\\
&\quad\lesssim
\norm{\nabla e}_{\dB^{\frac dp}_{p,1}}
\norm w_{\dB^{\frac dp}_{p,1}}.
\end{aligned}
\end{equation}
Combining 
\eqref{eq:general-lowpass-bony-decomposition}, \eqref{eq4}, \eqref{eq5} and \eqref{eq6} yields
\eqref{eq:general-lowpass-commutator}.

 Taking $\rho=2,\, e=b_m,\, P_2(D)=L(D),\, w=v$ for \eqref{eq:general-lowpass-commutator}, we obtain
\begin{equation}\label{eq:cutoff-low-coefficient-bound}
\norm{[\dS_N,b_m]L(D)v}_{\dB^{-1+\frac dp}_{p,1}}
\lesssim
\norm{\nabla b_m}_{\dB^{\frac dp}_{p,1}}
\norm v_{\dB^{\frac dp}_{p,1}}.
\end{equation}
For the high-frequency coefficient, by \cref{lem:basic-products}, we obtain
\[
\begin{aligned}
\norm{[\dS_N,b^h]L(D)v}
_{\dB^{-1+\frac dp}_{p,1}}
&\le
\norm{\dS_N\bigl(b^hL(D)v\bigr)}
_{\dB^{-1+\frac dp}_{p,1}}
+
\norm{b^h\dS_NL(D)v}
_{\dB^{-1+\frac dp}_{p,1}}
\\
&\lesssim
\norm{b^h}_{\dB^{\frac dp}_{p,1}}
\norm v_{\dB^{1+\frac dp}_{p,1}},
\end{aligned}
\]
from which, \eqref{eq:cutoff-low-coefficient-bound} together with
\[
[\dS_N,b]L(D)v
=
[\dS_N,b_m]L(D)v
+
[\dS_N,b^h]L(D)v,
\]
we can prove \eqref{eq:cutoff-commutator-lemma}.

Finally, since $[u\cdot\nabla,\dS_N]f=-\sum_{\alpha=1}^d[\dS_N,u^\alpha]\partial_\alpha f$, applying \eqref{eq:general-lowpass-commutator} with
$
\rho=1,
\,
e=u^\alpha,
\,
P_1(D)=\partial_\alpha,
\,
w=f$
and summing over \(\alpha\) gives
\[
\norm{[u\cdot\nabla,\dS_N]f}_{\dB^{\frac dp}_{p,1}}
\lesssim
\norm u_{\dB^{1+\frac dp}_{p,1}}
\norm f_{\dB^{\frac dp}_{p,1}}.
\]
This proves \eqref{eq:cutoff-transport-commutator}.
\end{proof}

\begin{lemma}{\cite[Lemma 3.5]{CMZ2010R}}
\label{lem:CMZ-time-weighted-products}
Let \(1\le p,q,q_1,q_2\le\infty\) with
$
\frac1{q_1}+\frac1{q_2}=\frac1q.
$
Assume that
$
f\in
\tdL^{q_1}_T\bigl(\dB^{s_1}_{p,1}(\omega_T)\bigr)$, 
$g\in
\tdL^{q_2}_T\bigl(\dB^{s_2}_{p,1}\bigr)$. Then there hold
\begin{enumerate}
\item[(a)] If \(s_2\le \frac{d}{p}\), we have
\[
\norm{T_gf}_{\tdL^q_T
\left(\dB^{s_1+s_2-\frac{d}{p}}_{p,1}(\omega_T)\right)}
\le
C
\norm f_{\tdL^{q_1}_T
\left(\dB^{s_1}_{p,1}(\omega_T)\right)}
\norm g_{\tdL^{q_2}_T
\left(\dB^{s_2}_{p,1}\right)};
\]
\item[(b)] If \(s_1\le \frac{d}{p}-1\), we have
\[
\norm{T_fg}_{\tdL^q_T
\left(\dB^{s_1+s_2-\frac{d}{p}}_{p,1}(\omega_T)\right)}
\le
C
\norm f_{\tdL^{q_1}_T
\left(\dB^{s_1}_{p,1}(\omega_T)\right)}
\norm g_{\tdL^{q_2}_T
\left(\dB^{s_2}_{p,1}\right)};
\]
\item[(c)] If
\(
s_1+s_2>
d\max\left(0,\frac2p-1\right),
\)
we have
\[
\norm{R(f,g)}_{\tdL^q_T
\left(\dB^{s_1+s_2-\frac{d}{p}}_{p,1}(\omega_T)\right)}
\le
C
\norm f_{\tdL^{q_1}_T
\left(\dB^{s_1}_{p,1}(\omega_T)\right)}
\norm g_{\tdL^{q_2}_T
\left(\dB^{s_2}_{p,1}\right)}.
\]
\end{enumerate}
\end{lemma}
\begin{remark}
\label{lem:time-weighted-critical-product}
From \cref{lem:CMZ-time-weighted-products}, we can easily obtain that if
\(
f\in
\tdL^\infty_T(\dB^{\frac{d}{p}}_{p,1})
\cap
\tdL^\infty_T(\dB^{\frac{d}{p}}_{p,1}(\omega_T))
\) and $g \in L^1_T(\dB^{\frac{d}{p}}_{p,1})
\cap
L^1_T(\dB^{\frac{d}{p}}_{p,1}(\omega_T))$, then 
\[
\begin{aligned}
\norm{fg}_{L^1_T(\dB^{\frac{d}{p}}_{p,1}(\omega_T))}
\le C\Bigl(
&\norm f_{\tdL^\infty_T(\dB^{\frac{d}{p}}_{p,1})}
 \norm g_{L^1_T(\dB^{\frac{d}{p}}_{p,1}(\omega_T))}
+
\norm f_{\tdL^\infty_T
(\dB^{\frac{d}{p}}_{p,1}(\omega_T))}
\norm g_{L^1_T(\dB^{\frac{d}{p}}_{p,1})}
\Bigr),
\end{aligned}
\]
where \(C\) is independent of \(T\).
\end{remark}
\begin{lemma}
\label{lem:time-weighted-composition-CMZ}
\textup{\cite[Proposition 3.8]{CMZ2010R}}
Let \(s>0\) and \(1\le p,q\le\infty\). Assume that
\[
F\in W^{[s]+3,\infty}_{\rm loc}(\R),
\qquad
F(0)=0.
\]
Then, for any
\[
f\in L^\infty_T(L^\infty)
\cap
\tdL^q_T(\dB^s_{p,1}(\omega_T)),
\]
we have
\[
\norm{F(f)}_{\tdL^q_T(\dB^s_{p,1}(\omega_T))}
\lesssim
\bigl(1+\norm f_{L^\infty_T(L^\infty)}\bigr)^{[s]+2}
\norm f_{\tdL^q_T(\dB^s_{p,1}(\omega_T))}.
\]
\end{lemma}
\begin{lemma}
\label{lem:time-weighted-transport-CMZ}
\textup{\cite[Proposition 4.2]{CMZ2010R}}
Let \(p\in[1,\infty]\), \(s\in\left(-d\min\left\{\frac1p,\frac1{p'}\right\},\frac dp\right]\) and \(f\) solve
\[
\partial_t f+v\cdot\nabla f=g,
\qquad
f|_{t=0}=f_0.
\]
Assume that
\[
\nabla v\in L^1_T(\dB^{\frac{d}{p}}_{p,1}),
\qquad
f_0\in \dB^s_{p,1},
\qquad
g\in L^1_T(\dB^s_{p,1}).
\]
Then, for every \(t\in[0,T]\),
\[
\norm f_{\tdL^\infty_t(\dB^s_{p,1}(\omega_T))}
\le
e^{CV(t)}
\left(
\norm{f_0}_{\dB^s_{p,1}(\omega_T)}
+
\int_0^t e^{-CV(\tau)}
\norm{g(\tau)}_{\dB^s_{p,1}(\omega_T)}\dd\tau
\right),
\]
where
\[
V(t):=\int_0^t
\norm{\nabla v(\tau)}_{\dB^{\frac{d}{p}}_{p,1}}\dd\tau.
\]
\end{lemma}
\begin{lemma}
\label{lem:static-products}\cite[Lemma~2.3 and Lemma~2.5]{GuoYangZhang2026}
Let
$
d\ge2,1<p<2d,\, 0<\delta_0<1$ and 
 $ \Omega\in AF(\delta_0).
$
We have the following estimates.

\begin{enumerate}
\item If $-d\min\left\{\frac1p,\frac1{p'}\right\}
<s\le \frac dp-\delta_0,$
then
\begin{equation}\label{eq:weighted-one-sided}
\|fg\|_{\dot B^s_{p,1}(\Omega)}
\lesssim
\|f\|_{\dot B^{\frac dp}_{p,1}}
\|g\|_{\dot B^s_{p,1}(\Omega)}.
\end{equation}

\item One has
\begin{equation}\label{eq:weighted-endpoint}
\begin{aligned}
\|fg\|_{\dot B^{\frac dp}_{p,1}(\Omega)}
\lesssim{}&
\|f\|_{\dot B^{\frac dp}_{p,1}(\Omega)}
\|g\|_{\dot B^{\frac dp}_{p,1}}+
\|f\|_{\dot B^{\frac dp}_{p,1}}
\|g\|_{\dot B^{\frac dp}_{p,1}(\Omega)}.
\end{aligned}
\end{equation}

\item If $-d\min\left\{\frac1p,\frac1{p'}\right\}
<s\le\frac dp,$
then
\begin{equation}\label{eq:weighted-mixed}
\begin{aligned}
\|fg\|_{\dot B^s_{p,1}(\Omega)}
\lesssim{}&
\|f\|_{\dot B^s_{p,1}}
\|g\|_{\dot B^{\frac dp}_{p,1}(\Omega)}
+
\|f\|_{\dot B^s_{p,1}(\Omega)}
\|g\|_{\dot B^{\frac dp}_{p,1}}.
\end{aligned}
\end{equation}

\item  If \(F\in C^\infty(\mathbb R)\), \(F(0)=0\), and $f\in L^\infty\cap\dot B^{\frac dp}_{p,1}(\Omega),$
then
\begin{equation}\label{eq:weighted-composition}
\|F(f)\|_{\dot B^{\frac dp}_{p,1}(\Omega)}
\le
C
\|f\|_{\dot B^{\frac dp}_{p,1}(\Omega)}
\end{equation}
with $C$ depending only on $\|f\|_{L^{\infty}}$, $F$, $s, p, r, d$ and $\delta_{0}$.
\end{enumerate}
\end{lemma}
\begin{remark}
\label{lem:CL-envelope-composition}
Let \(\Phi\in C^\infty(\R)\) satisfy \(\Phi(0)=0\). From \cref{lem:static-products}, if
\[
z\in L^\infty((0,T)\times\R^d)
\cap\widetilde L^\infty_T
\bigl(\dot B^{\frac dp}_{p,1}\bigr)
\cap\widetilde L^\infty_T
\bigl(\dot B^{\frac dp}_{p,1}(\Omega)\bigr),
\]
then
\begin{equation}\label{eq:CL-envelope-composition}
\|\Phi(z)\|_{\widetilde L^\infty_T
(\dot B^{\frac dp}_{p,1}(\Omega))}
\le
C
\|z\|_{\widetilde L^\infty_T
(\dot B^{\frac dp}_{p,1}(\Omega))},
\end{equation}
where \(C\) depends only on \(d,p,\delta_0\),
\(\|z\|_{L^\infty_{T,x}}\), and finitely many derivatives of
\(\Phi\) on a compact interval determined by
\(\|z\|_{L^\infty_{T,x}}\).
\end{remark}

\begin{lemma}
\label{lem:static-weighted-transport}
\textup{\cite[Lemma 2.4]{GuoYangZhang2026}.}
Let \(1\le p,r\le\infty\), let
\(\Omega\in AF(\delta_0)\), and assume that
\begin{equation}\label{eq:static-transport-index-range}
-d\min\left\{\frac1p,1-\frac1p\right\}
<\sigma<
1+\frac dp-\delta_0.
\end{equation}
Let \(\lambda\ge0\), and consider
\begin{equation}\label{eq:static-weighted-transport-equation}
\partial_t f+v\cdot\nabla f+\lambda f=g,
\qquad
f|_{t=0}=f_0.
\end{equation}
Assume that
\[
f_0\in\dB^\sigma_{p,r}(\Omega),
\qquad
g\in\tdL^1_T\bigl(\dB^\sigma_{p,r}(\Omega)\bigr),
\]
and that
\[
V(t):=
\int_0^t
\norm{\nabla v(\tau)}
_{\dB^{\frac{d}{p}}_{p,\infty}\cap L^\infty}
\dd\tau
<\infty .
\]
Then there exists a constant \(C>0\) such that every smooth solution
of \eqref{eq:static-weighted-transport-equation} satisfies, for all
\(t\in[0,T]\),
\begin{equation}\label{eq:static-weighted-transport}
\begin{aligned}
&\norm f_{\tdL^\infty_t
\bigl(\dB^\sigma_{p,r}(\Omega)\bigr)}
+\lambda
\norm f_{\tdL^1_t
\bigl(\dB^\sigma_{p,r}(\Omega)\bigr)}
\lesssim \exp\bigl(CV(t)\bigr)
\left(
\norm{f_0}_{\dB^\sigma_{p,r}(\Omega)}
+
\norm g_{\tdL^1_t
\bigl(\dB^\sigma_{p,r}(\Omega)\bigr)}
\right),
\end{aligned}
\end{equation}
where the constant \(C\) depends only on
\(d,p,r,\sigma,\delta_0\) and is independent of
\(t,\lambda,f_0,g\), and the terminal value of \(\Omega\).
\end{lemma}
\begin{remark}
By the standard Friedrichs approximation procedure, together with the
Fatou property of homogeneous Besov spaces, estimate
\eqref{eq:static-weighted-transport} remains valid for distributional
solutions of \eqref{eq:static-weighted-transport-equation} satisfying
the regularity assumptions stated above.
\end{remark}
\begin{lemma}
\label{lem:weighted-commutators}
Let $d\ge2,1<p<2d$ and  $\gamma_m:=-1+\frac{d}{p}+\min \left\{\frac{d}{p},\frac{d}{p'}\right\}$.
Assume that $
0 < \delta_0 < \min\{1, \gamma_m\},\, m \in \mathbb{Z}$ and $\Omega \in AF(\delta_0)$. Let $A(D)$ be a smooth homogeneous Fourier multiplier of order zero.
Define 
\[
 \quad q_m := \dot{S}_m q, \quad q^h := (\text{Id} - \dot{S}_m)q.
\]
If
$q \in \dot{B}_{p,1}^{\frac{d}{p}} \cap \dot{B}_{p,1}^{\frac{d}{p}}(\Omega)$ and
$
v \in \dot{B}_{p,1}^{\frac{d}{p}}(\Omega) \cap \dot{B}_{p,1}^{1+\frac{d}{p}} \cap \dot{B}_{p,1}^{1+\frac{d}{p}}(\Omega)
$, then
\begin{align}
&\sum_j\Omega_j2^{j({-1+\frac{d}{p}})}
\norm{\Div([\dDelta_j,q_m]\nabla v)}_{L^p}
\le C\norm{\nabla q_m}_{\dB^{\frac{d}{p}}_{p,1}}
\norm v_{\dB^{{\frac{d}{p}}}_{p,1}(\Omega)},
\label{eq:comm-low}\\
&\norm{q^h\nabla^2v}_{\dB^{{-1+\frac{d}{p}}}_{p,1}(\Omega)}
\le C\norm{q^h}_{\dB^{\frac{d}{p}}_{p,1}}
\norm v_{\dB^{{{1+\frac{d}{p}}}}_{p,1}(\Omega)},
\label{eq:high-principal}\\
&\norm{\nabla q^h\,\nabla v}_{\dB^{{-1+\frac{d}{p}}}_{p,1}(\Omega)}
\le C \Bigl(
\norm{\nabla q^h}_{\dB^{{-1+\frac{d}{p}}}_{p,1}}
\norm v_{\dB^{{{1+\frac{d}{p}}}}_{p,1}(\Omega)}
+
\norm{\nabla q^h}_{\dB^{{-1+\frac{d}{p}}}_{p,1}(\Omega)}
\norm v_{\dB^{{{1+\frac{d}{p}}}}_{p,1}}
\Bigr),
\label{eq:high-gradient}\\
&\norm{[A(D),q_m]\nabla^2v}_{\dB^{{-1+\frac{d}{p}}}_{p,1}(\Omega)}
\le C_A\norm{\nabla q_m}_{\dB^{\frac{d}{p}}_{p,1}}
\norm v_{\dB^{{{\frac{d}{p}}}}_{p,1}(\Omega)},
\label{eq:mult-low}\\
&\norm{[A(D),q^h]\nabla^2v}_{\dB^{{-1+\frac{d}{p}}}_{p,1}(\Omega)}
\le C_A \Bigl(
\norm{\nabla q^h}_{\dB^{{-1+\frac{d}{p}}}_{p,1}}
\norm v_{\dB^{{{1+\frac{d}{p}}}}_{p,1}(\Omega)}
+
\norm{\nabla q^h}_{\dB^{{-1+\frac{d}{p}}}_{p,1}(\Omega)}
\norm v_{\dB^{{{1+\frac{d}{p}}}}_{p,1}}
\Bigr),
\label{eq:mult-high}
\end{align}
where  $C_A$ depends only on finitely many seminorms of the symbol of \(A\) on the unit annulus and  $C$ is independent
of $m,q,v,$ and $\Omega$.
\end{lemma}
\begin{proof}
We use Bony's decomposition in the form
\[
fg = \dot T_fg + \dot T'_gf, \quad \dot T'_gf := \dot T_gf + \dot R(g, f).
\]
By \eqref{eq:weight-ratio-abstract},
\[
\frac{\Omega_j}{\Omega_k}
\le
2^{\delta_0|j-k|}
\qquad
\text{for all }j,k\in\Z.
\]
By assumption, we know that $\delta_0 < 1,\,\delta_0 < \gamma_m$ and 
\[
\gamma_m = 
\begin{cases} 
\frac{2d}{p} - 1, & p \geq 2, \\
d - 1, & 1 < p < 2.
\end{cases}
\]

 We first prove \eqref{eq:comm-low}. Bony's decomposition gives
\begin{equation}
[\dDelta_j,q_m]\nabla v
=I_j-II_j+III_j, \label{eq:comm-bony}
\end{equation}
where
\[
I_j:=[\dDelta_j,\dot T_{q_m}]\nabla v,
\qquad
II_j:=\dot T'_{\dDelta_j\nabla v}q_m,
\qquad
III_j:=\dDelta_j\dot T'_{\nabla v}q_m.
\]
There exists an integer \(N_0\ge1\), for the first term
\[
I_j
=
\sum_{|k-j|\le N_0}
[\dDelta_j,\dS_{k-1}q_m]
\nabla\dDelta_kv,
\]
by which, \cref{lem:convolution-commutator} and Bernstein's inequality, we can obtain
\begin{align}
\norm{\Div I_j}_{L^p}
&\lesssim
\sum_{|k-j|\le N_0}
\norm{\nabla\dS_{k-1}q_m}_{L^\infty}
2^k\norm{\dDelta_kv}_{L^p}
\nonumber
\\
&\lesssim
\sum_{|k-j|\le N_0}
\left(
\sum_{\ell\le k-2}
2^{\ell\frac dp}
\norm{\dDelta_\ell\nabla q_m}_{L^p}
\right)
2^k\norm{\dDelta_kv}_{L^p}.
\label{eq:comm-Ij}
\end{align}
Multiplying above term by $\Omega_j2^{j{(-1+\frac{d}{p})}}$, and summing in $j$,  we obtain
\begin{align}
\sum_j\Omega_j2^{j{(-1+\frac{d}{p})}}\|\Div I_j\|_{L^p}
&\lesssim
\|\nabla q_m\|_{\dB^{\frac{d}{p}}_{p,1}}
\|v\|_{\dB^{\frac{d}{p}}_{p,1}(\Omega)}.
\label{eq:comm-Ij-sum}
\end{align}
We next estimate \(II_j\) and \(III_j\) separately. Put $c=(c_{k})_{k\in \mathbb{Z}}$ and  $d=(d_{k})_{k\in \mathbb{Z}}$, where
\begin{equation}\label{eq:comm-sequences}
c_k:=2^{k\frac{d}{p}}\norm{\dDelta_k\nabla q_m}_{L^p},
\qquad
d_k:=\Omega_k2^{k(\frac{d}{p})}
\norm{{\dDelta}_kv}_{L^p},
\end{equation}
from which,  
\[
\| c \|_{\ell^1}=\| \nabla q_m \|_{\dot{B}_{p,1}^{\frac{d}{p}}}, \quad \| d \|_{\ell^1}=\| v \|_{\dot{B}_{p,1}^{\frac{d}{p}}(\Omega)}.
\]
\noindent For \(II_j\), the support properties of \(\dot T'\) give
\begin{equation}\label{eq:comm-II-kernel}
\Omega_j2^{j{(-1+\frac{d}{p})}}\norm{\Div II_j}_{L^p}
\lesssim  d_j\sum_{k\ge j-N_0}
2^{-(k-j)\frac{d}{p}}c_k.
\end{equation}
\noindent For \(III_j\), decompose
\[
III_j=III_j^{\,T}+III_j^{\,R}
:=\dDelta_j\dot T_{\nabla v}q_m
+\dDelta_j\dot R(\nabla v,q_m).
\]
The low-high interaction  $III_j^{\,T}$ satisfies
\begin{equation}\label{eq:comm-III-T-kernel}
\begin{aligned}
\Omega_j2^{j{(-1+\frac{d}{p})}}\norm{\Div III_j^{\,T}}_{L^p}
&\lesssim \sum_{|k-j|\le N_0}c_k
\sum_{\ell\le k-N_0}
2^{-(k-\ell)(1-\delta_0)}d_\ell,
\end{aligned}
\end{equation}
where the summability of the kernel follows from \(\delta_0<1\).

\noindent For the remainder term $III_j^{\,R}$. If \(p \geq 2\), H\"older's inequality in \(L^{\frac{p}{2}}\), followed
by Bernstein's inequality from \(L^{\frac{p}{2}}\) to \(L^p\), we have
\[
\begin{aligned}
\Omega_j2^{j(-1+\frac dp)}
\|\Div\dot\Delta_j(
\dot\Delta_k\nabla v\,\widetilde{\dot\Delta}_kq_m)\|_{L^p}
&\lesssim
\Omega_j2^{j(-1+\frac dp)}
2^j2^{j\frac dp}
\bigl(\Omega_k^{-1}2^{k(1-\frac dp)}d_k\bigr)
\bigl(2^{-k(1+\frac dp)}c_k\bigr)
\\
&=
2^{-(k-j)\frac{2d}{p}}
\frac{\Omega_j}{\Omega_k}c_kd_k.
\end{aligned}
\]
If \(1<p<2\), Bernstein from \(L^p\) to \(L^{p'}\) for the velocity
block and then from \(L^1\) to \(L^p\) for the output give
\[
\begin{aligned}
\Omega_j2^{j(-1+\frac dp)}
\|\Div\dot\Delta_j(
\dot\Delta_k\nabla v\,\widetilde{\dot\Delta}_kq_m)\|_{L^p}
&\lesssim
\Omega_j2^{j(-1+\frac dp)}
2^j2^{j\frac d{p'}}
2^{kd(\frac1p-\frac1{p'})}
\bigl(\Omega_k^{-1}2^{k(1-\frac dp)}d_k\bigr)
\bigl(2^{-k(1+\frac dp)}c_k\bigr)\\
&=
2^{-d(k-j)}
\frac{\Omega_j}{\Omega_k}c_kd_k.
\end{aligned}
\]
Thus both regimes give
\begin{equation}\label{eq:comm-III-R-kernel}
\Omega_j2^{j{(-1+\frac{d}{p})}}\norm{\Div III_j^{\,R}}_{L^p}
\lesssim \sum_{k\ge j-3}
2^{-(k-j)\gamma_c}
\frac{\Omega_j}{\Omega_k}c_kd_k,
\end{equation}
where
\[
\gamma_c := \frac{d}{p} + \min \left\{ \frac{d}{p}, \frac{d}{p'} \right\} = \begin{cases} 
\frac{2d}{p}, & p \geq 2, \\ 
d, & 1 < p < 2. 
\end{cases}
\]
Since \(\gamma_c>\gamma_m>\delta_0\),
\cref{rem:envelope-discrete-convolution} shows that the kernel in
\eqref{eq:comm-III-R-kernel} is summable. Summing 
\eqref{eq:comm-II-kernel},
\eqref{eq:comm-III-T-kernel}, and
\eqref{eq:comm-III-R-kernel}, and combining with
\eqref{eq:comm-Ij-sum}, it proves \eqref{eq:comm-low}.

We next prove\eqref{eq:high-principal}.
By the weighted product estimate \emph{\eqref{eq:weighted-one-sided}}, we have
\begin{align*}
\|q^h\nabla^2v\|_{\dB^{-1+\frac{d}{p}}_{p,1}(\Omega)}
&\lesssim \|q^h\|_{\dB^{\frac{d}{p}}_{p,1}}
\|\nabla^2v\|_{\dB^{-1+\frac{d}{p}}_{p,1}(\Omega)}
\lesssim \|q^h\|_{\dB^{\frac{d}{p}}_{p,1}}
\|v\|_{\dB^{1+\frac{d}{p}}_{p,1}(\Omega)}.
\label{eq:high-principal-proof}
\end{align*}
This proves \eqref{eq:high-principal}.

 For \eqref{eq:high-gradient}, applying \eqref{eq:weighted-mixed}, we obtain
\begin{align*}
\|\nabla q^h\,\nabla v\|_{\dB^{-1+\frac{d}{p}}_{p,1}(\Omega)}
&\lesssim 
\|\nabla q^h\|_{\dB^{-1+\frac{d}{p}}_{p,1}}
\|v\|_{\dB^{1+\frac{d}{p}}_{p,1}(\Omega)}
+
\|\nabla q^h\|_{\dB^{-1+\frac{d}{p}}_{p,1}(\Omega)}
\|v\|_{\dB^{1+\frac{d}{p}}_{p,1}}.
\end{align*}
Thus \eqref{eq:high-gradient} follows.

  We next shall prove \eqref{eq:mult-low}. Let
\[
A_j(D):=A(D)\dDelta_j.
\]
The convolution kernel of \(A_j(D)\) has the form $2^{jd}K_A(2^j\cdot)$ and 
$\int_{\R^d}|x|\,|K_A(x)|\dd x<\infty$. Hence, by \cref{lem:convolution-commutator}, we have
\begin{equation}\label{eq:multiplier-kernel}
\begin{aligned}
\norm{[A_j(D),a]h}_{L^p}
+
\norm{[\dDelta_j,a]A(D)h}_{L^p}
\le
C_A2^{-j}
\norm{\nabla a}_{L^\infty}
\norm h_{L^p},
\end{aligned}
\end{equation}
where \(C_A\) depends only on finitely many seminorms of the symbol
of \(A(D)\) on the unit annulus.
Bony's decomposition gives
\begin{equation}
[A(D),q_m]\nabla^2v
=
[A(D),\dot T_{q_m}]\nabla^2v
+
A(D)\dot T'_{\nabla^2v}q_m
-
\dot T'_{A(D)\nabla^2v}q_m.
\label{eq:multiplier-bony}
\end{equation}
For the first term of the right side \eqref{eq:multiplier-bony}, we have
\begin{equation}\label{eq:correct-multiplier-localization}
  \dot \Delta_j[A(D),\dot T_{q_m}]\nabla^2v
=
\sum_{|k-j|\le N_0}\left(
[A_j(D),\dS_{k-1}q_m]\nabla^2\dDelta_kv- [\dot{\Delta}_j, \dot{S}_{k-1} q_m] A(D) \nabla^2 \dot{\Delta}_k v .\right)  
\end{equation}
Using \eqref{eq:multiplier-kernel}, Bernstein's inequality, and the
same summation as in \eqref{eq:comm-Ij-sum}, we get
\begin{equation}
\|[A(D),\dot T_{q_m}]\nabla^2v\|_{\dB^{-1+\frac{d}{p}}_{p,1}(\Omega)}
\le
C_A
\|\nabla q_m\|_{\dB^{\frac{d}{p}}_{p,1}}
\|v\|_{\dB^{\frac{d}{p}}_{p,1}(\Omega)}.
\label{eq:multiplier-low-paraproduct}
\end{equation}
For the two remaining terms in the right side of \eqref{eq:multiplier-bony}, the
low-high contribution satisfies
\begin{equation}\label{eq:multiplier-T-kernel}
\begin{aligned}
&\sum_j\Omega_j2^{j{(-1+\frac{d}{p})}}
\norm{\dDelta_j\bigl(
A(D)\dot T_{\nabla^2v}q_m
-\dot T_{A(D)\nabla^2v}q_m
\bigr)}_{L^p}\\
&\qquad\le
C_A\sum_k c_k
\sum_{\ell\le k-2}
2^{-(k-\ell)(2-\delta_0)}d_\ell,
\end{aligned}
\end{equation}
and for the high-high remainder, the same two $p$-regimes as above yield
\begin{equation}\label{eq:multiplier-R-kernel}
\begin{aligned}
&\Omega_j2^{j{(-1+\frac{d}{p})}}
\norm{\dDelta_j\bigl(
A(D)\dot R(\nabla^2v,q_m)
-\dot R(A(D)\nabla^2v,q_m)
\bigr)}_{L^p}\\
&\qquad\le
C_A\sum_{k\ge j-3}
2^{-(k-j)\gamma_m}
\frac{\Omega_j}{\Omega_k}c_kd_k.
\end{aligned}
\end{equation}
The kernel in \eqref{eq:multiplier-T-kernel} is summable because $\delta_0<1$, while the one in
\eqref{eq:multiplier-R-kernel} is summable because $\delta_0<\gamma_m$. Therefore,
\begin{align}
&\|A(D)\dot T'_{\nabla^2v}q_m
-
\dot T'_{A(D)\nabla^2v}q_m\|_{\dB^{-1+\frac{d}{p}}_{p,1}(\Omega)}
\nonumber \le
C_A
\|\nabla q_m\|_{\dB^{\frac{d}{p}}_{p,1}}
\|v\|_{\dB^{\frac{d}{p}}_{p,1}(\Omega)},
\label{eq:multiplier-remainder}
\end{align}
which together with \eqref{eq:multiplier-low-paraproduct},we  obtain 
\eqref{eq:mult-low}.

 Finally, since $A(D)$ is bounded on $\dB^{-1+\frac{d}{p}}_{p,1}(\Omega)$,
\begin{equation}\label{eq:multiplier-high-first}
\|[A(D),q^h]\nabla^2v\|_{\dB^{-1+\frac{d}{p}}_{p,1}(\Omega)}
\le
C_A\left(\|q^h\nabla^2v\|_{\dB^{-1+\frac{d}{p}}_{p,1}(\Omega)}
+
\|q^hA(D)\nabla^2v\|_{\dB^{-1+\frac{d}{p}}_{p,1}(\Omega)}\right).
\end{equation}
The first term  in the right side of \eqref{eq:multiplier-high-first} is controlled by \eqref{eq:high-principal}. For the second term, applying \eqref{eq:weighted-mixed} , together with the order-zero
boundedness of $A(D)$, yields
\begin{align*}
  \|q^h A(D) \nabla^2 v\|_{\dot{B}_{p,1}^{-1+\frac{d}{p}}(\Omega)} &\le C_A \left( \|q^h\|_{\dot{B}_{p,1}^{\frac{d}{p}}} \|v\|_{\dot{B}_{p,1}^{1+\frac{d}{p}}(\Omega)} + \|q^h\|_{\dot{B}_{p,1}^{\frac{d}{p}}(\Omega)} \|v\|_{\dot{B}_{p,1}^{1+\frac{d}{p}}} \right)\\
  &\le C_A \left( \|\nabla q^h\|_{\dot{B}_{p,1}^{-1+\frac{d}{p}}} \|v\|_{\dot{B}_{p,1}^{1+\frac{d}{p}}(\Omega)} + \|\nabla q^h\|_{\dot{B}_{p,1}^{-1+\frac{d}{p}}(\Omega)} \|v\|_{\dot{B}_{p,1}^{1+\frac{d}{p}}} \right),  
\end{align*}
from which we can obtain \eqref{eq:mult-high}. The proof of \cref{lem:weighted-commutators} is complete.
\end{proof}
\section*{Acknowledgement}
This work was started while M. Yang was visiting Professor Zihua Guo at Monash University. M. Yang thanks Professor Zihua Guo  for helpful discussions and his kind hospitality. The article is supported by the National Natural Science Foundation of China (Grant Nos.~12661017, 1266104, 12561034) and the Jiangxi Province Natural Science Foundation (Grant Nos.~2052BAC250004,  2052BAC200146).

\end{document}